\documentclass[12pt]{article}

\usepackage[utf8]{inputenc}
\usepackage{textcomp}
\usepackage{graphicx}
\usepackage{multirow}
\usepackage{amssymb}
\usepackage{amsmath}
\usepackage{amsthm}
\usepackage{amssymb}
\usepackage{amsfonts}
\usepackage{a4wide}
\usepackage{tikz}
\usepackage{circuitikz}
\usepackage{dsfont}
\usepackage{graphicx} 
\usepackage{color}
\usepackage{yhmath}
\usepackage{cancel}

\usetikzlibrary{shapes,matrix,arrows,arrows.meta}
\tikzstyle{line}=[draw,-latex']
\tikzstyle{line2}=[-latex']
\begin{document}
\newcommand\mathplus{+}

\newtheorem{defi}{Definition}
\newtheorem{pr}{Proposition}
\newtheorem{Th}{Theorem}
\newtheorem{obs}{Observation}
\newtheorem{rrm}[obs]{Remark}
\newtheorem{cor}[obs]{Corollary}
\newtheorem{lema}[pr]{Lemma}
\newtheorem{ex}{Example}


 \title{Bicovariant Codifferential Calculi}
 \author{Andrzej Borowiec \footnote{andrzej.borowiec@uwr.edu.pl}, Patryk Mieszkalski \footnote{patryk.mieszkalski@uwr.edu.pl}\bigskip\\
 	University of Wroc{\l}aw, Institute of Theoretical Physics, \\
 	pl. M. Borna 9, 50-204 Wroc{\l}aw, Poland}
 \date{}

\maketitle

\begin{abstract}
	 	We develop a technique for studying first-order codifferential calculi (FOCCs) initiated by Doi and Quillen in the context of cyclic cohomology.  Their classification, for a given coalgebra,  reduces to the classification of subbicomodules in the universal bicomodule. For completing this task, the role of one-dimensional generating spaces (a.k.a. singletons) is found to be useful. We are particularly interested in classifying bicovariant codifferential calculi,  which we define over Hopf algebras. This, in turn, can be reduced to classifying Yetter-Drinfeld  (Y-D) submodules. In fact, there are two, mutually dual, Y-D structures on arbitrary Hopf algebra: one used by Woronowicz for constructing bicovariant differential calculi, and the another used here for FOCCs and shown to be related with Woronowicz construction of quantum tangent space. This  argues that such codifferential calculi are better suited to Drinfeld-Jimbo type quantized enveloping algebras, as they are dual to Woronowicz' bicovariant calculi over matrix quantum groups.  Relations with quantum Lie algebras and quantum vector fields are also shown. Some classification results are presented in numerous examples.
\end{abstract}
\tableofcontents
\section{Introduction}

  Noncommutative Differential Geometry (NCDG) has attracted considerable attention in recent decades, applying operator algebra to quantum geometry  \cite{Connes}. As a first step it appeared in the context of cyclic cohomology, following Connes' paper, in the form of universal differential calculus over some (usually noncommutative) algebra $\mathcal{A}$. 
  In fact, algebraic differential calculus becomes  one of the cornerstones of  NCDG, present there from the very beginning in various contexts. It consists of a First Order Noncommutative Differential Calculus (FODC) and its prolongation to a graded differential algebra. Thus, any  FODC  over  $\mathcal{A}$ can then be obtained as a result of the quotient construction. Its prolongation to the higher order (quantum de Rham complex) is functorial. Woronowicz Hopf algebra bicovariant differential calculi provided a new impetus for this development \cite{Woronowicz}. In particular, they become very useful as a tool for investigations of deformed symmetries in terms of Quantum Groups (QG). Following his paper, many examples of  bicovariant/covariant calculi on QG/quantum planes have been elaborated around this time, see, e.g. \cite{Klimyk} Part IV, and references therein, as well as \cite{BeggsMajid}.  More precisely, there are two types of QGs obtained in the process of quantum deformations from their classical analogs: Drinfeld's and Drinfeld-Jimbo quantized enveloping (Lie) algebras \cite{Drinfeld,Jimbo} and their dual counterparts, the so-called matrix QG introduced by Woronowicz \cite{Woronowicz2} on the level of $C^*$ algebras\footnote{Purely algebraic version was introduced later \cite{FRT}.}. 

We recall that the notions of coalgebras and comodules appeared as a formal mutual dual to the notion of algebras and modules, i.e., obtained by reversing arrows in the corresponding categorical formal definitions of those objects.
Similarly, the formal dual of differential calculus over algebras generates codifferential calculus over coalgebras, introduced by Doi and later by Quillen \cite{Doi81,Quillen1988}.
 In fact, except for the finite-dimensional cases, a real dualization works in the opposite direction: having a coalgebra, a bicomodule, and a coderivation one can construct an algebra, a bimodule, and a differential related to the initial objects by dualization. Unfortunately, there is no vast literature on coderivations, with a few exceptions, e.g., \cite{Anel,AB99,Khalkhali1996,Positelski2023}.  Main aim of the present paper is to study the bicovariant version of codifferential calculus over Hopf algebras, which are self-dual objects.
 Although both bicovariant FODCs and FOCCs can be realized on an arbitrary Hopf algebra, in the case of QGs we see essential differences at least from the perspective of classical limit and   possible physical applications. As demonstrated in subsection  4.2  and appendix~\ref{apx_bicov_FODC} --  there are two mutually dual Y-D structures  on arbitrary Hopf algebra.  
 One of them, we call Woronowicz' type has been used in the seminal paper \cite{Woronowicz} for constructing bicovariant differential calculi on matrix QGs.  The second one, which is used here to construct bicovariant FOCC, was also (not explicitly)
 used in \cite{Woronowicz} to construct a quantum tangent space and the corresponding quantum Lie algebra. This strongly suggests that the  bicovariant differential calculi are better suited for matrix QGs since they are quantum counterparts of the ordinary differential calculi on matrix group manifolds.

 NCG and noncommutative differential calculi have long-lasting and well-established   applications in physical models related to quantum gravity phenomenology and Planck scale physics, see e.g.  \cite{QGphenomenology} and references therein. In particular, the so-called $\kappa$-deformation \cite{kappa,Lukierski2017} provides ones of the best studied models in this context.
 It is also worth noting that coderivations have recently appeared in the context of $L_\infty$ algebras and generalized geometry  \cite{coderCabus,coderJonke,coderOkava}.

\subsection{Notational conventions }
 
If not specified otherwise, all algebras are unital and associative, and coalgebras are counital and coassociative over a field $\mathbb{K}$ of characteristic $0$, which in most cases would be the complex numbers (we will not specify it if not needed).  
\footnote{Advised references to this part and Appendix A, as well as the formalism used throughout,  are, e.g. \cite{BrzezinskiWisbauer,Montgomery_book,Radford,Sweedler69}.} Undorned tensor product $\otimes$ is over $\mathbb{K}$. In most cases, we will use $\mathcal{A}$ as a symbol for an algebra with a product $\mu:~\mathcal{A}\otimes\mathcal{A}\to\mathcal{A}$ and a unit $\eta:~\mathbb{K}\to\mathcal{A}$. Coalgebras will be denoted by $\mathcal{C}$ with a coproduct $\Delta:~\mathcal{C}\to\mathcal{C}\otimes\mathcal{C}$ and a counit $\varepsilon:~\mathcal{C}\to\mathbb{K}$. 
Hopf algebras will be denoted by $H$, and for an antipode we will use the symbol $S$.  

For a given algebra $\mathcal{A}$ we denote by $\mathfrak{M}_\mathcal{A}$, ${}_\mathcal{A} \mathfrak{M}$,
${}_\mathcal{A} \mathfrak{M}_\mathcal{A}$ the corresponding category of right/left modules or bimodules. For right/left comodules or bicomodules  over a coalgebra  $\mathcal{C}$ we shall use similar notation: $\mathfrak{M}^\mathcal{C}$  
${}^\mathcal{C} \mathfrak{M}$, ${}^\mathcal{C} \mathfrak{M}^\mathcal{C}$. 
In the case of Hopf  algebra $H$ some other combinations will be later also  possible, e.g. ${}_H^H \mathfrak{M}^H_H$.

In particular, we can take $C^*=Hom(\mathcal{C},\mathbb{K})$, which is an algebra of linear functionals on $\mathcal{C}$, and replace left/right/bicomodules over $\mathcal{C}$ by right/left/bimodules over $\mathcal{C}^*$,
see Appendix A.

Left free module for algebra $\mathcal{A}$ is isomorphic to: $\mathcal{A}\otimes V$, where $V$ is a vector space. The left action can be defined:
\begin{eqnarray}
    \mu_L(a_1,a_2^i\otimes v_i)&:=&(a_1a_2^i)\otimes v_i.
\end{eqnarray}

Analogously, Left free comodule for coalgebra $\mathcal{C}$ is isomorphic to: $\mathcal{C}\otimes V$. The left coaction then is defined:
\begin{eqnarray}
    \Delta_L(b^i\otimes v_i)&:=&(b^i)_{(1)}\otimes\left((b^i)_{(2)}\otimes v_i\right).
\end{eqnarray}
We are using Sweedler notation - for every element $a$ in coalgerba $\mathcal{C}=(\mathcal{C}, \Delta,\varepsilon)$:
\begin{eqnarray}
    \Delta^\mathcal{C}(a)=\Delta(a)&=&a_{(1)}\otimes a_{(2)}\,,\quad a=\varepsilon(a_{(1)}) a_{(2)}=a_{(1)}\varepsilon(a_{(2)})\\
    \Delta^2(a)&:=&(\Delta\otimes\mathrm{id})\circ\Delta=a_{(1)}\otimes a_{(2)}\otimes a_{(3)}.
\end{eqnarray}
 For element $m$ in a left/right $\mathcal{C}$ comodule $M=(M, \Delta_{L/R})$ we shall modify this notation:
\begin{eqnarray}
    \Delta_L(m)&=&m_{(-1)}\otimes m_{<0>}\,,\quad m=\varepsilon(m_{(-1)})m_{<0>}\\
    \Delta^2_L(m):=(\mathrm{id}\otimes\Delta_L)\circ\Delta_L(m)&=&(\Delta\otimes\mathrm{id})\circ\Delta_L(m)=m_{(-2)}\otimes m_{(-1)}\otimes m_{<0>}\,,\\
    \Delta_R(m)&=&m_{<0>}\otimes m_{(1)}\,,\quad m=m_{<0>}\varepsilon (m_{(1)})\\
    \Delta_R^2(m):=(\Delta_R\otimes\mathrm{id})\circ\Delta_R(m)&=&(\mathrm{id}\otimes\Delta)\circ\Delta_R(m)=m_{<0>}\otimes m_{(1)}\otimes m_{(2)}\,.\end{eqnarray}
Elements from the comodule will always have sharp brackets, and elements from algebra will have the standard brackets. 

In this paper, we mostly deal with bicomodules. For a bicomodule $M\in {}^\mathcal{C} \mathfrak{M}^\mathcal{D}$ it is convenient to use the bicomodule structure map 
\begin{eqnarray}\label{bcoms1} 
   {}_L \Delta_{R}(m)\doteq(\Delta_L\otimes\mathrm{id})\circ\Delta_R(m)={}_R \Delta_{ L}(m)=(\mathrm{id}\otimes\Delta_R)\circ\Delta_L(m)=m_{(-1)}\otimes m_{<0>}\otimes m_{(1)}\,,
\end{eqnarray}
which is injective ${}_L \Delta_{R}: M\rightarrow \mathcal{C}\otimes M\otimes \mathcal{D}$. It  satisfies the obvious relation
\begin{eqnarray}\label{bcoms2}   
  (\mathrm{id}\otimes{}_L \Delta_{R}\otimes\mathrm{id})\circ {}_L \Delta_{R}&=& (\Delta_\mathcal{C}\otimes\mathrm{id}\otimes\Delta_\mathcal{D})\circ {}_L \Delta_{R} \,.
\end{eqnarray}

The following lemma  is well-known (see e.g., \cite{Doi81,Quillen1988}) and will be used throughout:
 \begin{lema} \label{lemma1}Assume $M$ is    A) a left $\mathcal{C}$-comodule, B) a right $\mathcal{D}$-comodule, C) a $\mathcal{C}-\mathcal{D}$ bicomodule, and $V$ is a linear space.
 \begin{itemize}
       \item 
   There is a one-to-one correspondence, i.e., an isomorphism of linear spaces, between  the space of all linear maps $Hom_\mathbb{K}(M,V)$ and the following spaces: \\ A) the space of all left-comodule maps $Com^{\mathcal{C}, -}(M, \mathcal{C}\otimes V)$,\\ B) the space of all right-comodule maps $Com^{-, \mathcal{D}}(M,   V\otimes\mathcal{D})$,\\
C) the space of all bicomodule maps $Com^{\mathcal{C}, \mathcal{D}}(M, \mathcal{C}\otimes V\otimes\mathcal{D})$,
\item Moreover\\
    A') If $V$ is also  left $\mathcal{C}$-comodule then
   $f\in Com^{\mathcal{C}, -}(M,V)\subset Hom(M,V)$ iff 
   \begin{eqnarray}
   	{}\Delta^V_{L}\circ f&=&(\mathrm{id}_\mathcal{C}\otimes f)\circ\, \Delta^M_{L}\doteq\tilde f_L\,.\nonumber
   \end{eqnarray}
    B') If $V$ is also right  $\mathcal{D}$-comodule then
   $f\in Com^{-, \mathcal{D}}(M,V)\subset Hom(M,V)$ iff 
   \begin{eqnarray}
   	\Delta^V_{R}\circ f&=&(f\otimes\mathrm{id}_\mathcal{D})\circ\,\Delta^M_{R}\doteq\tilde f_R\,.\nonumber
   \end{eqnarray}
 C') If $V$ is also  $\mathcal{C}-\mathcal{D}$ bicomodule then
   $f\in Com^{\mathcal{C}, \mathcal{D}}(M,V)\subset Hom(M,V)$ iff 
   \begin{eqnarray}
   	  {}_L \Delta^V_{R}\circ f&=&(\mathrm{id}_\mathcal{C}\otimes f\otimes\mathrm{id}_\mathcal{D})\circ\, {}_L \Delta^M_{R}\doteq\tilde f\,.\nonumber
   \end{eqnarray}
 \end{itemize}
 \begin{proof} It is clear that bicomodules have both left and right comodule structure that commute. We show only the case C).
For any $f\in Hom_\mathbb{K}(M,V)$ we define $\tilde f\in Com^{\mathcal{C}\,\mathcal{D}}(M, \mathcal{C}\otimes V\otimes\mathcal{D})$, where  $\mathcal{C}\otimes V\otimes\mathcal{D}$ is the (bi)free bicomodule, by $\tilde f=(\mathrm{id}_\mathcal{C}\otimes f\otimes\mathrm{id}_\mathcal{D})\circ \, {}_L \Delta^M_{R}$. The inverse map is $f=(\varepsilon_\mathcal{C}\otimes \mathrm{id}_M\otimes\varepsilon_\mathcal{D})\circ\tilde f$. 
The second part says that $f$ is a bicomodule morphism.
      Particularly, taking $V=M$ and $f=\mathrm{id}_M$, the last condition is automatically satisfied and $\widetilde{\mathrm{id}_M}={}_L \Delta^M_{R}$.
 \end{proof}
\end{lema}

\section{Mathematical formalism of codifferential calculus}
\subsection{Basic definitions and properties}

We recall \cite{Doi81,Quillen1988} that a coderivation is a linear map $\delta:\Upsilon \rightarrow\mathcal{C} $, where $\Upsilon\in
{}^\mathcal{C} \mathfrak{M}^\mathcal{C}$, satisfying co-Leibniz rule
\begin{equation}\label{cLeibniz}
	\Delta\circ\delta=(\mathrm{id}\otimes\delta)\circ\Delta_L+(\delta\otimes\mathrm{id})\circ\Delta_R\,,
\end{equation}
or using Sweedler-type notation introduced earlier
\begin{equation}\label{cLeibniz2}
( \delta(m))_{(1)}\otimes(\delta(m))_{(2)} = m_{(-1)}\otimes \delta(m_{<0>}) + \delta(m_{<0>})\otimes m_{(1)}\,.
\end{equation} 
\begin{rrm}
Restriction of $\delta$ to any subbicomodule $\Upsilon_1\subset\Upsilon$ becomes a coderivation on $\Upsilon_1$.
\end{rrm}
\begin{rrm}\label{coideal}
	 It is easy to deduce, that 
	the image of any coderivation $\mathrm{Im}\delta$ is a coideal in $\mathcal{C}$:  the co-Leibniz formula  \eqref{cLeibniz} gives 
	\begin{eqnarray}\label{cLeibniz3}
		\Delta(\mathrm{Im\delta})&\subset&\mathcal{C}\otimes\mathrm{Im}\delta+\mathrm{Im}\delta\otimes\mathcal{C}.
	\end{eqnarray}
	We know that for every coderivation $\mathrm{Im}\delta$ is a coideal; however, the opposite is not always true. An example of such  a case is a matrix coalgebra Example~\ref{ex_M2x2_coalgebra}. 
    Moreover, applying $\epsilon\otimes\epsilon$ to both sides of \eqref{cLeibniz2}  we arrive to the conclusion that $$\epsilon\circ\delta=0\,.$$
    Consequently, composing $\pi_\delta\circ\delta=0$, where $\pi_\delta:\mathcal{C}\rightarrow\mathcal{C}/\mathrm{Im\delta}$ is the canonical projection,  one gets the trivial coderivation.
	Remember, that $\mathrm{Ker}\varepsilon$ is the biggest coideal in $\mathcal{C}$ and the factor-coalgebra $\mathcal{C}/\mathrm{Ker}\varepsilon\equiv \mathbb{K}$.   
 
\end{rrm}
 
A coderivation $\delta:\Upsilon \rightarrow\mathcal{C} $ is called internal if there exists an element $\Gamma\in \Upsilon^*$ such that
\begin{equation}
	\delta (m)\equiv \delta_\Gamma (m)= m_{(-1)} \Gamma( m_{<0>} )-
	\Gamma( m_{<0>}) m_{(1)}\,.
\end{equation}
For coseparble coalgebras every coderivation is internal (inner) \cite{Doi81}, see Remark \ref{coseparable} below. 
\begin{ex}	 
For any coalgebra one can consider internal coderivations from the coalgebra itself: every element of the algebra $\alpha\in\mathcal{C}^*$  provides the coderivation:
		\begin{eqnarray}
			\delta_\alpha(a)&:=&a_{(1)}\alpha(a_{(2)})-a_{(2)}\alpha(a_{(1)}).
		\end{eqnarray}
		Let's note, that $\delta_\varepsilon=0$.  
		We can also check that $\left[\delta_\alpha,\delta_\beta\right]=
		\delta_{[\alpha,\beta]_\star}$
which means that the algebra $\mathcal{C}^*$ can also be seen as a Lie algebra of internal coderivations of $\mathcal{C}$.
	\end{ex}
\begin{rrm}
	It is easy to observe now that the transpose map $\delta^*: \mathcal{C}^*\rightarrow M^* $, where $M^*\in
	{}_{\mathcal{C}^*} \mathfrak{M}_{\mathcal{C}^*}$ is a derivation satisfying the standard Leibniz rule
	\begin{equation}
		\delta^*(\alpha\star\beta)=\delta^*(\alpha)\star\beta+\alpha\star\delta^*(\beta)\,,
	\end{equation}
	and $\delta^*(\alpha)(m)=\alpha(\delta(m))$. \footnote{In general, except of finite-dimensional cases, not every derivation $d: \mathcal{C}^*\rightarrow N $, where $N\in
		{}_{\mathcal{C}^*} \mathfrak{M}_{\mathcal{C}^*}$ can be obtained in this way.} In particular, for the last example 
		$\delta^*_\alpha(\beta)= [\alpha,\beta]_\star$ becomes an internal derivation of the algebra $\mathcal{C}^*$.
\end{rrm}
Similarly to the case of algebras, for each coalgebra there exists a universal coderivation \cite{Doi81}. Firstly, for a coalgebra $(\mathcal{C}, \Delta, \varepsilon)$ one defines a universal bicomodule $\Upsilon^U_\mathcal{C}$ as a quotient 
\begin{eqnarray}\label{univCoder}
	\Upsilon^U_\mathcal{C}= \mathcal{C}\otimes \mathcal{C}\,/\,\mathrm{Im}\Delta\equiv \mathrm{Coker}\Delta
\end{eqnarray}
together with 
 \begin{eqnarray}\label{DeltaU}
 	\Delta_L^U[a\otimes b]:=a_{(1)}\otimes[a_{(2)}\otimes b]\,,\quad
 	\Delta_R^U[a\otimes b]:=[a\otimes b_{(1)}]\otimes b_{(2)}\,,
 \end{eqnarray}
 where $[a\otimes b]$ denotes the corresponding equivalence class of a simple tensor $a\otimes b\in \mathcal{C}\otimes \mathcal{C}$ ($[\mathrm{Im}\Delta]=0$). Then one gets the exact sequence of bicomodules  
 \begin{equation}\label{couniv}
 	0\longrightarrow   \mathcal{C} \stackrel{\Delta}{\longrightarrow} \mathcal{C}\otimes \mathcal{C}\stackrel{\pi}{\longrightarrow}
 	\Upsilon^U_\mathcal{C}\longrightarrow 0\,,
 \end{equation}
 where $\pi(a\otimes b)=[a\otimes b]$ is a canonical projection.
 The  universal coderivation is defined by \cite{Doi81}
 \begin{equation}
 \delta^U[a\otimes b]:=a\varepsilon(b)-b\varepsilon(a)
 \end{equation}
 This is not an internal derivation, except the case of coseparable coalgebras (Remark \ref{coseparable}).
\begin{lema} 
      For the coderivation $(\Upsilon^U_\mathcal{C},\delta^U)$ one gets the short exact sequence of vector spaces: 
   \begin{eqnarray}\label{kerdelta}
       \begin{tikzpicture}
                \matrix (m) [matrix of math nodes,row sep=3em,column sep=4em,minimum width=2em]
                {
          0 & \mathrm{Ker}\delta^U &  \Upsilon^U_\mathcal{C} & \mathrm{Ker}\varepsilon & 0 \\};
          \draw[->] (m-1-1) edge (m-1-2); \draw[->] (m-1-2) edge (m-1-3);
                \draw[->] (m-1-3) edge node [above] {$\delta^U$} (m-1-4);
                \draw[->] (m-1-4) edge (m-1-5);
            \end{tikzpicture}
   \end{eqnarray}
   \begin{proof}
       We already know $\varepsilon\circ\delta=0$. We need to show that $\delta^U: \Upsilon^U_\mathcal{C}\rightarrow \mathrm{Ker}\varepsilon$ is surjective. Take $a\in \mathrm{Ker}\varepsilon$ and $b\notin \mathrm{Ker}\varepsilon$ then
       \begin{eqnarray}    \delta^U([a\otimes b])&=& a\varepsilon(b)\,.
       \end{eqnarray}
   \end{proof}
\end{lema}
Thus in the finite-dimensional case ($\dim \mathcal{C}=n$) one gets $\dim \Upsilon^U_\mathcal{C}=n(n-1)$ and $\dim  \mathrm{Ker}\delta^U=(n-1)^2$.

Dualizing the exact sequence \eqref{couniv} one gets a commutative diagram with exact rows of being bimodule morphisms
 \begin{eqnarray}\label{codual}
 	\begin{tikzpicture}
 		\matrix (m) [matrix of math nodes,row sep=3em,column sep=4em,minimum width=2em]
 		{
 			0 &\mathcal{C}^{*} & (\mathcal{C}\otimes\mathcal{C})^* & (\Upsilon^U_\mathcal{C})^*& 0\\
 			& & \mathcal{C}^*\otimes\mathcal{C}^* & \mathrm{Ker}\mu& \\};
 		\draw[->] (m-1-2) edge node [above] {$\ $} (m-1-1);
 		\draw[<-] (m-1-2) edge node [above] {$\Delta^*$} (m-1-3);
 		\draw[<-] (m-1-3) edge node [above] {$\pi^*$} (m-1-4);
 		\draw[<-] (m-1-4) edge node [above] {$\ $} (m-1-5);
 		\draw[<-] (m-2-3) edge node [above] {$\ $} (m-2-4);
 		\draw[->] (m-1-5) edge node [left] {$\ $} (m-2-4);
 		\draw[<-] (m-1-2) edge node [right] {$\ \mu $} (m-2-3);
 		\draw[->] (m-2-4) edge node [left] {$\cup $} (m-1-4);
 		\draw[<-] (m-1-3) edge node [right] {$\cup $} (m-2-3);
 	\end{tikzpicture}
 \end{eqnarray}
 Notice that $\Phi\in (\mathcal{C}\otimes\mathcal{C})^*$ are bilinear forms on the vector space $\mathcal{C}$, and $(\Upsilon^U_\mathcal{C})^*$ consist of the forms vanishing on all coproducts, i.e. $\Phi(\Delta(a))=0$ for arbitrary $a\in \mathcal{C}$.
  The vertical arrows are injective maps that become identities in the finite-dimensional case. 
 \begin{rrm}
 In particular, $\mathrm{Ker}\mu\equiv\Omega^U_{\mathcal{C}^*}\subset 
 (\Upsilon^U_\mathcal{C})^*$ is a subbimodule consisting of universal differential one-forms of the algebra $\mathcal{C}^*$.	 
 \end{rrm}
  \begin{rrm}\label{coseparable}
  For coseparable coalgebras, the exact sequence \eqref{couniv} of bicomodules splits, see \cite{Doi81}, and all derivations are internal.   We recall that coseparability of a coalgebra $\mathcal{C}$ is equivalent to the existence of a cointegral functional $\omega: \mathcal{C}\otimes\mathcal{C}\rightarrow \mathbb{K}$ with the following properties: 
  \begin{equation}
  	\omega\circ\Delta=\varepsilon\,,\quad (\mathrm{id}\otimes\omega)\circ(\Delta\otimes\mathrm{id})=(\omega\otimes\mathrm{id})\circ(\mathrm{id}\otimes\Delta)\,. 
  \end{equation} Thus defining an injective bicomodule map $\hat\omega: \mathcal{C}\otimes\mathcal{C}\rightarrow\mathcal{C}$ by $\hat{\omega}(a\otimes b)= a_{(1)} \omega(a_{(2)}\otimes b)=\omega(a\otimes b_{(1)})b_{(2)}$ one gets $\omega=\varepsilon\circ\hat{\omega}$ and the splitting condition $\hat\omega\circ\Delta=\mathrm{id}_\mathcal{C}$ is fulfilled. Similarly, one can construct another injective bicomodule map $\omega^U:\Upsilon^U_\mathcal{C}\rightarrow\mathcal{C}\otimes\mathcal{C}$ by
  \begin{eqnarray}
  	\omega^U([a\otimes b])&:=&a\otimes b- a_{(1)}\otimes b_{(2)}\omega(a_{(2)}\otimes b_{(1)})=\nonumber\\
  	 a\otimes b- a_{(1)}\otimes a_{(2)}\omega(a_{(3)}\otimes b_{})&=& a\otimes b- \omega(a_{}\otimes b_{(1)})b_{(2)}\otimes b_{(3)}\,.
   \end{eqnarray}
  Thus $\pi\circ\omega^U=\mathrm{id}_{\Upsilon^U_\mathcal{C}}$ and $\hat\omega\circ\omega^U=0$. It causes the bicomodule decomposition $\mathcal{C}\otimes\mathcal{C}=\mathrm{Im}\Delta\oplus \mathrm{Im}\omega^U$. Moreover, introducing a new functional $\bar\omega^U=(\varepsilon\otimes\varepsilon)\circ\omega^U$ it turns out that $\bar\omega^U([a\otimes b])=\varepsilon(a)\varepsilon(b)-\omega(a\otimes b)$ and the universal coderivations becomes internal, i.e., $\delta^U= (\mathrm{id}\otimes\bar\omega^U)\circ\Delta^U_L-(\bar\omega^U\otimes\mathrm{id})\circ\Delta^U_R\,$. Since arbitrary codervation $\delta:\Upsilon\rightarrow\mathcal{C}$ has a form $\delta=\delta^U\circ\hat\delta$, where $\hat\delta\in\mathrm{Com}^{(\mathcal{C},\mathcal{C})}(\Upsilon, \Upsilon^U_\mathcal{C})$, see Proposition \ref{Doi} and formula \eqref{morphism3} below, then it automatically becomes internal for coseparable coalgebras.
 \end{rrm}
  However, in the general case, the following is true.
 
 \begin{pr}\label{split_couniv} The exact sequence \eqref{couniv} splits as a sequence of right (resp. left) comodules. In particular, the right (resp. left) comodule $\mathcal{C}\otimes\mathcal{C}$ decomposes into a direct sum $\mathcal{C}\otimes\mathcal{C}=(\mathrm{Ker}\varepsilon\otimes\mathcal{C})
 	\oplus \mathrm{Im}\Delta$ (resp. $\mathcal{C}\otimes\mathcal{C}= (\mathcal{C}\otimes\mathrm{Ker}\varepsilon)
 	\oplus \mathrm{Im}\Delta$ ) of its right (resp. left) subcomodules. 
 	\begin{proof}
 		We provide the proof for the right comodules. First, define the map $r_\varepsilon (a\otimes b)=\varepsilon(a)b$. It is a right comodule map $r_\varepsilon:\mathcal{C}\otimes\mathcal{C}\rightarrow \mathcal{C}$ such that $$r_\varepsilon\circ\Delta=\mathrm{id}_\mathcal{C}\,.$$
 		Next, define the right comodule map $\sigma_R\doteq (\delta^U\otimes\mathrm{id})\circ\Delta^U_R:\Upsilon^U_\mathcal{C} \rightarrow \ker\varepsilon\otimes\mathcal{C}\subset\mathcal{C}\otimes\mathcal{C}$ by the formula (c.f. Lemma \ref{lemma1}, B))
 		\begin{equation}\label{sigmaR}
 			\sigma_R([a\otimes b])=(a\varepsilon (b_{(1)})-b_{(1)}\varepsilon (a))\otimes b_{(2)}=
 			a\otimes b-\varepsilon(a)\Delta(b)
 		\end{equation}
 		such that $\pi\circ\sigma_R=\mathrm{id}_{\Upsilon^U_\mathcal{C}}$. In particular, $\sigma_R$ is injective, since its kernel is trivial. The last thing we need to check is that  $r_\varepsilon\circ\sigma_R=0$ and  $\mathrm{Im}\,\sigma_R=\mathrm{Ker}\varepsilon\otimes\mathcal{C}$. In fact, assuming $\varepsilon(a)=0$ one gets from \eqref{sigmaR}
 		\begin{equation}
 				\sigma_R([a\otimes b])= 
 			a\varepsilon (b_{(1)})\otimes b_{(2)}=a\otimes b\,, 
 		\end{equation}
 That completes the proof.		
 		 	\end{proof}
 \end{pr}
 We recall that a right free comodule has the form $V\otimes\mathcal{C}$, where $V$ is an arbitrary $\mathbb{K}$ space.
 \begin{cor}\label{cor_sigmaR_iso}
 	The mapping $\sigma_R$ (resp. $\sigma_L$) provides a right (resp. left) comodule isomorphism between $\Upsilon^U_\mathcal{C}$ and the right free comodule
 	$\mathrm{Ker}\varepsilon\otimes\mathcal{C}$ (resp. left free $ \mathcal{C}\otimes \mathrm{Ker}\varepsilon$),  where	 
 	\begin{eqnarray}\label{sigmaL}
 	 	\sigma_L([a\otimes b])=a_{(1)} \otimes(a_{(2)}\varepsilon (b)-b\varepsilon (a_{(2)}))
 	 	=-a\otimes b + \varepsilon(b)\Delta(a) \,.
 	 \end{eqnarray} 
     We can define the left inverse $\sigma_R^{-1}=\pi|_{\mathrm{Ker}\varepsilon\otimes\mathcal{C}}\,$, i.e.
     \begin{eqnarray}
 	 	\sigma_R^{-1}(a\otimes b)&=&[a\otimes b].
 	 \end{eqnarray}	
 if $\varepsilon(a)=0$ and,  similarly, $\sigma_L^{-1}(a\otimes b)=-[a\otimes b]$	 if $\varepsilon(b)=0$ . The subspace  
 $\sigma_R^{-1}(\mathrm{Ker}\varepsilon\otimes\mathrm{Ker}\varepsilon)=
 \ \sigma_L^{-1}(\mathrm{Ker}\varepsilon\otimes\mathrm{Ker}\varepsilon)= \mathrm{Ker}\delta^U$, c.f. \eqref{kerdelta}. \footnote{$\mathrm{Ker}\delta^U$ contains e.g. the diagonal elements $[a\otimes a]$ with $\varepsilon(a)\neq 0$, and then the elements $[a\otimes b+b\otimes a]$. (For group-like $a$ one gets $[a\otimes a]=0$.).}
We notice that $\mathrm{Ker}\varepsilon\otimes\mathcal{C}$ (resp. $ \mathcal{C}\otimes \mathrm{Ker}\varepsilon$) is a right (resp. left) free comodule which can be endowed with a bicomodule structure if we define
 induced left (resp. right) comodule structure by
 	\begin{eqnarray}
 	\widehat{\Delta^U}_L(a\otimes b)=a_{(1)} \otimes(((a_{(2)}\varepsilon (b_{(1)})-b_{(1)}\varepsilon (a_{(2)}))\otimes b_{(2)})\,,
 \end{eqnarray}
for $a\in \mathrm{Ker}\varepsilon$. Resp.
	\begin{eqnarray}
	\widehat{\Delta^U}_R(a\otimes b)=(a_{(1)} \otimes((a_{(2)}\varepsilon (b_{(1)})-b_{(1)}\varepsilon (a_{(2)})))\otimes b_{(2)}\,,
	\end{eqnarray}
for $b\in \mathrm{Ker}\varepsilon$. The last equations come from the embedding  $\mathrm{Ker}\varepsilon\otimes\mathcal{C}$ (resp. $ \mathcal{C}\otimes \mathrm{Ker}\varepsilon$) in $\mathcal{C}\otimes\mathcal{C}$. 
 	\end{cor}
\begin{rrm}\label{induced}
 	Let $\phi: \mathcal{C}\rightarrow\mathcal{D}$ be a coalgebra map, i.e. $(\phi\otimes\phi)\circ\Delta^\mathcal{C}=\Delta^\mathcal{D}\circ\phi$ and $\varepsilon_\mathcal{C}=\varepsilon_\mathcal{D}\circ\phi$. Then any left/right/bicomodule $\Upsilon$ over $\mathcal{C}$ automatically becomes a left/right/bicomodule over $\mathcal{D}$ if we introduce induced left coactions $\Delta^\phi_L=(\phi\otimes\mathrm{id})\circ\Delta_L$ which has the following property
$$(\mathrm{id}\otimes\Delta^\phi_L)\circ\Delta^\phi_L=(\Delta^\mathcal{D}\otimes\mathrm{id})\circ\Delta^\phi_L=(\phi\otimes\phi\otimes\mathrm{id})\circ(\Delta^\mathcal{C}\otimes\mathrm{id})\circ\Delta_L\,,$$
 and analogously for the case of a right  induced  coaction $\Delta^\phi_R=(\mathrm{id}\otimes\phi)\circ\Delta_R$. For a bicomodule module structure map (see Lemma \ref{lemma1} above) one gets instead 
 $$(\Delta^\phi_L\otimes\mathrm{id})\circ\Delta^\phi_R=(\phi\otimes\mathrm{id}\otimes\phi)\circ(\Delta_L\otimes\mathrm{id})\circ\Delta_R\,.$$
 Moreover, any $\mathcal{C}$ codervation
 	$\delta:\Upsilon\rightarrow\mathcal{C}$  becomes coderivation over $\mathcal{D}$, i.e. more exactly $\phi\circ\delta:\Upsilon^\phi \rightarrow\mathcal{D}$, where $\Upsilon^\phi$ bears the induced left and right coactions.	 
 \end{rrm}
 If $\phi$ is injective, i.e. $\mathcal{C}\subset \mathcal{D}$, then $\mathcal{C}$ is called subcoalgebra of $\mathcal{D}$ or $\mathcal{D}$ is called coalgebra extension of $\mathcal{C}$. In particular, any $\mathcal{C}$ comodule can be considered as well as a comodule over an arbitrary extension of $\mathcal{C}$. To avoid this non-uniqueness, one can additionally assume that
 $\mathcal{C}$ is the smallest coalgebra for $M$, i.e., $\mathcal{C}=C(M)$ \footnote{If $C(M)$ is a minimal subcoalgebra for $M$ then it does not need to be minimal for a codifferential $\delta: M\rightarrow \mathcal{C}$, in a sense that the image $\delta(M)$ not necessarily belongs to $C(M)$.}. 
 If $\phi$ is surjective, then $\ker\phi$ is a coidal and $\mathcal{D}\equiv \mathcal{C}/\ker\epsilon$.
We are now in a position to define a morphism between two coderivations.

\begin{defi} A morphism between two coderivations
$\delta_i: \Upsilon_i\rightarrow \mathcal{C}_i\,, i=1,2$ is as a pair $(\phi,\Psi)$, where  	$\phi: \mathcal{C}_1\rightarrow \mathcal{C}_2$ is a coalgebra map and $\Psi: \Upsilon_1\rightarrow  \Upsilon_2$ a $\mathcal{C}_2$-bicomodule  (called also colinear) map making the following diagram commutative (c.f. Lemma \ref{lemma1})
 \begin{eqnarray}\label{morphism-gen}
 	\begin{tikzpicture}
 		\matrix (m) [matrix of math nodes,row sep=3em,column sep=4em,minimum width=2em]
 		{
 			\mathbb{K}&\mathcal{C}_1& \Upsilon_1 && \mathcal{C}_2\otimes \Upsilon_1\otimes\mathcal{C}_2 \\
 		\ &\mathcal{C}_2 & \Upsilon_2  && \mathcal{C}_2\otimes \Upsilon_2\otimes\mathcal{C}_2 \\};
 		\draw[->] (m-1-3) edge node [above] {$\delta_1$} (m-1-2);
 		\draw[->] (m-1-2) edge node [above] {$\varepsilon_1$} (m-1-1);
 		\draw[->] (m-2-2) edge node [above] {$\varepsilon_2$} (m-1-1);
 		\draw[->] (m-1-3) edge node [above] {$(\phi\otimes\mathrm{id}\otimes\phi)\circ {}_L \Delta^1_{R}$} (m-1-5);
 		\draw[->] (m-2-3) edge node [above] {$\delta_2$} (m-2-2);
 		\draw[->] (m-2-3) edge node [above] {$ {}_L \Delta^2_{R}$} (m-2-5);
 		\draw[->] (m-1-2) edge node [right] {$\phi$} (m-2-2);
 		\draw[->] (m-1-3) edge node [right] {$\Psi$} (m-2-3);
 		\draw[->] (m-1-5) edge node [right] {$\mathrm{id}\otimes\Psi\otimes\mathrm{id}$} (m-2-5);
 	\end{tikzpicture}
 \end{eqnarray}
   
 i.e. $\phi\circ\delta_1=\delta_2\circ\Psi$ and $\Psi(\upsilon)_{(-1)}\otimes\Psi(\upsilon)_{<0>}\otimes \Psi(\upsilon)_{(1)}=\phi(\upsilon_{(-1)})\otimes\Psi(\upsilon_{<0>})\otimes \phi(\upsilon_{(1)})$. 
 \end{defi}
 The diagram commutes when all objects are considered over the algebra $\mathcal{C}_2$.  In particular,  
 every $\mathcal{C}_1$ comodule (bicomodule) becomes $\mathcal{C}_2$ one and every $\mathcal{C}_1$ coderivation becomes 	$\mathcal{C}_2$ coderivation.  
 
 Setting $\mathcal{C}_1=	\mathcal{C}_2=	\mathcal{C}$ and $\phi=\mathrm{id}_\mathcal{C}$ one gets
 \begin{equation}\label{morphism1}
 	 \delta_1=\delta_2\circ\Psi\,.
 \end{equation}
 i.e., the diagram below needs to be commutative:
  \begin{eqnarray}\label{morphism2}
 	\begin{tikzpicture}
 		\matrix (m) [matrix of math nodes,row sep=3em,column sep=4em,minimum width=2em]
 		{
 			\Upsilon_1 & \mathrm{Ker}\varepsilon\subset\mathcal{C} \\
 			\Upsilon_2 & \\};
 		\draw[->] (m-1-1) edge node [above] {$\delta_1$} (m-1-2);
 		\draw[->] (m-2-1) edge node [above] {$\delta_2$} (m-1-2);
 		\draw[->] (m-1-1) edge node [left] {$\Psi$} (m-2-1);
 	\end{tikzpicture}
 \end{eqnarray}  
 This implies $\mathrm{Ker}\Psi\subset\mathrm{Ker}\delta_1$, $\Psi(\mathrm{Ker}\delta_1)\subset\mathrm{Ker}\delta_2$ and $\delta_2(\mathrm{Im}\Psi)=\mathrm{Im}\delta_1\subset  \mathrm{Im}\delta_2$.  
    \begin{rrm}(Definition of a factor derivation)
 	If $\Psi$ in the diagram above is a bicomodule epimorphism, i.e. $\Upsilon_2=\Upsilon_1/\mathrm{Ker} \Psi$, then 
 	$\delta_2$ is called a factor derivation. We are most interested in the case when $\mathrm{Ker}\Psi$ is a maximal subbicomodule of $\mathrm{Ker}\delta_1$.
 \end{rrm}
 \begin{rrm}
 	$\delta^U$ can be considered as a factor derivation $\tilde\delta: \mathcal{C}\otimes \mathcal{C}\rightarrow \mathcal{C}$ by the subicomodule $\mathrm{Im}\,\Delta$, where $\tilde\delta (a\otimes b)=a\varepsilon(b)-b\varepsilon(a)$, and $\tilde\delta(\mathrm{Im}\,\Delta)=0$. One should notice that  $\mathrm{Im}\,\Delta$ is a maximal subbicomodule of $\mathrm{Ker}\tilde\delta$. It means, since $\pi(\mathrm{Ker}\tilde\delta)=\mathrm{Ker}\delta^U$, that there is no proper subbicomodules in $\mathrm{Ker}\delta^U$. Even more, one finds the following	 
 \end{rrm}
 \begin{lema}
 	There are no proper right or left subcomodules in $\mathrm{Ker}\delta^U$.
 	\begin{proof}
 	We show for right comodules. Consider an element $x=\sum_i [a^i\otimes b^i]\in \mathrm{Ker}\delta^U$.
  We want $\Delta^U_R(x)\in \mathrm{Ker}\delta^U\otimes\mathcal{C}$. Therefore, $(\delta^U\otimes \mathrm{id})(\Delta^U_R(x))=\sigma_R(x)$ should vanish. But, c.f. \eqref{sigmaR}, $\sigma_R(x)= \sum_ia^i\otimes b^i -\sum_i\varepsilon(a^i)\Delta(b^i)=0$ means that the element $x$ is trivial. 
 	\end{proof}	 
 \end{lema}
 
 Most important, a bicomodule  $\Upsilon^U_\mathcal{C}$ enjoys the  the following universality property: 
 \begin{pr}\label{Doi} (Y. Doi \cite{Doi81}) There is one-to-one correspondence between
 	\begin{equation}
 \mathrm{Com}^{(\mathcal{C},\mathcal{C})}(\Upsilon, \Upsilon^U_\mathcal{C})	\longleftrightarrow	 \mathrm{Coder}(\Upsilon,\mathcal{C})
 	\end{equation}
 which is an isomorphism of $\mathbb{K}$-spaces.
 	\begin{proof}
 		Obviously, having $\psi\in  \mathrm{Com}^{(\mathcal{C},\mathcal{C})}(\Upsilon, \Upsilon^U_\mathcal{C})$ we can define the corresponding $\delta_\psi=\delta^U\circ\psi$. Conversely, for given $\delta: \Upsilon\rightarrow \mathcal{C}$ one finds an inverse transformation by setting ($\epsilon\circ\delta=0$ ):
 		\begin{equation}
 			\psi_\delta(m)\equiv \hat\delta (m)\doteq [\delta(m_{<0>})\otimes m_{(1)}]\,= -\,[m_{(-1)}\otimes \delta(m_{<0>})]\,.\label{hatdelta}
 		\end{equation}
 		It follows from \eqref{cLeibniz} that the last two expressions differ by $[\Delta(\delta(m))]=0$.	 
 	\end{proof}
 \end{pr}
In other words there is a morphism $\psi=\hat\delta$ from a given derivation $\delta: \Upsilon\rightarrow \mathcal{C}$ to the universal one
 \begin{equation}\label{morphism3}
 	\delta=\delta^U\circ\hat\delta\,.
 \end{equation}
 such that the diagram below is commutative:
 \begin{eqnarray}\label{morphism4}
 	\begin{tikzpicture}
 		\matrix (m) [matrix of math nodes,row sep=3em,column sep=4em,minimum width=2em]
 		{
 			\Upsilon & \mathrm{Ker}\varepsilon\subset\mathcal{C} \\
 			\Upsilon^U_\mathcal{C} & \\};
 		\draw[->] (m-1-1) edge node [above] {$\delta$} (m-1-2);
 		\draw[->] (m-2-1) edge node [above] {$\delta^U$} (m-1-2);
 		\draw[->] (m-1-1) edge node [left] {$\hat\delta$} (m-2-1);
 	\end{tikzpicture}
 \end{eqnarray}  
implying $\mathrm{Ker}\hat\delta\subset\mathrm{Ker}\delta$, $\hat\delta(\mathrm{Ker}\delta)\subset\mathrm{Ker}\delta^U$ and $\delta^U({\mathrm{Im}}\hat\delta)={\mathrm{Im}}\delta\subset  \mathrm{Im}\delta^U= \mathrm{Ker}\varepsilon$.
It appears that $\widehat{\delta}=\pi\circ\sigma_R=-\pi\circ\sigma_L$.
 \begin{rrm} Since $\mathrm{Coder}(\Upsilon,\mathcal{C})\subset \mathrm{Hom}(\Upsilon,\mathcal{C})\leftrightarrow	\mathrm{Com}^{(\mathcal{C},\mathcal{C})}(\Upsilon, \mathcal{C}^{\otimes 3} )$ then, following Lemma \ref{lemma1}, for any coderivation $\delta$ one can associate the bicomodule map $\tilde\delta(m)=m_{(-1)}\otimes \delta(m_{<0>})\otimes m_{(1)}$. In particular, $\widetilde{\delta^U}\circ\pi=\Delta\otimes\mathrm{id}-\mathrm{id}\otimes\Delta$ \cite{Quillen1988}.
\end{rrm}
\begin{lema}\label{lema_[phiophi]}
If $\phi\circ\delta_1=\delta_2\circ\Psi$ is a morphism of two coderivations (cf. (\ref{morphism-gen}, \ref{morphism1})) then
	\begin{equation}\label{morphism5}
    [\phi\otimes\phi]\circ\widehat{\delta_1}=\widehat{\delta_2}\circ\Psi\,.
		\end{equation}
\begin{proof}
 We calculate (c.f. \eqref{hatdelta}): $\widehat{\delta_2} (\Psi(v))=[\delta_2\Psi(v)_{<0>}\otimes \Psi(v)_{(1)}] =
[\delta_2\Psi(v_{<0>})\otimes \phi(v_{(1)})]=$\\$[\phi(\delta_1 v_{<0>})\otimes \phi(v_{(1)})]  $, where the coalgebra morphism $\phi:\mathcal{C}_1\rightarrow \mathcal{C}_2$ implies a bicomodule morpism $[\phi\otimes\phi]:\Upsilon^U_{\mathcal{C}_1}\rightarrow\Upsilon^U_{\mathcal{C}_2}$, i.e. $[a\otimes b]\mapsto [\phi(a)\otimes\phi(b)]$
 as shown on the diagram below which extends the previous one \eqref{morphism-gen} by incorporating universal coderivations 
\begin{eqnarray}\label{morphism-genU}
	\begin{tikzpicture}
		\matrix (m) [matrix of math nodes,row sep=3em,column sep=4em,minimum width=2em]
		{\mathbb{K}	&\mathcal{C}_1&\Upsilon^U_{\mathcal{C}_1} & \Upsilon_1 && \mathcal{C}_2\otimes\Upsilon_1\otimes  \mathcal{C}_2\\
		\ &	\mathcal{C}_2 &\Upsilon^U_{\mathcal{C}_2} & \Upsilon_2  &&  \mathcal{C}_2\otimes\Upsilon_2\otimes\mathcal{C}_2 \\};
			\draw[->] (m-1-2) edge node [above] {$\varepsilon_1$} (m-1-1);
			\draw[->] (m-2-2) edge node [above] {$\varepsilon_2$} (m-1-1);
		\draw[->] (m-1-3) edge node [above] {$\delta^U_1$} (m-1-2);
		\draw[->] (m-1-4) edge node [above] {$\widehat{\delta_1}$} (m-1-3);
		\draw[->] (m-2-4) edge node [above] {$\widehat{\delta_2}$} (m-2-3);
		\draw[->] (m-1-4) edge node [above] {$(\phi\otimes\mathrm{id}\otimes\phi)\circ {}_L \Delta^1_{R}$ }(m-1-6);
		\draw[->] (m-2-3) edge node [above] {$\delta^U_2$} (m-2-2);
		\draw[->] (m-2-4) edge node [above] {$ {}_L \Delta^2_{R}$} (m-2-6);
		\draw[->] (m-1-2) edge node [right] {$\phi$} (m-2-2);
		\draw[->] (m-1-4) edge node [right] {$\Psi$} (m-2-4);
		\draw[->] (m-1-3) edge node [right] {$[\phi\otimes\phi]$} (m-2-3);
		\draw[->] (m-1-6) edge node [right] {$\mathrm{id}\otimes\Psi\otimes\mathrm{id}$} (m-2-6);
	\end{tikzpicture}
\end{eqnarray}
As already mentioned, the mappings $\Psi$ and $[\phi\otimes\phi]$ are bicomodule morphisms over $\mathcal{C}_2$.
\end{proof}
\end{lema}
\begin{rrm}\label{rrm_morphism} There are some special cases to be discussed:
	\begin{itemize}
		\item If $\phi$ is injective, i.e. $\mathcal{C}_1\equiv\mathrm{Im}\phi$ is a subcoalgebra of $\mathcal{C}_2$ then $[\phi\otimes\phi]$ is also injective, and $\Upsilon^U_{\mathcal{C}_1}\subset \Upsilon^U_{\mathcal{C}_2}$.
	 \item If $\phi$ is surjective then its kernel is a coideal in $\mathcal{C}_1$ and $\mathcal{C}_2\equiv \mathcal{C}_1/\mathrm{Ker}\phi$. Then $[\phi\otimes\phi]$ is also surjective and $\Upsilon^U_{\mathcal{C}_2}\equiv\Upsilon^U_{\mathcal{C}_1}/\mathrm{Ker}[\phi\otimes\phi]$.
	 	\item If $\phi$ is an automorpismm then $[\phi\otimes\phi]$ is an automorphism.
\end{itemize}
\end{rrm}  
 These mean that with any coderivation $\delta:\Upsilon\rightarrow \mathcal{C}$ one can associate a unique subbicomodule  $\hat\delta(\Upsilon)$ of $\Upsilon^U_\mathcal{C}$ such that the coderivation $\delta$ is the restriction of $\delta^U$ to the image $\hat\delta(\Upsilon)$.  Consequently, the bicomodule map $\hat\delta$ becomes a morphism between these two coderivations, i.e. $\delta^U\circ\hat\delta=\delta$. For example, for an internal derivation $\delta_\alpha\,,\alpha\in \mathcal{C}^*$ one gets $\widehat{\delta_\alpha}(\mathcal{C})=\mathrm{span}\{[a_{(1)}\alpha(a_{(2)})\otimes a_{(3)}]\,|\, a\in \mathcal{C}\}$.
  However, the correspondence between subbicomodules and coderivations is not one-to-one since non-isomorphic coderivations can have the same subbicomodule $\hat\delta_1(\Upsilon_1)=\hat\delta_2(\Upsilon_2)\subset \Upsilon^U_\mathcal{C}$. To fix this problem, we introduce the following definition
  
 \begin{defi}\label{def_cdiff}
 	First Order Codifferential Calculus (FOCC) over a coalgebra $\mathcal{C}$ is a pair $(\Upsilon,\delta)$, where $\Upsilon\in 
 	{}^\mathcal{C} \mathfrak{M}^\mathcal{C}$ and $\delta: \Upsilon\rightarrow \mathcal{C}$  is a coderivation, such that  
 	the corresponding bicomodule homomorphism $\hat\delta$ is injective, i.e. $\mathrm{Ker}\hat{\delta}=0$.	 
 \end{defi}
 The last requirement is dual to the one that is present in the definition of a First Order Differential Calculus (FODC): a bimodule of one-forms should be generated by differentials. The universal coderivation, of course, satisfies the above condition ($\widehat{\delta^U}=\mathrm{id}_{\Upsilon^U(\mathcal{C})}$).
 As an immediate consequence, one gets:
  \begin{rrm}
 	The classification of FOCC on $\mathcal{C}$ is equivalent to the classification of subbicomodules of $\Upsilon^U_\mathcal{C}$ up to automorphisms of $\mathcal{C}$. This is because every FOCC is isomorphic to the restriction of $\delta^U$ to some subbimodule of $\Upsilon^U_\mathcal{C}$. For example, taking subbicomodule $\Upsilon\subset \mathcal{C}\otimes \mathcal{C}$ one gets $\pi(\Upsilon)\subset \Upsilon^U_\mathcal{C}$. For a more specific subclass, one can assume $\Upsilon=\mathcal{L}\otimes \mathcal{R}$, where $\mathcal{L}$ (resp. $\mathcal{R}$) is a left (resp. right) subcomodule (left (resp. right) coideal) of $\mathcal{C}$. Alternatively, one can look for subspaces of $W\subset\mathrm{Ker}\varepsilon$, such that $\sigma_R^{-1}(W\otimes \mathcal{C})$ (resp. $\sigma_L^{-1}(\mathcal{C}\otimes W)$ is a subbimodule in $\mathcal{C}\otimes\mathcal{C}$.
    But not all subbicomodules of $\Upsilon^U_\mathcal{C}$ are either of these forms. More systematically, it would be convenient to decompose a colagebra $\mathcal{C}$ into its indecomposable components and perform an analogous decomposition for $\Upsilon^U_\mathcal{C}$, c.f. next section (for canonical decompositions see, e.g.,\cite{Montgomery,Xu}).
 \end{rrm}
  \begin{ex}\label{ex_C1otimesC2/()}
 	Consider a non-simple coalgebra $\mathcal{C}$. For every two non-trivial subcoalgebras $\mathcal{C}_1,\mathcal{C}_2\subseteq\mathcal{C}$ one can associate the subicomodule:
 	\begin{eqnarray}
 		\pi(\mathcal{C}_1\otimes\mathcal{C}_2)&=&\frac{\mathcal{C}_1\otimes\mathcal{C}_2}{\mathrm{Im}\Delta\cap\mathcal{C}_1\otimes\mathcal{C}_2}=\frac{\mathcal{C}_1\otimes\mathcal{C}_2}{\mathrm{Im}\Delta^{\mathcal{C}_1\cap\mathcal{C}_2}}.\label{ex_c1oc2/()_eq}
 	\end{eqnarray}
 	where $\pi$ is the canonical projection. The corresponding FOCC has values in $\mathcal{C}_1+\mathcal{C}_2\subseteq\mathcal{C}$.
 	
 	In particular, for a coaugmented coalgebra, we can choose one subcoalgebra generated by a group-like element $ \{\mathbb{K}e\}$, and then  for any coaugmented subcoalgebra $\mathcal{C}_1$ we obtain FOCCs:
 	\begin{eqnarray}
 		\pi(\mathcal{C}_1\otimes \{\mathbb{K}e\})=\overline{\mathcal{C}_1}\otimes  \{\mathbb{K}e\},&\qquad&
 		\pi(\{\mathbb{K}e\}\otimes\mathcal{C}_1 )=\{\mathbb{K}e\}\otimes\overline{\mathcal{C}_1},
 	\end{eqnarray}
 	where $\overline{\mathcal{C}_1}=\mathcal{C}_1/\{e\cdot\mathbb{K}\}$ is a factor bicomodule. 
 	\begin{proof}
 	  We see that $\mathrm{Im}\Delta^{\mathcal{C}_1\cap\mathcal{C}_2}\subseteq\mathrm{Im}\Delta\cap\mathcal{C}_1\otimes\mathcal{C}_2$. However, we show that these subspaces are equal. Notice that  $\mathrm{Im}\Delta\cap\mathcal{C}_1\otimes\mathcal{C}_2=\{\Delta x\in\mathcal{C}_1\otimes\mathcal{C}_2| x\in\mathcal{C}\}\,.$
 		From the other hand:
 		\begin{eqnarray}
 			(\Delta\otimes\mathrm{id})\circ\Delta x&\in&\mathcal{C}_1\otimes\mathcal{C}_1\otimes\mathcal{C}_2,\nonumber\\
 			(\mathrm{id}\otimes\Delta)\circ\Delta x&\in&\mathcal{C}_1\otimes\mathcal{C}_2\otimes\mathcal{C}_2,\nonumber\\
 			\Rightarrow\Delta^2x&\in&\mathcal{C}_1\otimes(\mathcal{C}_1\cap\mathcal{C}_2)\otimes\mathcal{C}_2\,.\nonumber
 		\end{eqnarray}
 		So $x=(\varepsilon\otimes\mathrm{id}\otimes\varepsilon)\circ\Delta^2 x\in\mathcal{C}_1\cap\mathcal{C}_2,$
 		which means $\mathrm{Im}\Delta\cap\mathcal{C}_1\otimes\mathcal{C}_2\subseteq\mathrm{Im}\Delta^{\mathcal{C}_1\cap\mathcal{C}_2}$. 
 	\end{proof}
 \end{ex}

 \subsection{Direct sum decompositions}
 We recall see, eg., \cite{BrzezinskiWisbauer} that for any family $\{\mathcal{C}_i\}_{i\in\Lambda}$ of coalgebras one can introduce on the vector direct sum $\mathcal{C}= \oplus_{i\in\Lambda}\mathcal{C}_i$ the structure of coalgebra by settings the coproduct $\Delta:\mathcal{C}\rightarrow\mathcal{C}\otimes\mathcal{C}$ and the counit
 $\varepsilon: \mathcal{C}\rightarrow\mathbb{K}$  as unique maps satisfying the properties
 \begin{eqnarray}
 	\Delta\circ\xi_k=(\xi_k\otimes\xi_k)\circ\Delta_k\quad \mbox{and}\quad \varepsilon\circ\xi_k=\varepsilon_k\quad \mbox{for each}\quad k\in\Lambda\,,
 \end{eqnarray} 
 where $\xi_k:\mathcal{C}_k\rightarrow\mathcal{C}$ are canonical inclusions (now coalgebra morpfisms), $\Delta_k:\mathcal{C}_k\rightarrow\mathcal{C}_k\otimes\mathcal{C}_k\subset \mathcal{C}\otimes\mathcal{C}$ and $\varepsilon_k: \mathcal{C}_k\rightarrow\mathbb{K}$ denote the corresponding coproducts and counits.\footnote{Here $\Lambda$, unless otherwise stated, is an arbitrary set. Further, if no confusion arises, we skip writing it explicitly.} Thus we call $( \mathcal{C}, \Delta,\varepsilon)$ the direct sum  (or coproduct) of  the coalgebras $( \mathcal{C}_i, \Delta_i,\varepsilon_i)_{i\in\Lambda}$. Moreover, the universal property holds: for any family $\phi_i:\mathcal{C}_i\rightarrow\mathcal{D}$ of coalgebra morphisms there exists a unique coalgebra
 morphism $\phi:\mathcal{C}\rightarrow\mathcal{D}$  making the diagram commutative:
 \begin{eqnarray}\label{uni3}
 	\begin{tikzpicture}
 		\matrix (m) [matrix of math nodes,row sep=3em,column sep=4em,minimum width=2em]
 		{
 			\mathcal{C}_k & \mathcal{C}=\oplus_i \mathcal{C}_i\\
 			\ &\mathcal{D}\\};
 		\draw[->] (m-1-1) edge node [above] {$\xi_k$} (m-1-2);
 		\draw[->] (m-1-2) edge node [right] {$\phi$} (m-2-2);
 		\draw[->] (m-1-1) edge node [left] {$\phi_k$} (m-2-2);
 	\end{tikzpicture}
 \end{eqnarray} 
  \begin{cor}
     In contrast to the canonical inclusions  $\xi_n$, canonical projections  $\pi_n:\mathcal{C}\rightarrow\mathcal{C}_n$  are not coalgebra morphisms. We have $\Delta_n\circ\pi_n=(\pi_n\otimes\pi_n)\circ\Delta$,  $\pi_n\circ\xi_k\circ\pi_k=\delta_{nk}\pi_k$ and $\varepsilon\circ\xi_n\circ\pi_n=\varepsilon_n\circ\pi_n $, i.e., $\varepsilon_n\circ\pi_n$ is not equal to $\varepsilon$\,.
 \end{cor}
 Similarly, for any family $\Upsilon_i$ of right (left or bi) comodules over a coalgebra $\mathcal{C}$ one can uniquely define their direct sum (coproduct) $\Upsilon=\oplus_i\Upsilon_i$ together with the following universal property: for any family of comodule morphisms $\Phi_k: \Upsilon_k\rightarrow\Gamma$ there exists a unique (right/left/bi) comodule morphism  $\Phi: \Upsilon\rightarrow\Gamma$
 satisfying the relations $\Phi\circ\bar\xi_k=\Phi_k$ as shown on the commutative diagram
 \begin{eqnarray}\label{uni3b}
 	\begin{tikzpicture}
 		\matrix (m) [matrix of math nodes,row sep=3em,column sep=4em,minimum width=2em]
 		{
 			\Upsilon_k & \Upsilon=\oplus_i \Upsilon_i \\
 			\ &\Gamma \\};
 		\draw[->] (m-1-1) edge node [above] {$\bar\xi_k$} (m-1-2);
 		\draw[->] (m-1-2) edge node [right] {$\Phi$} (m-2-2);
 		\draw[->] (m-1-1) edge node [left] {$\Phi_k$} (m-2-2);
 	\end{tikzpicture}
 \end{eqnarray}  
 where $\bar\xi_k:\Upsilon_k\rightarrow \Upsilon$ are canonical inclusions (now morphisms). Particularly, for a family of bicomodules $\Upsilon_i\in {}^{\mathcal{C}}\mathfrak{M}^{\mathcal{C}}$, one gets
 \begin{equation}
    {}_L{} \Delta_{R} \circ\bar\xi_k =(\mathrm{id}\otimes\bar\xi_k\otimes \mathrm{id})\circ {}_L{} \Delta^k_{R}\,.
 \end{equation}
 If additionally $\Upsilon_i\in {}^{\mathcal{C}_i}\mathfrak{M}^{\mathcal{C}_i}$ and $\mathcal{C}=\oplus_i \mathcal{C}_i$ then
 \begin{equation}
    {}_L{} \Delta_{R} \circ\bar\xi_k =(\xi_k\otimes\bar\xi_k\otimes \xi_k)\circ {}_L{} \Delta^k_{R}\,.
 \end{equation}
   
 In such a case, one can say that the decompositions $\Upsilon=\oplus_i \Upsilon_i $ is compatible with the coalgebra decomposition $\mathcal{C}=\oplus_i\mathcal{C}_i$. 
 \begin{ex}
 	Let $\mathcal{C}=\oplus_i\mathcal{C}_i$. For any vector space $V$ the free right comodule $V\otimes\mathcal{C}=\oplus_i V\otimes \mathcal{C}_i$ decomposes in a compatible way into free subcomodules. 
 	Since $\ker\varepsilon\supseteq\oplus_i\ker\varepsilon_i$ then $\oplus_i(\ker\varepsilon_i\otimes \mathcal{C}_i)$ is a subcomodule in $\ker\varepsilon\otimes \mathcal{C}$. In general, a direct sum of free comodules does not need to be a free comodule.
 \end{ex}
 \begin{pr}\begin{enumerate}
 		\item  Let as above $\Upsilon=\oplus_i \Upsilon_i $ be a direct sum of $\mathcal{C}$-bicomodules and let $\delta_i:\Upsilon_i\rightarrow\mathcal{C}$ be a family of $\mathcal{C}$-coderivations. Then there exists a unique coderivation $\delta:\Upsilon\rightarrow\mathcal{C}$ such that $\delta\circ\bar\xi_i=\delta_i$. Moreover, $\bar\xi_i$ become morphisms between coderivations.
 		\item Assume additionally that the decomposition $\Upsilon=\oplus_i \Upsilon_i $ is compatible with the coalgebra decomposition $\mathcal{C}=\oplus_i\mathcal{C}_i$ and  
 		  $\delta_i:\Upsilon_i\rightarrow\mathcal{C}_i$ are FOCC over $\mathcal{C}_i$.    
 		Then $\delta$ is a FOCC as well. We call it a direct sum of FOCCs, i.e. $\delta=\oplus_i\delta_i$.
 	\end{enumerate}
 	\begin{proof}
 		Utilizing the universal property of the vector space direct sum we define $\delta$ as a unique linear map satisfying the property $\delta\circ\bar\xi_i=\delta_i$. The remaining properties then follow from this definition.
 	\end{proof}	 
 \end{pr}
 \begin{ex} 
 	Let $\mathcal{C}=\oplus_i\mathcal{C}_i$ and $\delta^U_i: \Upsilon_{\mathcal{C}_i}^U\rightarrow\mathcal{C}_i$ denote the corresponding universal codifferentials.
 	Then  $\oplus_i\delta^U_i:\oplus_i \Upsilon_{\mathcal{C}_i}^U\rightarrow\oplus_i\mathcal{C}_i$ defines a direct sum FOCC which is \begin{color}{red}\underline{not}\end{color}\cancel{not} universal for the coalgebra  $\oplus_i\mathcal{C}_i$. This is because the universal bicomodule  $\Upsilon^U_{\mathcal{C}}$ contains all possible FOCC, which becomes a restriction of the universal coderivation to a suitable subbicomodule. More exactly, one gets:
 \end{ex}
  
\begin{pr}\label{pr_bicomodul_decom}
    Let $\mathcal{C}=\bigoplus_i\mathcal{C}_i$. For any bicomodule $M$ over $\mathcal{C}$, there exists a  direct sum decomposition into the corresponding subbicomodules:
    \begin{eqnarray}
        M=\bigoplus_{ij} M_{ij}=\bigoplus_{M_{ij}\neq \{0\}} M_{ij}\,,\qquad   M_{ij}&\doteq& {}_L \Delta_{R}^{-1} (\mathcal{C}_i\otimes M\otimes\mathcal{C}_j)
    \end{eqnarray}
    where subspaces \footnote{In fact, we can write $M_{ij}= {}_L \Delta_{R}^{-1} ((\mathcal{C}_i\otimes M\otimes\mathcal{C}_j)\cap \mathrm{Im}{}_L \Delta_{R})$ which explains why some subspaces $M_{ij}$  might be trivial even when $\mathcal{C}$ is a minimal coalgebra for $M$, see example below.}
   are defined as inverse images of the bicomodule structure map ${}_L\Delta_{R}: M\rightarrow \mathcal{C}\otimes M\otimes\mathcal{C}=\bigoplus_{ij}\mathcal{C}_i\otimes M\otimes\mathcal{C}_j$. Moreover, if $N\subset M$ is a subbicomodule and
   $ N=\bigoplus_{ij} N_{ij}$ denotes its corresponding decomposition then $ N_{ij}\subset  M_{ij}$.   
   \begin{proof}
       Since ${}_L\Delta_{R}$ is injective then $M_{ij}\cap M_{i'j'}=\{0\}$ for $(i,j)\neq (i',j')$. We need to show that ${}_L\Delta_{R}(M_{ij})\subset \mathcal{C}_i\otimes M_{ij}\otimes\mathcal{C}_j$, i.e., 
   $M_{ij}\in {}^{\mathcal{C}_i}\mathfrak{M}^{\mathcal{C}_j}\subset {}^{\mathcal{C}}\mathfrak{M}^{\mathcal{C}}$. Equivalently, we show that  $M_{ij}\in  {}_{\mathcal{C}^*}\mathfrak{M}_{\mathcal{C}^*}$ is a $\mathcal{C}^*$ subbimodule of $M$ (see Appendix A, \eqref{a1}), i.e. that  ${}_L\Delta_{R}(\alpha\star m\star\beta)\in  \mathcal{C}_i\otimes M\otimes\mathcal{C}_j$ for $m\in M_{ij}\,, \alpha,\beta\in \mathcal{C}^*$. 
    Consider ${}_L\Delta_{R}(m)=\sum_{x,y} c_x\otimes m_{xy}\otimes c_y$, where $m\in M_{ij}\,,c_x\in \mathcal{C}_i\,,c_y\in\mathcal{C}_j$, the sum is finite and without loss of generality it can be assumed  that  the elements $m_{xy}$  are linearly independent. Now, ${}_L\Delta_{R}(\alpha\star m\star\beta)= 
   (\beta\otimes {}_L \Delta_{R}\otimes\alpha)\circ {}_L \Delta_{R}(m)=  (\beta\otimes \mathrm{id}_{\mathcal{C}_i}
   \otimes\mathrm{id}_M\otimes\mathrm{id}_{\mathcal{C}_j}\otimes\alpha)\circ (\Delta_{\mathcal{C}_i}\otimes\mathrm{id}_M\otimes\Delta_{\mathcal{C}_j})\circ {}_L \Delta_{R}(m) =\sum_{x,y} \beta((c_x)_{(1)})(c_x)_{(2)}\otimes m_{xy}\otimes \alpha((c_y)_{(2)})(c_y)_{(1)}\,.$
   
   Next, we have to show that $M=\bigoplus_{ij} M_{ij}$. Indeed, if $m\in M$ then its image ${}_L\Delta_{R}(m)=\sum_{kl}\sum_{xy}c_{k x}\otimes m^{kl}_{xy}\otimes c_{ly}$ is a finite sum of some elements $\sum_{xy}c_{k x}\otimes m^{kl}_{xy}\otimes c_{l y}\in \mathcal{C}_k\otimes M_{kl}\otimes\mathcal{C}_l$ which are images of the corresponding elements from $M_{kl}$. 
   
   The last statement is obvious since $\mathcal{C}_i\otimes N\otimes\mathcal{C}_j\subset\mathcal{C}_i\otimes M\otimes\mathcal{C}_j$. The proof is done.\\
    \end{proof}
\end{pr}
\begin{ex}
 	 The coalgebra  $\mathcal{C}=\oplus_i \mathcal{C}_i$ can be treated as a bicomodule with $\mathcal{C}$ being minimal coalgebra for itself. In this case  ${}_L \Delta_{R}(c)=c_{(1)}\otimes c_{(2)}\otimes c_{(3)}$ and $\mathcal{C}_{ij}=0$ for $i\neq j$ and $\mathcal{C}_{ii}=\mathcal{C}_i $.  
 \end{ex}

 \begin{lema}\label{pr_UPU(C1otimesC2)}
 	For any decomposition \footnote{In practice, we are most interested in some special (canonical) decompositions, e.g., into indecomposable or irreducible components \cite{Montgomery,Xu,Zhang}. However, this is not a subject of the present publication.} $\mathcal{C}=\bigoplus_i\mathcal{C}_i$, its universal bicomodule  decomposes accordingly into subbicomodules:
 	\begin{eqnarray}\label{Udecomp}
 		\Upsilon_\mathcal{C}^U&=&\bigoplus_i\Upsilon_{\mathcal{C}_i}^U\oplus\bigoplus_{i\neq j}\mathcal{C}_i\otimes\mathcal{C}_j,
 	\end{eqnarray}
 	\begin{proof}
 		We know that bicomodule $\mathcal{C}\otimes\mathcal{C}=\bigoplus_{i,j}\mathcal{C}_i\otimes\mathcal{C}_j$ decomposes as a vector space into subbicomodules. From the other hand $\mathrm{Im}\Delta=\oplus_i\mathrm{Im}\Delta_i \subset\bigoplus_i\mathcal{C}_i\otimes\mathcal{C}_i$. These imply \eqref{Udecomp} after factorization as well as a family of canonical injections $\xi_{ij}$ of bicomodule morphisms ; $\bar\xi_{ii}:\Upsilon_{\mathcal{C}_i}\rightarrow\Upsilon_\mathcal{C}^U$, for $i\neq j$
 		$\bar\xi_{ij}: \mathcal{C}_i\otimes\mathcal{C}_j\rightarrow\Upsilon_\mathcal{C}^U$. Thus
 		\begin{eqnarray}
 		 \delta^U_\mathcal{C}\circ\bar{\xi}_{ii}&=& \delta^U_{\mathcal{C}_i}\,,\\
 			\delta^U_\mathcal{C}\circ\bar{\xi}_{ij}&=&\mathrm{id}_i\otimes\varepsilon_j-\varepsilon_i\otimes\mathrm{id}_j\quad \mbox{for}\quad i\neq j
 		\end{eqnarray}
 		where $\mathrm{Im}(\delta^U_\mathcal{C}\circ\bar{\xi}_{ij}) \subset\mathcal{C}_i\oplus\mathcal{C}_j \supset \mathrm{Im}(\delta^U_\mathcal{C}\circ\bar{\xi}_{ji})$.
 	\end{proof}
 \end{lema}
 \begin{rrm}
 	It follows from the above considerations that for $i\neq j$ the vector space $\mathcal{C}_i\otimes\mathcal{C}_j\oplus \mathcal{C}_j\otimes\mathcal{C}_i $ has $\mathcal{C}_i\oplus\mathcal{C}_j$ bicomodule structure with the codifferential 
 		\begin{equation}
 		\delta^U_\mathcal{C}\circ(\bar{\xi}_{ij}+\bar{\xi}_{ji})=(\mathrm{id}_i\otimes\varepsilon_j-\varepsilon_i\otimes\mathrm{id}_j) +( \mathrm{id}_j\otimes\varepsilon_i-\varepsilon_j\otimes\mathrm{id}_i)\,.
 	\end{equation}
 	As a vector space it decomposes into antisymmetric and symmetric part $\mathcal{C}_i\wedge\mathcal{C}_j\bigoplus \mathcal{C}_j\odot\mathcal{C}_i $, where the Grassmannian (wedge) product  $\mathcal{C}_i\wedge\mathcal{C}_j =\mathrm{Span}\{a\otimes b-b\otimes a|a\in \mathcal{C}_i\,,b\in\mathcal{C}_j\}$. It is ease to observe that the symmetric subspace $\mathcal{C}_i\odot\mathcal{C}_j =\mathrm{Span}\{a\otimes b+b\otimes a|a\in \mathcal{C}_i\,,b\in\mathcal{C}_j\}$ belongs to the kernel of $\delta^U_\mathcal{C}$. For $i=j$ one gets the standard decomposition  $\mathcal{C}_i\otimes\mathcal{C}_i=\mathcal{C}_i\wedge\mathcal{C}_i\bigoplus \mathcal{C}_i\odot\mathcal{C}_i $.
 \end{rrm}

  \subsection{Cocommutative case}
 We recall that on a tensor product of two arbitrary vector spaces one can define the canonical flip (switch) transformation
 $\tau: V\otimes W\rightarrow W\otimes V$ by $\tau (a\otimes b)=b\otimes a$. It satisfies $\tau^2=\mathrm{id}$ as well the celebrated braid relation
 \begin{equation}\label{braid}
 	\hat\tau\doteq(\mathrm{id}\otimes\tau)\circ(\tau\otimes \mathrm{id})\circ(\mathrm{id}\otimes\tau)=(\tau\otimes \mathrm{id})\circ(\mathrm{id}\otimes\tau)\circ (\tau\otimes \mathrm{id})
 \end{equation}
 on $V\otimes W\otimes Z$. It defines the canonical braiding $\hat\tau:V\otimes W\otimes Z\rightarrow Z\otimes W\otimes V$ by $\hat\tau(v\otimes w\otimes z)=z\otimes w\otimes v$.
 We recall that any right $\mathcal{C}$ comodule $(M, \Delta_R)$ becomes automatically a left $\mathcal{C}^\mathrm{cop}$ comodule with the coaction $\Delta^\tau_R=\tau\circ\Delta_R$, where the coalgebra $\mathcal{C}^\mathrm{cop}$ has the same counit and flipped coproduct $\Delta^\tau=\tau\circ\Delta$. Similarly, $\Delta^\tau_L=\tau\circ\Delta_L$ becomes a right $\mathcal{C}^\mathrm{cop}$ coaction. Therefore, for cocommutative coalgebras ($\Delta=\Delta^\tau$), every left (resp. right) comodule becomes automatically a cocommutative bicomodule if we set $\Delta_R=\tau\circ\Delta_L$ (resp. $\Delta_L=\tau\circ\Delta_R$)
 
 In general, one can introduce a cocommutator subspace for a bicomodule $M$ over coalgebra $\mathcal{C}$ (see \cite{Quillen1988,Khalkhali1996}) \footnote{Further, we shall use the same symbol to denote the canonical inclusion $\natural: M^\natural\rightarrow M$.}:
 \begin{eqnarray}
 	M^\natural&:=&\mathrm{Ker}(\Delta_L-\tau\circ\Delta_R)=\mathrm{Ker}(\tau\circ\Delta_L-\Delta_R)\,.
 \end{eqnarray}
 The subspace $M^\natural$ (which can be trivial -- see example~\ref{ex_P_cA1}) is not a subbicomodule in general; however, for cocommutative coalgebras, $\Delta=\Delta^\tau$, it is. Before showing this let us notice that the vector space decomposition
 \begin{equation}
 	\Upsilon^{U}_\mathcal{C}=\Upsilon^{U\wedge}_\mathcal{C}\oplus\Upsilon^{U\odot}_\mathcal{C}\,,
 \end{equation}
 into antisymmetric and symmetric subspaces, can be obtained from the obvious decomposition 
 $\mathcal{C}\otimes\mathcal{C}=\mathcal{C}\wedge\mathcal{C}\oplus \mathcal{C}\odot\mathcal{C}$,
 where    $\Upsilon^{U\wedge}_\mathcal{C}\doteq\pi(\mathcal{C}\wedge\mathcal{C})$ and $\Upsilon^{U\odot}_\mathcal{C}\doteq\pi(\mathcal{C}\odot\mathcal{C})$.  
 The Grassmannian (wedge) product  $\mathcal{C}\wedge\mathcal{C} =\mathrm{Span}\{a\otimes b-b\otimes a|a, b\in \mathcal{C}\}$ while the symmetric subspace $\mathcal{C}\odot\mathcal{C} =\mathrm{Span}\{a\otimes b+b\otimes a|a, b\in \mathcal{C}\}$.  
 It is ease to observe that $\Upsilon^{U\odot}_\mathcal{C}$ belongs to the kerner of $\delta^U_\mathcal{C}$.   
 \begin{pr} \label{natural} 
 	Let $\mathcal{C}$ be a cocommutative coalgebra, then for  any bicomodule  $M\in {}^{\mathcal{C}}\mathfrak{M}^{\mathcal{C}} $  
 	the cocommutator subspace $M^\natural\subset M$ is the biggest subbicomodule for which $\Delta_L=\Delta^\tau_R$.
 	
 	Additionally, for any coderivation $\delta: M\rightarrow \mathcal{C}$  one gets
 	$\hat\delta (M^\natural)\subseteq \Upsilon^{U\natural}_\mathcal{C}\subset \Upsilon^{U\wedge}_\mathcal{C}$.
 	
 	Moreover, if $\widehat{\delta}(M)\subset\Upsilon_\mathcal{C}^{U\wedge}$, then $\widehat{\delta}(M)\subseteq\Upsilon_\mathcal{C}^{U\natural}$.
 	\begin{proof}
 		Since the algebra is cocommutative we have two left coaction $\Delta_L$ and $\Delta^\tau_R$  (equivalently, two right coactions  $\Delta_R$ and $\Delta^\tau_L$), which coincide on $M^\natural$, i.e., 
 		\begin{eqnarray}
 			\Delta_L(m)&=& \Delta^\tau_R(m)=\sum_k c_k\otimes m_k
 		\end{eqnarray}
 		where the sum is finite and we can assume without loosing generality that $c_k\in \mathcal{C}$ are linearly independent. We have to show that $\Delta_L(M^\natural)=\mathcal{C}\otimes M^\natural$ or $\Delta_L(m_k)=\Delta^\tau_R(m_k)$. 
 		Applying $(\Delta\otimes \mathrm{id})=(\Delta^\tau\otimes \mathrm{id})$ to both sides of the last equation  one gets   $\sum_k c_k\otimes \Delta_L(m_k)=\sum_k c_k\otimes \Delta^\tau_R(m_k)$. 
 		The first embedding results from the fact that   $\hat\delta$ is a bicomodule map. Indeed, $\Delta^U_R\circ\hat\delta=(\hat\delta\otimes\mathrm{id})\circ\Delta_R=(\hat\delta\otimes\mathrm{id})\circ\Delta^\tau_L=\Delta^{U\tau}_L\circ\hat\delta$. To show the second embedding we assume that
 		$\sum_i a^i_{(1)}\otimes[a^i_{(2)}\otimes b^i]=\sum_i b^i_{(2)}\otimes[a^i\otimes b^i_{(1)}]$. Applying now $\pi\circ(\mathrm{id}\otimes \delta^U)$ to both sides one gets $\sum_i  [a^i\otimes b^i]=-\sum_i [b^i\otimes a^i]$\,.
 		For the last embedding, we assume that for all elements in $\hat\delta(M)$: $\sum_i  [a^i\otimes b^i]=-\sum_i [b^i\otimes a^i]$. Now we can check:
 		\begin{eqnarray}
 			\Delta_L(\sum_i[a^i\otimes b^i])&=&-\Delta_L(\sum_i [b^i\otimes a^i])=-\sum_ib^i_{(1)}\otimes[b^i_{(2)}\otimes a^i]\nonumber\\
 			&=&\sum_ib^i_{(2)}\otimes[a^i\otimes b^i_{(1)}]=\Delta_R^\tau(\sum_i[a^i\otimes b^i])\,.
 		\end{eqnarray}
 	\end{proof}
 \end{pr}
 \begin{ex} (\cite{Quillen1988,Khalkhali1996})
 	$(\mathcal{C}\otimes\mathcal{C})^\natural=\mathrm{Im}(\Delta^\tau)\cong\mathcal{C}$ as a vector space.  More generally, for any vector space $V$ and bifree bicomodule $(\mathcal{C}\otimes V\otimes\mathcal{C})^\natural\cong \mathcal{C}\otimes V$.
 	\begin{proof}
 		Applying $(\mathrm{id}\otimes\tau)\circ(\tau\otimes \mathrm{id})$ to the coassociativity constraints one concludes that $\mathrm{Im}(\tau\circ\Delta)\subset (\mathcal{C}\otimes\mathcal{C})^\natural$. Coversely, let $\sum_i a^i\otimes b^i\in (\mathcal{C}\otimes\mathcal{C})^\natural$, i.e.
 		$\sum_i (a^i)_{(1)}\otimes  (a^i)_{(2)}\otimes b^i =\sum_i  (b^i)_{(2)}\otimes a^i\otimes  (b^i)_{(1)}$. Applying now $\mathrm{id}\otimes\varepsilon\otimes\mathrm{id} $ to both sides yields\\
 		$\sum_i a^i\otimes b^i= \sum_i \varepsilon(a^i)(\tau\circ\Delta )(b^i)$.	 
 	\end{proof} 
 	Of course, for cocommutative coalgebra $(\mathcal{C}\otimes\mathcal{C})^\natural=\mathrm{Im}(\Delta)=\ker\pi$ is a subbicomodule, where $\pi:\mathcal{C}\otimes\mathcal{C}\rightarrow \Upsilon^U_\mathcal{C}$.
 \end{ex}
 In the case of direct sum decomposition one finds
  \begin{lema}
 	If $M=\oplus_i M_i$ then its cocommutator subspace $M^\natural=\oplus_i M_i^\natural$, where $M, M_i\in {}^{\mathcal{C}}\mathfrak{M}^{\mathcal{C}}$.
 	\begin{proof}
 		It is clear that $M^\natural\supset\oplus_i M_i^\natural$. We show the opposite. Let $M^\natural\ni m=\sum_k m_k$, where $m_k\in M_k$ and the sum is finite. Then
 		\begin{eqnarray}
 			\Delta_L(\sum_k m_k) = \Delta^\tau_R(\sum_k m_k) &= &\sum_k\Delta_L\circ\bar\xi_k (m_k)=\sum_k \Delta_R^\tau\circ \bar\xi_k (m_k)\nonumber\\
 			=	\sum_k(\mathrm{id}\otimes\bar\xi_k)\circ\Delta^k_L  (m_k)&=& 	\sum_k(\mathrm{id}\otimes\bar\xi_k)\circ\Delta_R^{k\,\tau}(m_k)\in \bigoplus_k\mathcal{C}\otimes M_k\,,
 		\end{eqnarray}
 		where $\bar\xi_k: M_k\rightarrow M$ are canonical injections. Since $m_k$ are linearly independent and all maps are injective, then one gets  the result.  
 	\end{proof}
 \end{lema}
 \begin{pr}
 	Let $\mathcal{C}=\bigoplus_i\mathcal{C}_i$ be a decomposition of cocommutative coalgebra and $M=\bigoplus M_{ij}$ the corresponding decomposition of some bicomodule $M$.  Then
 	\begin{eqnarray}
 		M^\natural&=&\bigoplus_i M^\natural_{ii}\,,
 	\end{eqnarray}
 	where $M_{ij}$ are the subbimodules from Proposition \ref{pr_bicomodul_decom}. In particular, applying this result to the universal bicomodule one gets:
 	\begin{eqnarray}
 		\Upsilon^{U\natural}_\mathcal{C}&=&\bigoplus_i\Upsilon^{U\natural}_{\mathcal{C}_i}\,.
 	\end{eqnarray}
 	\begin{proof}
 		From the last Lemma, we have $M^\natural=\bigoplus_{ij} M_{ij}^\natural$. However, the cocommutator subspace $M_{ij}^\natural$ can be non trivial only for $i=j$ since
 		\begin{eqnarray} 
 			\Delta_L(M_{ij})\subset\mathcal{C}_i\otimes M_{ij}\,\quad \mbox{and}\quad\Delta_R^\tau (M_{ij})&\subset&\mathcal{C}_j\otimes M_{ij}\,.
 		\end{eqnarray}
 		It implies  $\mathrm{Ker}(\Delta_L-\Delta_R^\tau)\cap M_{i\neq j}=0$. So the statement.
 	\end{proof}
 \end{pr}
 We could say that for cocommutative coalgebras $\mathcal{C}$, the $\Upsilon^{U\natural}$ is universal cocommutative FOCC. One can call it a co-K\"ahler FOCC \footnote{We recall that K\"ahler differential is a universal one for differential calculi over commutative algebras with equal right and left action on forms.}'
 
 
 	

\subsection{Some notes about classification of FOCCs}

In fact, following Proposition \ref{Doi} and Remark \ref{rrm_morphism}, classification of FOCC over some coalgebra $\mathcal{C}$ can be reduced to the classification of subbicomodules of the universal one  up to an isomorpism. We consider two subbicomodules isomorphic $\Upsilon_1\cong \Upsilon_2\subset \Upsilon^U_\mathcal{C}$ if there exists an automorphism $\phi: \mathcal{C}\rightarrow \mathcal{C}$ such that $[\phi\otimes\phi](\Upsilon_1)=\Upsilon_2$. It is well-known known that right (resp. left) $\mathcal{C}$-(sub)comodules can be equivalently characterized  as a left  (resp. right) (sub)modules over the algebra $\mathcal{C}^*$. Therefore, one can use the  well established terminology borrowed from  the general theory of  comodules and their decompositions, see e.g.  \cite{Montgomery,Zhang,Xu}.\footnote{However, in the current literature, similar methodology is not used for bicomodules.}
Nevertheless, it is convenient to consider  subbicomudules $\Upsilon \subset \Upsilon^U_\mathcal{C}$ equivalently as a subbimodules over the algebra $\mathcal{C}^*$ (see Appendix A, \eqref{a1})
\begin{equation}\label{a1bis}
	\alpha\star [a\otimes b]\star\beta\doteq(\beta\otimes\mathrm{id}\otimes\alpha)\circ {}_L \Delta^U_{R}([a\otimes b])=[a_{(2)}\otimes b_{(1)}]\alpha(b_{(2)})\beta(a_{(1)})\,.
\end{equation}
Then
\begin{equation}\label{a1delta}
\delta^U(\alpha\star [a\otimes b]\star\beta)=(\varepsilon(a_{(2)}) b_{(1)}- \varepsilon(b_{(1)}) a_{(2)} )\alpha(b_{(2)})\beta(a_{(1)})=  \beta(a)\alpha\star b -\alpha(b) a\star\beta\,.
\end{equation}

Further, we need the notion of generating  space.
 \begin{pr}Let $V$ be a non-trivial vector subspace $V\subset \Upsilon^U_\mathcal{C}$. We denote by $\Upsilon(V)$ the smallest subbicomodule containing $V$, which is called a generating space for $\Upsilon(V)\supset V$. Then 
    	\begin{eqnarray}\label{spanBimodule}
    		\Upsilon(V)=\mathrm{Span}\{  \alpha\star v\star\beta | v\in V, \alpha,\beta\in \mathcal{C}^* \}\,.    		
    	\end{eqnarray}
    	\begin{proof}
    		The proof follows from the $\mathcal{C}$-comodule -- $\mathcal{C}^*$-module correspondence as presented in Appendix A, \eqref{a1}.
    		\end{proof}
    \end{pr}
    
    We call $V$ a maximal generating space if $V=\Upsilon(V)$. For a given subbicomodule, can exist many non-isomorphic generating spaces, but only one maximal. 
 
\begin{lema}\label{lem_gen_space}
	Let $\Upsilon(V)$  be as above.  Thus the following holds:\\
	i)
	\begin{eqnarray}\label{subset}
		U\subset V&\Rightarrow& \Upsilon(U)\subseteq \Upsilon(V).
	\end{eqnarray}
	ii)  
	\begin{eqnarray}\label{sum}
		\Upsilon(\oplus_i V_i)&=&\mathbf{+}_i\;\Upsilon(V_i).
	\end{eqnarray}
	i.e. the internal direct sum is not, in general, preserved by this operation. \\
iii) If $\phi\in \mathrm{Aut}(\mathcal{C})$ then $\Upsilon([\phi\otimes\phi](V))=[\phi\otimes\phi](\Upsilon(V))\cong\Upsilon(V)$ (c.f. Remark \ref{rrm_morphism}).
\begin{proof}
	First two properties follow directly from \eqref{spanBimodule}. Third one comes from the formula (c.f. \eqref{a4})
	\begin{equation}
		[\phi\otimes\phi](\alpha\star[a\otimes b]\star\beta)=(\phi^{-1})^*(\alpha)\star[\phi(a)\otimes\phi(b)]\star(\phi^{-1})^*(\beta)\,.
	\end{equation}
	
\end{proof}
\end{lema}
\begin{rrm}\label{Singleton}
	It follows from  \eqref{sum} that one dimensional subspaces play a fundamental role in this decomposition. Any higher dimensional subspace generates  subbicomodule of $\Upsilon^U_\mathcal{C}$ that is a sum of subbicomodules generated by one-dimensional ones. 
	Therefore we \begin{color}{red}will\end{color}\cancel{can} focus on one-dimensional subspaces $<v>\doteq \{ \mathbb{K}\upsilon\}$ as a generator of subbimodules $\Upsilon <v>$, where $0\neq v\in \Upsilon^U_\mathcal{C}$ are called singleton in \cite{Xu}.
	One should notice that $<v>\in\mathbb{P}(\Upsilon^U_\mathcal{C})$ can be seen as an element of the projective space associated with the vector space $\Upsilon^U_\mathcal{C}$. Thus, making use of some decomposition $v=\sum_i \lambda_i v_i\in +_{\lambda_i\neq 0} <v_i>$ one gets
	\begin{equation}\label{singleton}
		\Upsilon<\sum_i \lambda_i v_i>\subseteq +_i \Upsilon<\lambda_i v_i>= +_{\lambda_i\neq 0} \Upsilon<v_i>\,.
	\end{equation}
\end{rrm}

\begin{defi}
    A finite-dimensional bicomodule $M$ is called elementary if there are \underline{no} proper subbicomodules $M_1,M_2\subset M$, such that $M=M_1+M_2$.
\end{defi}
It is important to consider \underline{proper} subbicomodules $M_1,M_2\subset M$, since requirement $M=M_1+M_2$ would be trivial otherwise. One can also verify that every finite-dimensional  elementary  bicomodule can be generated by a singleton.
\begin{rrm}
Finally, let us consider cocommutative coalgebra $\mathcal{C}$,and take $V\subset\Upsilon^{U\natural}_\mathcal{C}$. Then
	\begin{eqnarray}\label{spanBimoduleComm}
	\Upsilon^\natural(V)=\mathrm{Span}\{\alpha\star v| v\in V, \alpha\in \mathcal{C}^*\}
    =\mathrm{Span}\{v\star\alpha| v\in V, \alpha\in \mathcal{C}^*\}\,    		
\end{eqnarray}
is the smallest cocommutative subbicomodule in $\Upsilon^{U\natural}_\mathcal{C}$ containing $V$.
\end{rrm}

In the following examples, we attempt to classify, up to coalgebra automorphisms of $\mathcal{C}$ singleton-generated subbicomodules.
Such an approach provides the classification of finite-dimensional FOCC - unless we consider an infinite sum of subbimodules.

%

 \section{Illustrative examples}
 	
 	

 \begin{ex}\label{ex_M2x2_coalgebra}
     Consider 4-dimensional coalgebra of $2\times 2$ matrices $\mathcal{C}\doteq\mathrm{Span}\{x,u,v,y\}$, with the matrix coproduct:
     \begin{eqnarray}
         \Delta x=x\otimes x+u\otimes v,&\qquad&\Delta u=x\otimes u+u\otimes y,\\
         \Delta v=y\otimes v+v\otimes x,&\qquad&\Delta y=y\otimes y+v\otimes u,
     \end{eqnarray}
     and $\varepsilon (x)=\varepsilon (y)=1$, $\varepsilon (u)=\varepsilon (v)=0$. there is an involutive coalgebra automorphism $\phi:(x\leftrightarrow y, u\leftrightarrow v)$.

     There are only two non-trivial, non-isomorphic FOCC and both are not only elementary, but they are also simple (does not have any subbicomodules)\footnote{Any singleton in a simple bicomodule generates it whole.}:
    \begin{eqnarray}
        \Upsilon<[y\otimes x]>&=&\mathrm{span}\{[y\otimes x],[y\otimes u],[u\otimes x],[u\otimes u]\},\\
        \Upsilon<[x\otimes x]>&=&\mathrm{span}\{[x\otimes x],[x\otimes u],[v\otimes x],[v\otimes u]\}.
    \end{eqnarray}
    One can check that $\Upsilon<[x\otimes x]>=\Upsilon<[y\otimes y]>$. The universal calculus  decomposes into:
    \begin{eqnarray}
        \Upsilon_\mathcal{C}^U&=&\Upsilon<[x\otimes x]>\oplus\Upsilon<[y\otimes x]>\oplus[\phi\otimes\phi](\Upsilon<[y\otimes x]>).
    \end{eqnarray}
    We know that the image od exterior coderivation is always coideal. In this case:
    \begin{eqnarray}
        \delta^U({\Upsilon<[y\otimes x]>})=\mathrm{Span}\{x-y,u\},&\qquad&\delta^U({\Upsilon<[x\otimes x]>})=\mathrm{Span}\{u,v\}.
    \end{eqnarray}
    In this example, one can find the coideal $\mathcal{I}=\mathrm{span}\{u\}$ which is not the image of any FOCC, c.f. Remark \ref{coideal}. \footnote{This is no longer true for the 3-dimensional upper-triangular matrix coalgebra, i.e. when we set $v=0$.}
 \end{ex}

 \subsection{Sweedler coalgebra}\label{ex_sweedler_coalgebra}
 We classify all elementary FOCCs for Sweedler coalgebra $H=\mathrm{span}\{1,g,X,Xg\}$:
\begin{eqnarray}
    \Delta 1=1\otimes 1,&\qquad&\Delta g=g\otimes g,\\
    \Delta X=X\otimes 1+g\otimes X,&\qquad&\Delta Xg=Xg\otimes g+1\otimes Xg
\end{eqnarray}
with $\varepsilon1=\varepsilon g=1$, $\varepsilon X=\varepsilon Xg=0$\footnote{It is in fact Hopf algebra, however in this section we are interested only in its coalgebraic structure. We will come back to this example as Hopf algebra later on.}.

There exists a coalgebra (involutive) automorphism $\phi$:
\begin{eqnarray}
    \left[\begin{array}{c}1\\X\end{array}\right]&\leftrightarrow\left[\begin{array}{c}g\\Xg\end{array}\right]
\end{eqnarray}
which will be used to reduce the number of solutions. Below, we present all the elementary FOCCs:

\begin{enumerate}
    \item Infinite family of 1-dimensional cocalculi parametrized by   $[\alpha:\beta]\in \mathbb{P}(\mathbb{K}^2)$
    \begin{eqnarray}
        \Upsilon<\alpha[g\otimes 1]+\beta[X\otimes 1]>&=&\mathrm{Span}\{\alpha[g\otimes 1]+\beta[X\otimes 1]\}.
    \end{eqnarray}
    \item Only two 2-dimensional  FOCCs
    \begin{eqnarray}
        \Upsilon<[1\otimes X]>&=&\mathrm{Span}\{[1\otimes X],[1\otimes g]\},\\
        \Upsilon<[Xg\otimes 1]>&=&\mathrm{Span}\{[Xg\otimes 1],[g\otimes 1]\}.
    \end{eqnarray}
    \item Two families of 3-dimensional cocalculi, $\gamma\neq 0, (\alpha,\beta)\in\mathbb{K}^2$
    \begin{eqnarray}
         \Upsilon<[1\otimes X]+\frac{1}{\gamma}[Xg\otimes 1]>
        =\mathrm{Span}\{[1\otimes X]+\frac{1}{\gamma}[Xg\otimes 1],[1\otimes g],[g\otimes 1]\},\end{eqnarray}
    \begin{eqnarray}
       &&\Upsilon<[Xg\otimes X]+\alpha[1\otimes X]+\beta[Xg\otimes 1]>=\nonumber\\
        &&\qquad\mathrm{Span}\{[Xg\otimes X]+\alpha[1\otimes X]+\beta[Xg\otimes 1],\nonumber\\
        &&\qquad\phantom{=span\{}[1\otimes Xg]+\alpha[1\otimes g],[X\otimes 1]+\beta[g\otimes 1]\}.
    \end{eqnarray}
    \item Two-parameter family of 4-dimension  FOCCs $(\alpha,\beta)\in\mathbb{K}^2$
    \begin{eqnarray}
        &&\Upsilon<[X\otimes X]+\alpha[g\otimes 1]+\beta[X\otimes 1]>=\nonumber\\
        &&\qquad\mathrm{Span}\{[X\otimes X]+\alpha[g\otimes 1]+\beta[X\otimes 1],[X\otimes g],[1\otimes X],[1\otimes g]\}.
    \end{eqnarray}
\end{enumerate}
We note that FODCs on more general Radford Hopf algebras have recently been studied in \cite{Weber2025}.

 \subsection{Coalgebra generated by a vector space}
 
 For given vector space $\mathbf{V}$, the space $\mathcal{C}_\mathbf{V}:=\mathbb{K}\oplus\mathbf{V}$ has well defined cocommutative coalgebraic structure with one group-like element $1$ and every vector $v\in\mathbf{V}$ having a primitive coproduct.
 The universal bicoalgebra decomposes accordingly
 $$\Upsilon^U_{\mathcal{C}_\mathbf{V}}=\pi(\mathbf{V}\otimes\mathbf{V})\oplus \pi(\mathbf{V}\otimes 1)$$
 into the vector subspace $\pi(\mathbf{V}\otimes\mathbf{V})\subset [\ker\varepsilon\otimes\ker\varepsilon]\subset\ker\delta^U$ and the cocommutative subbicomodule $[\mathbf{V}\otimes 1]\subset \Upsilon^{U\natural}_{\mathcal{C}_\mathbf{V}}$.
 Automorphisms of $\mathbf{V}$  are in one-to-one correspondence with automorphisms of $\mathcal{C}_\mathbf{V}$. Therefore, all elementary singletons of the form $<[x\otimes 1]>\,, 0\neq x\in \mathbf{V}$ generate
$$ \Upsilon(<[x\otimes 1]>) =\mathrm{Span}\{[x\otimes 1]\}=\mathrm{Span}\{[1\otimes x]\}\subset\Upsilon^{U\natural}_{\mathcal{C}_\mathbf{V}}$$
isomorphic cocommutative 1-dimensional subbicomodules.

By choosing an unordered basis $\mathbf{V}=\mathrm{Span}\{x_1,x_2,\ldots,x_N\}
$ we reduce the number of automorphisms (for simplicity we assume that the dimension $N$ is finite).

Singletons $\hat \omega=\sum_{i=1,j=1}^{n,m} \omega^{ij}[x_i\otimes x_j]$ from $\pi(\mathbf{V}\otimes\mathbf{V})$ represent bilinear forms on $\mathbf{V}^*$. They generate indecomposable (even elementary) FOCCs
$$  \Upsilon(<\hat \omega>) =\mathrm{Span}\{\hat\omega, [x_i\otimes 1], [v_i\otimes 1]| i=1\ldots n\}= \mathrm{Span}\{\hat\omega, [x_j\otimes 1], [u_j\otimes 1]| j=1\ldots m\}$$
where $v_i=\sum_{j=1}^m\omega^{ij}x_j$,   $u_j=\sum_{i=1}^m \omega^{ij}x_i$ (some spanning vectors $v_i$ or $u_j$ might be linearly dependent). If $\hat\omega$ is symmetric then $m=n$ and $v_i=u_i$.
If it is antisymmetric then $m=n$ and $v_i=-u_i$ and, moreover,  $\Upsilon<\hat \omega>\ \subset \Upsilon^{U\natural}_{\mathcal{C}_\mathbf{V}}$. One can check, that for both symmetric and antisymmetric elements $\hat\omega$ we have $\dim\Upsilon<\hat\omega>=1+r$, where $r$ is a rank of the bilinear form $\omega$.

 \subsection{Divided power coalgebra}
 
Consider graded coalgebra generated by one primitive element $X$
 \begin{eqnarray}
 	\mathcal{C}&:=&\mathrm{Span}\{X^i;0\leq i\in \mathbb{N}\},
 \end{eqnarray}
 where ($X^0\equiv 1, X^1\equiv X$)\footnote{After the transformation $X^n\to \frac{1}{n!}X^n$ we would obtaint coproduct in a form $\Delta(X^n)=\sum_{i=0}^n\binom{n}{i}[X^i\otimes X^{n-i}]$. For simplicity, we used a basis without binomial coefficients.}:
 \begin{eqnarray}
 	\Delta(X^n)=\sum_{i=0}^nX^i\otimes X^{n-i}\,, \quad \varepsilon(x^n)=\delta^{0n}.
 \end{eqnarray}
 This is a Hopf algebra, but for now, we focus only on coalgebraic structure. This coalgebra does not have any automorphisms. 
 
 There is a coalgebra filtration $\mathcal{C}^{0}\subset\ldots\subset\mathcal{C}^{n-1}\subset\mathcal{C}^n \subset \ldots$ 
 of $\mathcal{C}$ by finite-dimensional subcoalgebras:
 \begin{eqnarray}
 	\mathcal{C}^n&:=&\mathrm{Span}\{X^i;0\leq i\leq n\}\,,\quad \dim\mathcal{C}^n=n+1\,.
 \end{eqnarray}
It induces the universal bicomodule filtration by finite-dimensional subbicomodules, which turns out to be elementary FOCCs generated by singletons:
\begin{eqnarray}
 	\Upsilon^{U}_{\mathcal{C}^n}=\pi(\mathcal{C}^n\otimes\mathcal{C}^n)
=\Upsilon<[X^n\otimes X^n]>=\mathrm{Span}\{[1\otimes X^i], [X^i\otimes X^j];1\leq i, j\leq n\}\,,
 \end{eqnarray}
 Notice that the spanning vectors are linearly independent. 
 In fact, all FOCCs   (c.f. example~\ref{ex_C1otimesC2/()}), $n+m\geq 1$:
 \begin{equation}
     \Upsilon<[X^n\otimes X^m]>=\pi(\mathcal{C}^n\otimes\mathcal{C}^m)\, 
 \end{equation}
 are non-isomorphic (except $\Upsilon<[X\otimes 1]>\cong\Upsilon<[1\otimes X]>$) and elementary. We can also calculate the dimension:
 \begin{eqnarray}
 	\dim\Upsilon<[X^n\otimes X^m]>&=&nm+\max(n,m).
 \end{eqnarray}
 Since this coalgebra is cocommutative, we are also interested in cocommutative FOCCs.
 \begin{pr}
    For any $m\geq 1$ there is a family of finite-dimensional cocommutative FOCCs for the coalgebra $\mathcal{C}^m$, $m\geq n\geq 1$, in the form:
     \begin{eqnarray}
         \Upsilon<\upsilon^n>=\mathrm{Span}\{\upsilon^k;1\leq k\leq n\}&\subseteq&\Upsilon^{U\natural}_{\mathcal{C}^m}\,,
     \end{eqnarray}
     where the vectors:
     \begin{eqnarray}
         \upsilon^k&=&\sum_{i=0}^{k-1}\left(1-\frac{i}{k}\right)[X^{k-i}\otimes X^i].
     \end{eqnarray}
do satisfy the antisymmetry property $\upsilon^k+ \tau(\upsilon^k)=[\Delta(X^k)]\equiv 0$.
     \begin{proof}
     The last formula determining antisymmetry can be checked by a direct calculation.
     The poof of the bicomodule property is more involved.
     We use formulas from Appendix E showing that for any $n\geq 1$:
         \begin{eqnarray}
             \Delta^U_L(\upsilon^n)&=&1\otimes\upsilon^n +\sum_{j=1}^{n-1}\left(1-\frac{j}{n}\right)X^j\otimes\upsilon^{n-j}\,,\label{E1}\\
             \Delta^U_R(\upsilon^n)&=&\upsilon^n\otimes 1+\sum_{j=1}^{n-1}\left(1-\frac{j}{n}\right)\upsilon^{n-j}\otimes X^j\,.\label{E2}
         \end{eqnarray}
         They show that  $\Upsilon<\upsilon^n>$ is closed under the left and right comultiplications, and consists of cocommutative elements.
     \end{proof}
 \end{pr}
 Our calculation suggests that in fact $\Upsilon^{U\natural}_{\mathcal{C}^n}=\mathrm{span}\{\upsilon^i;1\leq i\leq n\}$, however, we did not prove it. Again, one gets natural filtration of the family of cocommutative FOCCs:
 \begin{eqnarray} \Upsilon<\upsilon^1>\subset\Upsilon<\upsilon^2>\subset\cdots\subset\Upsilon<\upsilon^n>\subset\cdots\,.
 \end{eqnarray}
 Therefore, one also finds infinite-dimensional cocommutative FOCC:
 \begin{eqnarray}
     \Upsilon_\infty&:=&+_{n\in\mathbb{N}_+}\Upsilon<\upsilon^n>=\mathrm{Span}\{\upsilon^i;1\leq i\}\,.
 \end{eqnarray}

  \subsection{Set coalgebra}\label{ex_P_cA1}
 Set coalgebra is $\mathcal{C}=\mathbb{K}(\mathcal{O})=\bigoplus_{p\in\mathcal{O}}\mathcal{C}_p$, where $\mathcal{C}_p\cong\mathbb{K}$ are one-dimensional subcoalgebras  and the set $\mathcal{O}$ constitutes a basis for $\mathcal{C}$. To classify FOCCs for this example, we use lemma~\ref{pr_UPU(C1otimesC2)}. 
 As a result, one gets decomposition into one-dimensional subbicomodules:\footnote{ Although the set coalgebra is cocommutative,  it does not admit cocommutative FOCCs since $\Upsilon^{U\natural}_{\mathcal{C}_p}=0$.}
 \begin{eqnarray}
 	\Upsilon^U_{\mathbb{K}(\mathcal{O})}&=&\bigoplus_{\substack{p,q \in\mathcal{C}\\p\neq q}}(\mathcal{C}_p\otimes\mathcal{C}_q)=\bigoplus_{\substack{p,q \in\mathcal{C}\\p\neq q}}\Upsilon<[p\otimes q]>=\bigoplus_{\substack{p,q \in\mathcal{C}\\p\neq q}}<[p\otimes q]>  \,,
 \end{eqnarray}
 with $\delta^U([p\otimes q])=p-q$. Applying this decomposition, one can associate to any subset $\mathcal{I}\subset\{(p,q)\in\mathcal{O}\times\mathcal{O};p\neq q\}$ the corresponding FOCC:  
 \begin{eqnarray}\label{set_singleton}
 	\Upsilon_\mathcal{I}&=&\bigoplus_{(p,q)\in\mathcal{I}}<[p\otimes q]>
 \end{eqnarray} 
  decomposes into elementary singletons $<[p\otimes q]>$. In particular, for any non-elementary singleton (obviously with a finite number of non-zero coefficients $\alpha(p,q)$)
 \begin{eqnarray}\label{set_singlet}
 	\Upsilon\left(<\sum_{(p,q)\in\mathcal{I}}\alpha(p,q)[p\otimes q]>\right)&=&\bigoplus_{\alpha(p,q)\neq 0}<[p\otimes q]>\ \subset \Upsilon_\mathcal{I}
 \end{eqnarray}
 the dimension is equal to the number of non-zero coefficients, regardless of the values of $\alpha(p,q)$. Moreover, any finite-dimensional FOCC has this form. Here, we are interested in the classification of all such calculi.
 \begin{obs}
 	Every bijection $\phi$ on the generating set $\mathcal{O}$ prolongs to an automorphism of $ \mathbb{K}(\mathcal{O})$ and provides isomorphic calculi $[p\otimes q]\mapsto [\phi(p)\otimes\phi(q)]$. In particular, isomorphic calculi need to have the same dimension, however it is a not sufficient requirement. For example, the following two-dimensional FOCC are not isomorphic 
 	\begin{eqnarray}
 		\Upsilon_1&=&\Upsilon<[p\otimes q],[r\otimes s]>=\mathrm{Span}\{[p\otimes q],[r\otimes s]\}\,,\nonumber\\
 		\Upsilon_2&=&\Upsilon<[p\otimes q],[q\otimes p]>=\mathrm{Span}\{[p\otimes q],[q\otimes p]\}\,.\nonumber
 	\end{eqnarray}
 \end{obs}
 We are now in a position to provide a graphical method for the classification.  For any subbicomodule 
 we associate a directed graph $G(\mathcal{O},E)$, with $\mathcal{O}$ and $E$ being the set of vertices and edges, respectively, see e.g., \cite{Harary}. It means that we have a source map $s:~E\to \mathcal{O}$ and a target map $t:~E\to \mathcal{O}$\footnote{Source map returns the starting point of an edge and target map returns its ending point.}. 
  Every singleton in the subbicomodule produces an edge:
 \begin{eqnarray}
     (<[p\otimes q]>\subseteq\Upsilon)&\longleftrightarrow\left(\begin{tikzpicture}[baseline=(o)]
        \node[] at (0,-0.1) (o) {};
        \node[circle,draw,fill=black,scale=0.5] at (0,0) (q) {};
        \node[] at (-0.2,0.2) {$q$};
        \node[circle,draw,fill=black,scale=0.5] at (1,0) (p) {};
        \node[] at (1.2,0.2) {$p$};
        \draw[->] (q) to (p);
    \end{tikzpicture}\subseteq G(\mathcal{O},E) \right)\,
 \end{eqnarray}
 and represents a direct sum component.
 Self-loops (edges with $s(e)=t(e)$) are not allowed, and there can be at most one edge in a given direction between two points (i.e., two edges in the opposite directions are allowed).
 This leads to the following theorem:
 \begin{Th}\label{th_set_iso}
 	Let $\mathcal{C}$ to be set coalgebra. Two finite-dimensional FOCCs are isomorphic if, and only if they generate isomorphic graphs.
 \end{Th}
 We show how it works, classifying low-dimensional FOCCs. We count weak graphs separately\footnote{Directed graph is weak if it is connected as non-directed graph (\cite[chapter 16]{Harary}).}. 
There is only $N=1$ non-isomorphic one-dimensional FOCC represented by graph:
 
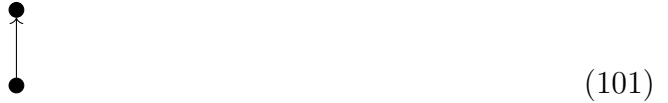
\begin{eqnarray}
    \begin{tikzpicture}[baseline=(o)]
        \node[] at (0,-0.1) (o) {};
        \node[circle,draw,fill=black,scale=0.5] at (0,0) (p) {};
        \node[circle,draw,fill=black,scale=0.5] at (0,1) (q) {};
        \draw[->] (p) to (q);
    \end{tikzpicture}
 \end{eqnarray}
This means that all one-dimensional FOCCs are isomorphic.
 For two-dimensional FOCCs one finds $N=1+4$ different graphs:
 
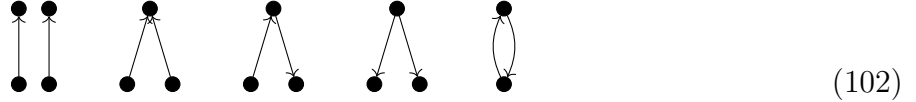
\begin{eqnarray}
    \begin{tikzpicture}[baseline=(o)]
        \node[] at (0,-0.1) (o) {};
        \node[circle,draw,fill=black,scale=0.5] at (-0.2,0) (p1) {};
        \node[circle,draw,fill=black,scale=0.5] at (-0.2,1) (q1) {};
        \node[circle,draw,fill=black,scale=0.5] at (0.2,0) (p2) {};
        \node[circle,draw,fill=black,scale=0.5] at (0.2,1) (q2) {};
        
        \draw[->] (p1) to (q1);
        \draw[->] (p2) to (q2);
    \end{tikzpicture}\qquad
    \begin{tikzpicture}[baseline=(o)]
        \node[] at (0,-0.1) (o) {};
        \node[circle,draw,fill=black,scale=0.5] at (-0.3,0) (p1) {};
        \node[circle,draw,fill=black,scale=0.5] at (0.3,0) (p2) {};
        \node[circle,draw,fill=black,scale=0.5] at (0,1) (q) {};
        
        \draw[->] (p1) to (q);
        \draw[->] (p2) to (q);
    \end{tikzpicture}\qquad
    \begin{tikzpicture}[baseline=(o)]
        \node[] at (0,-0.1) (o) {};
        \node[circle,draw,fill=black,scale=0.5] at (-0.3,0) (p1) {};
        \node[circle,draw,fill=black,scale=0.5] at (0.3,0) (p2) {};
        \node[circle,draw,fill=black,scale=0.5] at (0,1) (q) {};
        
        \draw[->] (p1) to (q);
        \draw[<-] (p2) to (q);
    \end{tikzpicture}\qquad
    \begin{tikzpicture}[baseline=(o)]
        \node[] at (0,-0.1) (o) {};
        \node[circle,draw,fill=black,scale=0.5] at (-0.3,0) (p1) {};
        \node[circle,draw,fill=black,scale=0.5] at (0.3,0) (p2) {};
        \node[circle,draw,fill=black,scale=0.5] at (0,1) (q) {};
        
        \draw[<-] (p1) to (q);
        \draw[<-] (p2) to (q);
    \end{tikzpicture}\qquad
    \begin{tikzpicture}[baseline=(o)]
        \node[] at (0,-0.1) (o) {};
        \node[circle,draw,fill=black,scale=0.5] at (0,0) (p) {};
        \node[circle,draw,fill=black,scale=0.5] at (0,1) (q) {};
        
        \draw[->,bend left=25] (p) to (q);
        \draw[<-,bend right=25] (p) to (q);
    \end{tikzpicture}
 \end{eqnarray}
There are $N=5+12$ non-isomorphic graphs for three-dimensional FOCCs:
 \begin{eqnarray}
    &
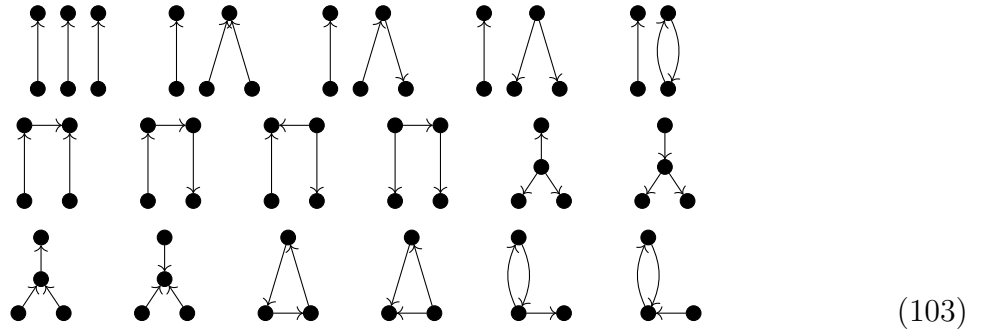
\begin{tikzpicture}[baseline=(o)]
        \node[] at (0,-0.1) (o) {};
        \node[circle,draw,fill=black,scale=0.5] at (-0.4,0) (p1) {};
        \node[circle,draw,fill=black,scale=0.5] at (-0.4,1) (q1) {};
        \node[circle,draw,fill=black,scale=0.5] at (0,0) (p2) {};
        \node[circle,draw,fill=black,scale=0.5] at (0,1) (q2) {};
        \node[circle,draw,fill=black,scale=0.5] at (0.4,0) (p3) {};
        \node[circle,draw,fill=black,scale=0.5] at (0.4,1) (q3) {};
        
        \draw[->] (p1) to (q1);
        \draw[->] (p2) to (q2);
        \draw[->] (p3) to (q3);
    \end{tikzpicture}\qquad\begin{tikzpicture}[baseline=(o)]
        \node[] at (0,-0.1) (o) {};
        \node[circle,draw,fill=black,scale=0.5] at (-0.7,0) (p0) {};
        \node[circle,draw,fill=black,scale=0.5] at (-0.7,1) (q0) {};
        \node[circle,draw,fill=black,scale=0.5] at (-0.3,0) (p1) {};
        \node[circle,draw,fill=black,scale=0.5] at (0.3,0) (p2) {};
        \node[circle,draw,fill=black,scale=0.5] at (0,1) (q) {};

        \draw[->] (p0) to (q0);
        \draw[->] (p1) to (q);
        \draw[->] (p2) to (q);
    \end{tikzpicture}\qquad
    \begin{tikzpicture}[baseline=(o)]
        \node[] at (0,-0.1) (o) {};
        \node[circle,draw,fill=black,scale=0.5] at (-0.7,0) (p0) {};
        \node[circle,draw,fill=black,scale=0.5] at (-0.7,1) (q0) {};
        \node[circle,draw,fill=black,scale=0.5] at (-0.3,0) (p1) {};
        \node[circle,draw,fill=black,scale=0.5] at (0.3,0) (p2) {};
        \node[circle,draw,fill=black,scale=0.5] at (0,1) (q) {};

        \draw[->] (p0) to (q0);
        \draw[->] (p1) to (q);
        \draw[<-] (p2) to (q);
    \end{tikzpicture}\qquad
    \begin{tikzpicture}[baseline=(o)]
        \node[] at (0,-0.1) (o) {};
        \node[circle,draw,fill=black,scale=0.5] at (-0.7,0) (p0) {};
        \node[circle,draw,fill=black,scale=0.5] at (-0.7,1) (q0) {};
        \node[circle,draw,fill=black,scale=0.5] at (-0.3,0) (p1) {};
        \node[circle,draw,fill=black,scale=0.5] at (0.3,0) (p2) {};
        \node[circle,draw,fill=black,scale=0.5] at (0,1) (q) {};

        \draw[->] (p0) to (q0);
        \draw[<-] (p1) to (q);
        \draw[<-] (p2) to (q);
    \end{tikzpicture}\qquad
    \begin{tikzpicture}[baseline=(o)]
        \node[] at (0,-0.1) (o) {};
        \node[circle,draw,fill=black,scale=0.5] at (-0.4,0) (p0) {};
        \node[circle,draw,fill=black,scale=0.5] at (-0.4,1) (q0) {};
        \node[circle,draw,fill=black,scale=0.5] at (0,0) (p) {};
        \node[circle,draw,fill=black,scale=0.5] at (0,1) (q) {};

        \draw[->] (p0) to (q0);
        \draw[->,bend left=25] (p) to (q);
        \draw[<-,bend right=25] (p) to (q);
    \end{tikzpicture}&\nonumber\\
    &\begin{tikzpicture}[baseline=(o)]
        \node[] at (0,-0.1) (o) {};
        \node[circle,draw,fill=black,scale=0.5] at (-0.3,0) (p1) {};
        \node[circle,draw,fill=black,scale=0.5] at (0.3,0) (p2) {};
        \node[circle,draw,fill=black,scale=0.5] at (-0.3,1) (q1) {};
        \node[circle,draw,fill=black,scale=0.5] at (0.3,1) (q2) {};
        
        \draw[->] (p1) to (q1);
        \draw[->] (p2) to (q2);
        \draw[->] (q1) to (q2);
    \end{tikzpicture}\qquad
    \begin{tikzpicture}[baseline=(o)]
        \node[] at (0,-0.1) (o) {};
        \node[circle,draw,fill=black,scale=0.5] at (-0.3,0) (p1) {};
        \node[circle,draw,fill=black,scale=0.5] at (0.3,0) (p2) {};
        \node[circle,draw,fill=black,scale=0.5] at (-0.3,1) (q1) {};
        \node[circle,draw,fill=black,scale=0.5] at (0.3,1) (q2) {};
        
        \draw[->] (p1) to (q1);
        \draw[<-] (p2) to (q2);
        \draw[->] (q1) to (q2);
    \end{tikzpicture}\qquad
    \begin{tikzpicture}[baseline=(o)]
        \node[] at (0,-0.1) (o) {};
        \node[circle,draw,fill=black,scale=0.5] at (-0.3,0) (p1) {};
        \node[circle,draw,fill=black,scale=0.5] at (0.3,0) (p2) {};
        \node[circle,draw,fill=black,scale=0.5] at (-0.3,1) (q1) {};
        \node[circle,draw,fill=black,scale=0.5] at (0.3,1) (q2) {};
        
        \draw[->] (p1) to (q1);
        \draw[<-] (p2) to (q2);
        \draw[<-] (q1) to (q2);
    \end{tikzpicture}\qquad
    \begin{tikzpicture}[baseline=(o)]
        \node[] at (0,-0.1) (o) {};
        \node[circle,draw,fill=black,scale=0.5] at (-0.3,0) (p1) {};
        \node[circle,draw,fill=black,scale=0.5] at (0.3,0) (p2) {};
        \node[circle,draw,fill=black,scale=0.5] at (-0.3,1) (q1) {};
        \node[circle,draw,fill=black,scale=0.5] at (0.3,1) (q2) {};
        
        \draw[<-] (p1) to (q1);
        \draw[<-] (p2) to (q2);
        \draw[->] (q1) to (q2);
    \end{tikzpicture}\qquad
    \begin{tikzpicture}[baseline=(o)]
        \node[] at (0,-0.1) (o) {};
        \node[circle,draw,fill=black,scale=0.5] at (-0.3,0) (p1) {};
        \node[circle,draw,fill=black,scale=0.5] at (0.3,0) (p2) {};
        \node[circle,draw,fill=black,scale=0.5] at (0,0.45) (r) {};
        \node[circle,draw,fill=black,scale=0.5] at (0,1) (q) {};
        
        \draw[<-] (p1) to (r);
        \draw[<-] (p2) to (r);
        \draw[->] (r) to (q);
    \end{tikzpicture}\qquad
    \begin{tikzpicture}[baseline=(o)]
        \node[] at (0,-0.1) (o) {};
        \node[circle,draw,fill=black,scale=0.5] at (-0.3,0) (p1) {};
        \node[circle,draw,fill=black,scale=0.5] at (0.3,0) (p2) {};
        \node[circle,draw,fill=black,scale=0.5] at (0,0.45) (r) {};
        \node[circle,draw,fill=black,scale=0.5] at (0,1) (q) {};
        
        \draw[<-] (p1) to (r);
        \draw[<-] (p2) to (r);
        \draw[<-] (r) to (q);
    \end{tikzpicture}&\nonumber\\
    &\begin{tikzpicture}[baseline=(o)]
        \node[] at (0,-0.1) (o) {};
        \node[circle,draw,fill=black,scale=0.5] at (-0.3,0) (p1) {};
        \node[circle,draw,fill=black,scale=0.5] at (0.3,0) (p2) {};
        \node[circle,draw,fill=black,scale=0.5] at (0,0.45) (r) {};
        \node[circle,draw,fill=black,scale=0.5] at (0,1) (q) {};
        
        \draw[->] (p1) to (r);
        \draw[->] (p2) to (r);
        \draw[->] (r) to (q);
    \end{tikzpicture}\qquad
    \begin{tikzpicture}[baseline=(o)]
        \node[] at (0,-0.1) (o) {};
        \node[circle,draw,fill=black,scale=0.5] at (-0.3,0) (p1) {};
        \node[circle,draw,fill=black,scale=0.5] at (0.3,0) (p2) {};
        \node[circle,draw,fill=black,scale=0.5] at (0,0.45) (r) {};
        \node[circle,draw,fill=black,scale=0.5] at (0,1) (q) {};
        
        \draw[->] (p1) to (r);
        \draw[->] (p2) to (r);
        \draw[<-] (r) to (q);
    \end{tikzpicture}\qquad
    \begin{tikzpicture}[baseline=(o)]
        \node[] at (0,-0.1) (o) {};
        \node[circle,draw,fill=black,scale=0.5] at (-0.3,0) (p1) {};
        \node[circle,draw,fill=black,scale=0.5] at (0.3,0) (p2) {};
        \node[circle,draw,fill=black,scale=0.5] at (0,1) (q) {};
        
        \draw[<-] (p1) to (q);
        \draw[->] (p2) to (q);
        \draw[->] (p1) to (p2);
    \end{tikzpicture}\qquad
    \begin{tikzpicture}[baseline=(o)]
        \node[] at (0,-0.1) (o) {};
        \node[circle,draw,fill=black,scale=0.5] at (-0.3,0) (p1) {};
        \node[circle,draw,fill=black,scale=0.5] at (0.3,0) (p2) {};
        \node[circle,draw,fill=black,scale=0.5] at (0,1) (q) {};
        
        \draw[<-] (p1) to (q);
        \draw[->] (p2) to (q);
        \draw[<-] (p1) to (p2);
    \end{tikzpicture}\qquad
    \begin{tikzpicture}[baseline=(o)]
        \node[] at (0,-0.1) (o) {};
        \node[circle,draw,fill=black,scale=0.5] at (0,0) (p1) {};
        \node[circle,draw,fill=black,scale=0.5] at (0.6,0) (p2) {};
        \node[circle,draw,fill=black,scale=0.5] at (0,1) (q) {};
        
        \draw[->,bend left=25] (p1) to (q);
        \draw[<-,bend right=25] (p1) to (q);
        \draw[->] (p1) to (p2);
    \end{tikzpicture}\qquad
    \begin{tikzpicture}[baseline=(o)]
        \node[] at (0,-0.1) (o) {};
        \node[circle,draw,fill=black,scale=0.5] at (0,0) (p1) {};
        \node[circle,draw,fill=black,scale=0.5] at (0.6,0) (p2) {};
        \node[circle,draw,fill=black,scale=0.5] at (0,1) (q) {};
        
        \draw[->,bend left=25] (p1) to (q);
        \draw[<-,bend right=25] (p1) to (q);
        \draw[<-] (p1) to (p2);
    \end{tikzpicture}&
 \end{eqnarray}

\section{Bicovariant FOCC over Hopf algebras}
In this section we will work with Hopf algebra $H=(H, \Delta, \mu=\cdot,\varepsilon, 1, S)$.   
Basic facts and notation concerning Hopf bimodules (a.k.a. bicovariant bimodules \cite{Woronowicz2}) are recalled in Appendix B. Having these in mind, we can define a bicovariant coderivation.  
\begin{defi}\label{def_bCov}\footnote{We can also define a left covariant bicomodule as an object $(M,\triangleright,\Delta_L,\Delta_R)\in{}^H_H\mathfrak{M}^H$.}
    Let $\Upsilon\in {}^H_H\mathfrak{M}^H_H$ be a bicovariant bimodule over a Hopf algebra $H$. Then a coderivation $(\Upsilon,\delta)$ over $H$ is called bicovariant if 
    \begin{eqnarray}\label{bicovariant-coderivation}
        \delta(a\triangleright m\triangleleft b)&=&a\,\delta(m)\,b.
    \end{eqnarray}
 If $(\Upsilon,\delta)$ is a FOCC, then we will call it  bicovariant FOCC.
\end{defi}
 
We are now in position to formulate the following
\begin{Th}\label{bicovUniversal}
    Let $H$ be a Hopf algebra. The universal FOCC $(\Upsilon^U_H,\delta^U)$ is bicovariant provided we define a bimodule structure as
     \begin{eqnarray}\label{triangle-actions}
    	x\triangleright[a\otimes b]\doteq [x_{(1)}a\otimes x_{(2)}b]\,,\quad
    	{}[a\otimes b]\triangleleft x\doteq [ax_{(1)}\otimes bx_{(2)}]\,.
    \end{eqnarray}
    Moreover, one has Hopf bimodule isomorphisms 
 \begin{eqnarray}
        \Upsilon_H^U\cong \overline{H}_L\otimes H\cong H\otimes\overline{H}_R,
\end{eqnarray}
    where  $\overline{H}_L=\{\overline{H}, \blacktriangleright, \Xi_L \}\in {}^H_H\mathfrak{YD}$
    \begin{eqnarray} \label{LYD}
    	\Xi_L(\overline{a})\doteq a_{(1)}\otimes\overline{a_{(2)}}\,,\quad
    	x\blacktriangleright\overline{a}\doteq \mathrm{ad}^L_x a=\overline{x_{(1)}aS(x_{(2)})}\,,
    \end{eqnarray}
    and  $\overline{H}_R=\{\overline{H}, \blacktriangleleft, \Xi_R \}\in \mathfrak{YD}^H_H$
     \begin{eqnarray}\label{RYD}
    	\Xi_R(\overline{a})\doteq\overline{a_{(1)}}\otimes a_{(2)}\,,\quad
    	\overline{a}\blacktriangleleft x\doteq\mathrm{ad}^R_x a=\overline{S(x_{(1)})ax_{(2)}}\,.
    \end{eqnarray}
 Here, the vector space $\overline{H}=H/(1\cdot\mathbb{K})$ inherits two factor Yetter-Drinfeld (Y-D) structures, $\overline{H}_L$ and $\overline{H}_R$, from the corresponding Y-D structures on the Hopf algebra $H$ itself, see Appendix B. We  denote by $\overline{a}=\overline{a+\lambda 1}\in \overline{H}$, $\lambda\in \mathbb{K}$, the image of $a\in H$ by the canonical projection $H\rightarrow \overline{H}$.
    \begin{proof}
        Firstly, it is easy to check by direct calculations that the actions \eqref{triangle-actions} and coactions $(\Delta^U_L, \Delta^U_R)$, cf. \eqref{DeltaU}, do satisfy the conditions \eqref{b1}-\eqref{b2} in Appendix B. This means that $\Upsilon^U_H$ is a bicovariant bimodule. Also the condition \eqref{bicovariant-coderivation} can be easily verified. 
     	Now
     	\begin{eqnarray}
     		P_R([a\otimes b])&=&[a\otimes b_{(1)}]\triangleleft S(b_{(2)})=[aS(b_{(2)})_{(1)}\otimes b_{(1)}S(b_{(2)})_{(2)}]\nonumber\\
     		&=&[aS(b_{(3)})\otimes b_{(1)}S(b_{(2)})]=[aS(b)\otimes 1],
     	\end{eqnarray}
     	 i.e. $P_R(\Upsilon^U_H)=\mathrm{Span}\{[aS(b)\otimes 1]\}$.\footnote{One sees that $\mathrm{Span}\{ aS(b)| a,b\in H\}\subseteq H$. From the other hand for $H\ni c=c_{(1)}c_{(2)}S(c_{(3)})\in \mathrm{Span}\{ aS(b)| a,b\in H\}$, i.e., equality holds.} Moreover, we have linear isomorphism  between subspaces $\overline{H}\otimes 1\ni \overline{a}\otimes 1\mapsto \Phi_R(\overline{a}\otimes 1)= [a\otimes 1]=[(a+\lambda 1)\otimes 1]\in P_R(\Upsilon^U_H)$. It is not difficult to check that this isomorphism can be used to induce left-left Y-D structure \eqref{LYD} on $\overline{H}_L\equiv \overline{H}\otimes 1$.
       	By similar arguments we can obtain  an isomorpism 
       	\begin{equation}\label{PsiR}
       		\Phi_L: \overline{H}_R\equiv 1\otimes\overline{H}\rightarrow P_L(\Upsilon^U_H)\subset \Upsilon^U_H\quad \mbox{by}\quad \Phi_L(1\otimes\overline{a})=[1\otimes a]
       	\end{equation}
        and the right-right Y-D structure $\overline{H}_R\in \mathfrak{YD}^H_H$ \eqref{RYD} induced from $P_R(\Upsilon^U_H)$:
        \begin{eqnarray}\label{RYDb}
        	\Xi_R([1\otimes a])\doteq [1\otimes a_{(1)}]\otimes a_{(2)}\,,\quad
        	[1\otimes a]\blacktriangleleft x\doteq [1\otimes S(x_{(1)})ax_{(2)}]\,.
        \end{eqnarray}
         This ends the proof.  
    \end{proof}
\end{Th}

\begin{rrm}\label{ker-bar}
	There is a canonical vector space isomorphism $\overline H\ni\bar a\leftrightarrow \delta^U([a\otimes 1])\doteq \delta(\bar a)=a-\varepsilon(a)1\in\ker\varepsilon_H$ between $\bar H$ and $\ker\varepsilon_H$. Therefore, all structures arranged on $\overline H$ can be translated to those on $\ker\varepsilon_H$ and vice versa. For example,  $(\overline{H}, \Xi_L,\Xi_R)\in {}^{H}\mathfrak{M}^{H}$ is a factor bicomodule over $H$. The map $\Phi_L$ (resp. $\Phi_R$) gives only right (resp. left) comodule embedding into    $\Upsilon^U_H$. Therefore, $\bar a\mapsto \delta(\bar a)$ is not a coderivation of $(\overline{H}, \Xi_L,\Xi_R)$.

\end{rrm}


\begin{rrm}\label{PhiR} One finds that  Hopf bimodule isomorpfism $\Phi_R: \overline{H}_L\otimes H\to\Upsilon_H^U$  and its inverse (see  Appendix B) can be written as 
    \begin{eqnarray}
        \Phi_R(\overline{a}\otimes b)&=& [a\otimes 1]\triangleleft b=[r(a\otimes b)]\,,\quad  \Phi_R(\overline{a}\otimes bc)= \Phi_R(\overline{a}\otimes b)\triangleleft c\,,\nonumber\\
          	\Phi^{-1}_R([a\otimes b])&=&\overline{aS(b_{(1)})}\otimes b_{(2)} \,,\quad\mbox{and}\quad  \mathrm{Im} (\mathrm{id}\otimes \varepsilon)\circ\Phi^{-1}_R=\overline{H}_L\,,\nonumber
    \end{eqnarray}
    where $r(a\otimes b)\doteq ab_{(1)}\otimes b_{(2)}$ is invertible mapping introduced by Woronowicz \cite{Woronowicz}. The corresponding Hopf bimodule structure on $\overline{H}_L\otimes H$ making the map $\Phi_R $ a Hopf bimodule morphism is as follows
     \begin{eqnarray}
    	c\triangleright(\overline{a}\otimes b)\triangleleft d&=&  	c_{(1)}\blacktriangleright\overline{a}\otimes 	c_{(2)}bd\,,\nonumber\\
     \Delta_R (\overline{a}\otimes b)=(\overline{a}\otimes b_{(1)})\otimes b_{(2)} \,,&\mbox{and}&  \Delta_L (\overline{a}\otimes b)=  a_{(1)}b_{(1)}\otimes (\overline{a_{(2)}}\otimes b_{(2)}) \,.\nonumber
    \end{eqnarray}
    We also find $$\delta(\overline{a}\otimes b)\doteq \delta^U(\Phi_R(\overline{a}\otimes b))= (a-\varepsilon(a)1)b\,.$$
    Similary, for  $\Phi_L: H\otimes\overline{H}_R\to\Upsilon^U$  one gets
     \begin{eqnarray}
     	\Phi_L(a\otimes \overline{b})&=&a\triangleright [1\otimes b]= [s'(a\otimes b)]\,,\quad 	\Phi_L(ab\otimes \overline{c})=	a\triangleright\Phi_L(b\otimes \overline{c})\,,\nonumber\\
    	\Phi^{-1}_L([a\otimes  b])&=& a_{(1)}\otimes \overline{S(a_{(2)})b}\,,\quad\mbox{and}\quad \mathrm{Im} ( \varepsilon\otimes\mathrm{id})\circ\Phi^{-1}_L=\overline{H}_R\,.\nonumber
    \end{eqnarray}
   where $s'(a\otimes b)\doteq a_{(1)}\otimes a_{(2)}b$  is another Woronowicz map. 
\end{rrm}
    
    We recall that two other  Woronowicz mappings: $r'(a\otimes b)\doteq a_{(1)}b\otimes a_{(2)}$ and $s(a\otimes b)\doteq b_{(1)}\otimes a b_{(2)}$ in order to be invertible require a bijective (i.e. invertible) antipod. The last condition is always satisfied for  finite-dimensional Hopf algebras. For commutative or cocomutative Hopf algebras, regardless of a dimension,  $S^2=\mathrm{id}$, i.e. $S^{-1}=S$. The map $S^{-1}$ being an antipod, provides a Hopf algebra structure to $H_{op}$, i.e, a bialgebra with the opposite multiplication as well as for $H^{coop}$,  a bialgebra with the opposite comultiplication.\footnote{$S$ is also an antipod for $H^{coop}_{op}$.} Coming back to our case one finds that the map
    $\Phi(m)\doteq S(m_{(-1)})\triangleright m_{<0>}\triangleleft  S(m_{(1)})$ has the inverse 
    $\Phi^{-1}(m)\doteq S^{-1}(m_{(1)})\triangleright m_{<0>}\triangleleft  S^{-1}(m_{(-1)})$ \cite{Klimyk}. It give rise a vector space isomorphism between $P_R(\Upsilon^U_H)$ and $P_L(\Upsilon^U_H)$:
    \begin{equation}
    	\Phi([a\otimes 1])=[1\otimes S(a)]\,,\quad\mbox{and}\quad \Phi([1\otimes a])= [S(a)\otimes 1]\,.
    \end{equation}

\begin{pr}[Classification theorem]\label{ClasBicov}
    Let $H$ be a Hopf algebra. There is one-to-one correspondence between bicovariant FOCC  $\Upsilon$ over $H$ and left-left Y-D submodules $\mathcal{L}\subseteq\overline{H}_L$ (or right-right Y-D submodules $\mathcal{R}\subset\overline{H}_R$). 
    Equivalently, the subspace $\mathcal{L}\subset \overline{H}$ has to satisfy two conditions
    \begin{equation}\label{LLYDcondition}
    	\Xi_L(\mathcal{L})\subseteq H\otimes \mathcal{L}\,,\quad\mbox{and}\quad  H\blacktriangleright \mathcal{L}\subseteq \mathcal{L}\,.
    \end{equation}   
    where $(\Xi_L, \blacktriangleright)$ are defined by \eqref{LYD}.
    Further, we should identified those submodules which are isomorphic by a Hopf algebra automorphism.
\end{pr}
    \begin{proof}
    	$\mathcal{L}$ has to be at the same time subcomodule and submodule of  $\overline{H}_L$, therefore, the left-left Y-D subcomodule condition \eqref{LLYDcondition} is obvious. Since $\Phi_R$ is Hopf bimodule morphism, it is also clear that if $\mathcal{L}$ is Y-D subcomodule of $\overline{H}_L$ then $\Phi_R(\mathcal{L}\otimes H)\in {}^H_H\mathfrak{M}^H_H$ is a bicovariant subbimodule of $\Upsilon^U_H$. 
        Conversely, if  $\Upsilon\subset\Upsilon^U_H$ is a Hopf subbimodule then 
        $\mathcal{L}\doteq  (\mathrm{id}\otimes \varepsilon)\circ\Phi^{-1}_R(\Upsilon)$ is required left-left Y-D submodule in $\overline{H}_L$.
\end{proof}
\begin{rrm}\label{right_covariant}
There is a one-to-one correspondence   between subspaces $\mathcal{L}\subset \overline{H}$  being  left subcomodules, i.e.  satisfying $ \Xi_L(\mathcal{L})\subseteq H\otimes \mathcal{L}$ and their images $\Phi_R(\mathcal{L})\in {}^H\mathfrak{M}^H_H$  being a right covariant bicomodules providing right covariant FOCC: $\delta^U([a\otimes b]\triangleleft c)=\delta^U([a\otimes b])c$. We should notice that subspace $\mathcal{L}$ is a left subcomodule if and only if it is right submodule over the algebra $H^*$ with the action $\bar a\star\alpha= \alpha(a_{(1)})\overline{a_{(2)}}$. For example, any group-like element $g\in H$ (except the unit) generates a one-dimensional right-covariant FOCC on  $<\overline{g}>\otimes H$ with the codifferential $\delta^U([ga_{(1)}\otimes a_{(2)}])=(g-1)a$. Similarly, having a primitive element $x\in H$, one finds a one-dimensional 
right covariant FOCC on $<\overline{x}>\otimes H$ with the codifferential $\delta^U([xa_{(1)}\otimes a_{(2)}])= xa$. In the last case, it is a cocommutative right covariant cocalculus (see below).
\end{rrm}
Finally, we are in a position to analyze the case of cocommutative Hopf algebras.
 \begin{pr}[Cocommutative case, see Subsection 2.2]\label{ClasBicovCocom}
Let $H$ be a cocommutative Hopf algebra. Then
\begin{itemize}
    \item[a)]$H=H^\natural$ and $\overline{H}=\overline{H}^\natural$ as bicomodules.\footnote{It does not imply that its image $\Phi_R(\overline{H}\otimes 1)$ (or  $\Phi_L(1\otimes\overline{H})$ ) shares this property.};
    \item[b)]the subspace $\mathcal{L}\subset \overline{H}$ is a left-left Y-D submodule of $\overline{H}_L$ if and only if it is a right-right Y-D submodule of $\overline{H}_R$;
    \item[c)]left-left Y-D submodule $\mathcal{L}\subset \overline{H}_L$ generates cocomutative bicovariant FOCC if and only if $\Phi_R(\mathcal{L}\otimes H)\subset \Upsilon_H^{U\natural}\subset \Upsilon_H^{U\wedge}$ (c.f. Proposition \ref{natural}).
    \item[d)]every element of left-left Y-D submodule $\overline{x}\in\mathcal{L}\subset\overline{H}$ that generates cocommutative FOCC is primitive as an element of the coalgebra $x\in H$.  
\end{itemize}
    \begin{proof}
    	a) The first property is obvious. The second follows from the fact that $H$ is factorized by the cocommutative subbicomodule $\{ \mathbb{K}1\}$.
    	
    	b) We have $\Xi_L(\mathcal{L})= H\otimes \mathcal{L}$ and $\Xi_R(\mathcal{L})=\tau\circ\Xi_L(\mathcal{L})$, so  $\mathcal{L}$ is a right subbicomodule in $\overline{H}=\overline{H}^\natural$. Next, since $S^{-1}=S$ and $(S\otimes S)\circ\Delta=\Delta\circ S$ one checks
    	$$h\blacktriangleright\overline{l}=\overline{h_{(1)}lS(h_{(2)})}= \overline{S^2(h_{(2)})lS(h_{(1)})}
    	= \overline{S(S(h)_{(1)})lS(h)_{(2)})}= \overline{l}\blacktriangleleft S(h)\\,,$$
    	i.e. $H\blacktriangleright \mathcal{L}=\mathcal{L}\blacktriangleleft S(H) \subset \mathcal{L}$. Since $S$ is bijective the proof is done.
    	
    	c) easily follows from Proposition \ref{ClasBicov}, see also Remark \ref{PsiR}.
    
        d) if $\overline{x}\in\mathcal{L}$ generates bicomodule $\Upsilon$, then $[x\otimes 1]\in\Upsilon$ (we can choose $x\in\mathrm{Ker}\varepsilon$).Thus $\delta^U[x\otimes 1]=x$. Making use of co-Leibniz  rules one gets
        \begin{eqnarray}
            \Delta(x)&=&(\mathrm{id}\otimes\delta^U)\circ\Delta_R^\tau ([x\otimes 1])+(\delta^U\otimes\mathrm{id})\circ\Delta_R([x\otimes 1])=1\otimes x+x\otimes 1\,.
        \end{eqnarray}
    \end{proof}
\end{pr}
\subsection{Universal quantum Lie algebras over $H$}

As we indicate in Appendix C there are two types of left-left (as well right-right) Y-D structures on an arbitrary Hopf algebra. The first one i) with an adjoint action, is related to FOCC and turns out to be also related with Woronowicz idea of quantum tangent space. Therefore, we shell call it a quantum Lie algebra type (qLA-type). In fact, we are going to show that any Y-D submodule in $\overline{H}_L$ can be treated as a left quantum Lie algebra in a sense of abstract definition proposed in \cite[Definition 2.7.1 p. 170]{BeggsMajid}. For this it is enough to show the following (we skip right handed version):
\begin{Th}\label{can_qLA}
	For any Hopf algebra $H$ the Y-D module $\overline{H}_L$ \eqref{LYD}, equipped with the canonical Y-D braiding $\tau_q: \overline{H}_L\otimes\overline{H}_L\rightarrow \overline{H}_L\otimes\overline{H}_L$ \footnote{For $M,N\in  {}^H_H\mathfrak{YD}$ we have the canonical braiding $\tau_{M,N}: M\otimes N\rightarrow N\otimes M$ given by $\tau_{M,N}(m\otimes n)\doteq m_{(-1)}\blacktriangleright n\otimes m_{<0>}$ satisfying the braid relation \eqref{braid}.}
	\begin{equation}\label{qBraid}
		 \tau_q(\overline{X}\otimes\overline{Y})\doteq  \mathrm{ad}^L_{X_{(1)}}\overline{Y}\otimes\overline{X_{(2)}}=
		 \overline{X_{(1)}Y S(X_{(2)})}\otimes\overline{X_{(3)} }
	\end{equation}
	and the (quantum) bracket $[\,,]_q:\overline{H}_L\otimes\overline{H}_L\rightarrow \overline{H}_L $
	\begin{equation}\label{qLie}
		 [\overline{X}, \overline{Y}]_q\doteq \mathrm{ad}^L_{\delta(\overline{X})}\overline{Y} 
		 =\overline{X_{(1)}YS(X_{(2)})}-\varepsilon(X)\overline{Y}
	\end{equation}
bears a left quantum Lie algebras structure, where $\delta(\overline{X})\doteq \delta^U([X\otimes 1])=X-\varepsilon(X)1\in \ker\varepsilon$\footnote{The Y-D structures on $\overline{H}_L$ and $[H\otimes 1]$ are isomorphic, see also Remark \ref{ker-bar}.}. In particular
\begin{equation}\label{qLie2}
	 [\overline{X},\overline{Y}]_q= \overline{\delta(\overline{ X})\delta(\overline{ Y}) - \delta( \overline{X_{(1)}YS(X_{(2)}}))\delta(X_{(3)})}\,,
\end{equation}
can be rewritten as 
\begin{equation}\label{qLie3}
 [\overline{X},\overline{Y}]_q=\overline{\mu\circ(\delta\otimes \delta)\circ(\mathrm{id} -\tau_q)(\overline{X}\otimes\overline{Y})}\,,
\end{equation}
what implies braided anticommutativity:
\begin{equation}\label{qLie0}
	[\,,]_q\quad\mbox{vanishes on}\quad \ker(\mathrm{id}-\tau_q)\,.\nonumber
\end{equation}

Further, three braided Jacobi identities on $\overline{H}_L\otimes\overline{H}_L\otimes \overline{H}_L$
are also satisfied
\begin{eqnarray}
[\,,]_q\circ([\,,]_q\otimes \mathrm{id})&=&	[\,,]_q\circ(\mathrm{id}\otimes [\,,]_q)\circ((\mathrm{id}-\tau_q)\otimes\mathrm{id})\,,\label{qJ1}\\
\tau_q\circ (\mathrm{id}\otimes [\,,]_q)&=&([\,,]_q\otimes\mathrm{id})\circ(\mathrm{id}\otimes\tau_q)\circ(\tau_q\otimes\mathrm{id})\,,\label{qJ2}\\
\tau_q\circ([\,,]_q\otimes\mathrm{id})&-&(\mathrm{id}\otimes[\,,]_q)\circ (\tau_q\otimes\mathrm{id})\circ(\mathrm{id}\otimes\tau_q)\nonumber\\
&=& ([\,,]_q\otimes\mathrm{id})\circ(\mathrm{id}\otimes\tau_q)\circ((\mathrm{id}-\tau_q^2)\otimes\mathrm{id})\,.\label{qJ3}
\end{eqnarray}
\end{Th}
\begin{proof}
	 Firstly, we have to check that all formulas are well defined, i.e., they do not depend on a choice of representing elements $X\in \overline{X}, Y\in\overline{Y}$ and the multiplication is performed in the algebra $H$. To show \eqref{qLie2} one calculates
\begin{eqnarray}
	 \overline{\delta(\overline{X})\delta(\overline{Y})-\delta(\mathrm{ad}^L_{X_{(1)}}\overline{Y})\delta(\overline{X_{(2)}})}
	&=&\overline{\delta(\overline{X})Y-\mathrm{ad}^L_{X_{(1)}}Y\delta(\overline{X_{(2)}})}\nonumber\\
	\overline{\delta(\overline{X})Y-XY+\mathrm{ad}^L_XY}
	&=&\overline{\mathrm{ad}^L_XY-Y\varepsilon(X)}=\mathrm{ad}^L_{\delta(\overline{X})}\overline{Y}\,.
\end{eqnarray}
Braided anticommutativity follows directly from  \eqref{qLie3}. More technical Jacobi identities are postponed to Appendix E. 
\end{proof}
In particular, any Y-D submodule  $\mathcal{L}\subset \overline{H}_L$ becomes automatically left quantum Lie subalgebra over $H$. Therefore, it is justified to call  $\overline{H}_L$ a universal left quantum Lie algebra over~$H$.

\begin{rrm}
  It follows, that the coderivation of the Lie bracket ~\eqref{qLie2} can be written as:
	\begin{eqnarray}
		\delta[\overline{X},\overline{Y}]_q&=&\mu\circ(\delta\otimes\delta)\circ(\mathrm{id}-\tau_q)(\overline{X}\otimes\overline{Y})\nonumber\\
		&=&\delta(\overline{X})\delta(\overline{Y})+X_{(1)}\delta(\overline{Y})S(\delta(\overline{X_{(2)}})).
	\end{eqnarray}
	The sign changes because $S(X_{(1)})\delta(\overline{X_{(2)}})=-S(\delta(\overline{X}))$.
For a right-right Y-D module $\overline{H}_R$ one gets \\
$[\overline{X},\overline{Y}]^R_q\doteq 
\mathrm{ad}^R_{\delta(Y)}\overline{X}= \overline{S(Y_{(1)})X Y_{(2)}}-\varepsilon(Y)\overline{X}$ \\
 $\tau^R_q(\overline{X}\otimes\overline{Y})\doteq \overline{Y_{(1)}}\otimes\mathrm{ad}^R_{Y_{(2)}}\overline{X}=
  \overline{Y_{(1)} }\otimes\overline{S(Y_{(2)})X Y_{(3)}}$
\end{rrm}
 The second Y-D structure indicated in Appendix C ii) introduced in the inspiring   Woronowicz' paper, which we call Woronowicz type (W-type),  found application in a  bicovariant differential calculus. Both types are related by the duality which holds for an arbitrary  Hopf algebra. 
 
 \begin{rrm}
 	Some authors are using $\mathrm{Ker}\varepsilon$ instead of the quotient $\overline{H}=H/{1}$ (see \cite{Klimyk,Woronowicz}). These two objects are isomorphic as vector spaces, and the resulting quantum Lie algebras are also equivalent (see formulas \eqref{DualqBraid} and \eqref{DualqLie} below). We consider the description with $\overline{H}$ as more natural.
 \end{rrm}
 
\subsection{Canonical pairing between universal FOCC and FODC}

We know \cite{Sweedler69,Montgomery_book,BrzezinskiWisbauer,Radford} that for every  Hopf algebra $(H,\mu,\Delta, S, 1,\varepsilon)$ there exists a dual Hopf algebra  
\begin{equation}\label{vf6}
(H^\circ,\mu^\circ=\Delta^*\mid_{H^\circ\otimes H^\circ}=\star,\Delta^\circ=\mu^*\mid_{H^\circ}, S^\circ=S^*\mid_{H^\circ}, 1^\circ=\varepsilon\mid_{H^\circ},\epsilon^\circ=1)\,,	 
\end{equation}
where $H^\circ$ is the biggest subalgebra of $(H^*, \Delta^*=\star, 1\equiv\varepsilon)$ which has well defined coproduct $\Delta^\circ=\mu^*\mid_{H^\circ}: H^\circ\rightarrow H^\circ\otimes H^\circ\subset H^*\otimes H^*\subset (H\otimes H)^*$ (a.k.a restricted dual) \footnote{In general  $\mu^*: H^*\rightarrow (H\otimes H)^*$.}.
 Multiplication $\mu^\circ$ is given by a restriction of  the convolution product (c.f. Appendix A) which is well defined on whole $H^*$.   
 Equality holds in the finite-dimensional case, i.e.  when  $H^\circ=H^*$.  
 There is  a canonical duality between Hopf algebras $H, H^\circ$ given by the evaluation map
 \begin{equation}\label{vf7}
 	 <X|\alpha> \doteq\alpha(X)\,, \quad 
 \end{equation}
 such that
  \begin{equation}\label{vf8}
   <XY|\alpha>=\alpha_{(1)}(X)\alpha_{(2)}(Y)\,,\qquad <X|\alpha\star\beta>=\alpha(X_{(1)})\beta(X_{(2)})\,,
 \end{equation}
 where $\Delta^\circ(\alpha)=\alpha_{(1)}\otimes \alpha_{(2)}$ and $\star$ denotes the convolution product. Of course, we also have $\epsilon^\circ(\alpha)=<1|\alpha>$,  $\varepsilon(X)=<X|\varepsilon>$ and $	<S(X)|\alpha>=<X|S^\circ(\alpha)>$.
  
  Consider a universal FODC $d^\circ:H^\circ\rightarrow\Omega^U_1(H^\circ)\equiv\ker\mu^\circ\subset \ker\Delta^*$, which is now a bicovariant universal FODC (c.f. Appendic C), i.e. $H^\circ\otimes H^\circ$ bears Woronowicz type bicovariant structure.  In contrast, $H\otimes H$ is equipped with
  a bicovariant structure adopted to codifferential calculi.
  \begin{pr}
  		The pairing  \eqref{vf7} extends naturally to the pairing between $H\otimes H$ and  $H^\circ\otimes H^\circ$ as well as to the pairing between  $\Upsilon^U_H$ and $\ker\mu^\circ= \{\omega=\sum_i \alpha^i\otimes\beta^i\in H^\circ\otimes H^\circ | \sum_i(\alpha^i\star\beta^i)(X)=0 \,\ \mbox{for}\,\ X\in H\}$, where 		
  	\begin{equation}\label{cd1}
  	<[X\otimes Y] |\omega> =\sum_i \alpha^i(X)\beta^i(Y)\,.
  \end{equation}		
  The following, preserving the bicovariance properties are thus satisfied  
  	\begin{equation}\label{cd2a}
  	< Z\triangleright[X\otimes Y]|\omega> = <Z\otimes [X\otimes Y]|\Delta^\circ_L(\omega)>\,
  \end{equation}
  where  $\Delta^\circ_L(\omega)= \alpha_{(1)}\star \beta_{(1)} \otimes (\alpha_{(2)}\otimes \beta_{(2)})$ 
  	\begin{equation}\label{cd2}
  	< [X\otimes Y]\triangleleft Z|\omega> = < [X\otimes Y]\otimes Z|\Delta^\circ_R(\omega)>\,
  \end{equation}
  where  $\Delta^\circ_R(\omega)= (\alpha_{(1)}\otimes \beta_{(1)}) \otimes \alpha_{(2)}\star \beta_{(2)}$.
  	\begin{equation}\label{cd3}
  	< \Delta^U_L([X\otimes Y])|\gamma\otimes\omega> = < [X\otimes Y]|\gamma\triangleright\omega>\,
  \end{equation}
  where $\gamma\triangleright\omega= \sum_i\gamma\star\alpha^i\otimes\beta^i $.
  \begin{equation}\label{cd4}
  	< \Delta^U_R([X\otimes Y])|\omega\otimes\gamma> = < [X\otimes Y]|\omega\triangleleft\gamma>\,
  \end{equation}
  where $\omega\triangleleft\gamma= \sum_i\alpha^i\otimes\beta^i\star\gamma $.
  Moreover ($d^U\alpha=\alpha\otimes\varepsilon-\varepsilon\otimes\alpha$)
  	\begin{equation}\label{cd5}
  	<\delta^U[X\otimes Y] |\alpha> =<[X\otimes Y] |d^\circ\alpha> =\alpha(X)\varepsilon(Y)-\alpha(Y)\varepsilon(X)\,.
  \end{equation}
  Finally,
  	\begin{equation}\label{cd6}
  <[X\otimes 1]\triangleleft Y|\omega>=	<r(\overline{X}\otimes Y) |\omega> =<\overline{X}\otimes Y |s'(\omega)>= <[X\otimes 1]\otimes Y| \Delta_R^\circ(\omega)>\,,  
  \end{equation}
  where $r(\overline{X}\otimes Y) =[XY_{(1)}\otimes Y_{(2)}]$	and $s'(\omega)=\sum_i \alpha^i_{(1)}\otimes\alpha^i_{(2)}\star\beta^i$ are Woronowicz maps.
  \begin{proof}
  	All these properties can be checked by direct computations using \eqref{vf8}. For example,  to  show  \eqref{cd2} one computes:
  	$$	< Z\triangleright[X\otimes Y]|\omega> =\sum_i \alpha^i(Z_{(1)}X)\beta^i (Z_{(2)}y)=
  	\sum_i \alpha^i_{(1)}(Z_{(1)})\alpha^i_{(2)}(X)\beta^i_{(1)}(Z_{(2)})\beta^i_{(2)}(Y)$$
  	$$=\sum_i (\alpha^i_{(1)}\star\beta^i_{(1)})(Z)\alpha^i_{(2)}(X)\beta^i_{(2)}(Y)=\omega_{<0>}(Z)<[X\otimes Y] | \omega_{(1)}>\,.$$
  	See also Appendix C.
  \end{proof}\end{pr}	
   
  Using  isomorpisms of bicovariant bimoduules $\Upsilon^U_H\cong \overline{H}_L\otimes H$ and 
  $\ker\mu^\circ\cong (\ker\epsilon^\circ)_L \otimes H^\circ$ one gets
  \begin{lema}
  	The pairing \eqref{vf7} reduces to the pairing  between left-left Y-D modules  $\overline{H}_L$ and $(\ker\epsilon^\circ)_L  = \{\alpha\in H^\circ | \alpha(1)=0\}$ 
  	\begin{equation}\label{vf8b}
  		<\overline{X}|\alpha> \doteq\alpha(X)\,,\quad \mbox{for}\quad \alpha\in\ker\epsilon^\circ
  	\end{equation}
  	In particular, there is a duality between left-left Y-D module  structures  on  $\overline{H}_L\in {}^{H}_{H}\mathfrak{YD}$ and  $(\ker\epsilon^\circ)_L\in {}^{H^\circ}_{H^\circ}\mathfrak{YD}$  
  	\begin{equation}
  		 \langle \mathrm{ad}^L_Y\overline{X} | \alpha\rangle =\langle Y\otimes \overline{X} |  \widetilde\Xi_L(\alpha) \rangle\,,
  	\end{equation} 
  	where $\widetilde\Xi_L(\alpha)=\alpha_{(1)} \star S^\circ(\alpha_{(3)})\otimes \alpha_{(2)} $ and $\alpha, \alpha_{(2)}\in \ker\epsilon^\circ$. 
  	\begin{equation}
  		\langle \overline{X} | \beta\succ \alpha\rangle =\langle   \Xi_L(\overline{X}) | \beta\otimes \alpha \rangle\,,
  	\end{equation}
  	where $\beta\succ\alpha= \beta\star\alpha $, $\beta\in H^\circ$.
  	
  	\end{lema}
 Of course, similar results are valid if we set a universal bicovariant FODC on $H$ and a universal FOCC on $H^\circ$. For more concrete examples of duality between matrix and Lie algebra type quantum groups see, e.g. \cite{Gavarini24} and references therein.
  
 \subsection{Universal covariant and  universal Killing vector fields}
Another interesting relationship, being a counterpart of the classical situation, is between quantum
tangent vectors, ie. vectors attached to a point and quantum vector fields, i.e. "global" operations on functions or differential forms, noticed in \cite{BeggsMajid,Weber2025} and earlier \cite{Zumino,Delius,Pflaum}. Here we generalize this idea to the case of universal calculus. We mention that general idea of nocommutative (i.e. quantum) vector fields appeared in \cite{Borowiec96} under the name of (right/left) Cartan pairs, see also \cite{AB23} for more recent presentation and historical perspective. 
They have been used and developed further for the construction of quantum  Riemannian geometry in \cite{BeggsMajid} and quite recently in \cite{Julius26,Rita26}. 
\begin{rrm} (FODC and vector fields \cite{Borowiec96,AB23})
	Let  $(\mathcal{A}, \mu, 1)$ be a unital associative algebra.  Consider a  FODC $d:\mathcal{A}\rightarrow \Omega$, where $\Omega$ is a bimodule 
	of one-forms. We define its right dual  $\Omega^\dagger\doteq \mathrm{Hom}_{(-,\mathcal{A})}(\Omega, \mathcal{A})$ as the space of right module morphisms, i.e. $\mathcal{A}$-linear maps together with a bimodule structure defined on it \footnote{It should be noticed that $\Omega\cong V\otimes \mathcal{A}$ is a right free module. Therefore, as a vector space $\Omega^\dagger\cong \mathrm{Hom}_\mathbb{K}(V, \mathcal{A})$, where $V=d(\mathcal{A})$. In particular $(\ker\mu^\circ)^\dagger\cong \mathrm{Hom}_\mathbb{K}(\ker\varepsilon^\circ, H^\circ)$.}
	\begin{equation}\label{cvf1}
		(f.X.g)(\omega h) \doteq f X(g\omega)h
	\end{equation}
	where $X\in\Omega^\dagger, f,g,h\in \mathcal{A}, \omega\in\Omega$. Further, with any element $X\in\Omega^\dagger $ one can associate the endomorphism $X^\triangleright\in \mathrm{End}^0_\mathbb{K}(\mathcal{A})\doteq \{A\in \mathrm{End}_\mathbb{K}(\mathcal{A}) | A(1)=0\}$
	\begin{equation}\label{cvf2}
		X^\triangleright (f)\doteq X(df)
	\end{equation}
	via the quantum contraction.
	Thus, the following Leibniz rule  is satisfied
	\begin{equation}\label{cvf3}
		X^\triangleright (fg)= X^\triangleright (f)g + (X.f)^\triangleright (g)\,.
	\end{equation}
	 In this way, the pair $(\Omega^\dagger, \triangleright)$ satisfies the requirements of an abstract definition of a Cartan pair from \cite{Borowiec96}. In this case associated with the FODC $d: \mathcal{A}\rightarrow\Omega$.
	Particularly, for the universal  FODC, $\Omega=\ker\mu$: $d^U f=f\otimes 1-1\otimes f$, one gets a "global" universal action 
	\begin{equation}\label{cvf4}
		X^\triangleright (f) \doteq X(f\otimes 1-1\otimes f)\,,
	\end{equation}
	and $(\ker\mu)^\dagger \cong \mathrm{Hom}_\mathbb{K}(\overline{\mathcal{A}},\mathcal{A})\cong
	\mathrm{End}^0_\mathbb{K}(\mathcal{A})\ni X^\triangleright$ stands for a universal vector field .
\end{rrm}
Let now $H$ be a Hopf algebra. The universal differential calculus over $H$ (considered as "an algebra of noncommutative functions") $d^U: H\rightarrow  \ker\mu$, where $d^U f= f\otimes 1-1\otimes f$ is bicovariant. If we define a universal tangent vector space for $d^U$ according to Woronowicz (c.f.\cite{Woronowicz,Klimyk,BeggsMajid,Weber2025}) ($\varepsilon^\circ(\alpha)=\alpha(1)$) 
\begin{equation}\label{T1}
	\mathcal{T}^U_H \doteq\{ \alpha\in H^\circ \subset H^* \ |\ \alpha(1)=0\} \equiv \ker\varepsilon^\circ\,,
\end{equation}
then it has a canonical left Y-D structure \footnote{Notice that Woronowicz in his definition requires extra condition $\alpha(\mathcal{R})=0$, where $\mathcal{R}$ is a right Y-D submodule in $(\ker\epsilon^\circ)_R$ of a finte codimension.}.
Therefor, it automatically becomes a quantum Lie algebra with the braiding
\begin{equation}\label{DualqBraid}
	\tau_q(\alpha\otimes\beta)\doteq  \alpha_{(1)}\star\beta\star S^\circ(\alpha_{(2)})\otimes\delta^\circ(\alpha_{(3)})
\end{equation}
where $\alpha, \beta\in \ker\varepsilon^\circ $ and $\delta^\circ(\gamma)\doteq \gamma-\gamma(1)\varepsilon^\circ$ for $\gamma\in H^\circ$. Similarly, one gets the (quantum) bracket  
\begin{equation}\label{DualqLie}
	[\alpha, \beta]_q\doteq   
	\alpha_{(1)}\star\beta\star S^\circ (\alpha_{(2)}) 
\end{equation}
This reconstructs Woronowicz universal quantum Lie algebra. Notice that $H^\circ$ can be equipped in a dual FOCC.
Introducing  "global" objects acting on $H$ is quite challenging in such setting  and will not be considered here. It leads to the notion of a cotangent bundle and covector fields \cite{AB99}. Here instead,  we shall study global objects in the dual setting.
 
Consider now a universal FODC over dual Hopf algebra 
$d^\circ: H^\circ\rightarrow  \ker\mu^\circ$, where $d^\circ\alpha= \alpha\otimes\varepsilon-\varepsilon\otimes\alpha$, c.f. \eqref{vf6}. We are going to construct some "global" quantities. Notice that the  bicovariant bimodule $\ker\mu^\circ\cong \ker\epsilon^\circ\otimes H^\circ$ can be considered as a right free module and comodule. Due to Lemma \ref{lemma1} the following vector spaces are isomorphic
\begin{equation}\label{cvf5}
	 (\ker\mu^\circ)^* \cong \mathrm{Com}^{(-,H^\circ)}(\ker\mu^\circ,H^\circ)\,
	  \subset \mathrm{Hom}_\mathbb{K}(\ker\mu^\circ,H^\circ)\,,
\end{equation}
by $(\ker\mu^\circ)^*\ni v\leftrightarrow \tilde v\in \mathrm{Com}^{(-,H^\circ)}(\ker\mu^\circ,H^\circ)$, where $\tilde v(d^\circ\alpha\star\beta)= (v\otimes\mathrm{id})\Delta_R^\circ(d^\circ\alpha\star\beta)$ is a right comodule map and $v=\epsilon^\circ \circ \tilde v$. More explicitly (c.f. Appendix C)
\begin{equation}\label{cvf6}
	 \tilde v(d^\circ\alpha\star\beta)\doteq v(d^\circ\alpha_{(1)}\star\beta_{(1)}) \alpha_{(2)}\star\beta_{(2)} 
\end{equation}
denotes a global action on differential forms.
One also defines a map to the (right) universal vector fields $(\ker\mu^\circ)^*\ni v\rightarrow \hat v\in (\ker\mu^\circ)^\dagger\cong 
\mathrm{Mod}_{(-,H^\circ)}(\ker\mu^\circ,H^\circ)\subset \mathrm{Hom}_\mathbb{K}(\ker\mu^\circ,H^\circ)$, 
 where  the right module map $\hat v$ is defined as 
\begin{equation}\label{cvf7}
\hat v(d^\circ\alpha\star\beta)\doteq \tilde v(d^\circ\alpha)\star\beta=v(d^\circ\alpha_{(1)})\alpha_{(2)}\star\beta\,.
\end{equation}
Both maps agrees on differential $d^\circ\alpha$.
Consequently, an action on $H^\circ$ is given by the map $(\ker\mu^\circ)^* \ni \hat v\rightarrow  v^\triangleright\in \mathrm{End}^0_\mathbb{K}(H^\circ)\doteq \{A\in \mathrm{End}_\mathbb{K}(H^\circ) | A(\varepsilon)=0\}$ reads (c.f. \eqref{cvf2})
\begin{equation}\label{cvf8}
	 v^\triangleright (\alpha)\doteq\hat v(d^\circ\alpha)=\tilde v(d^\circ\alpha)=v(d^\circ\alpha_{(1)})\alpha_{(2)} \,,
\end{equation}
i.e. 
\begin{equation}\label{cvf12}
	v^\triangleright =(v\circ d^\circ\otimes \mathrm{id})\circ\Delta^\circ\,.
\end{equation}
Let us introduce the  vector space  
	($\alpha(1)=\epsilon^\circ(\alpha)$ and $X(1^\circ)=\varepsilon(X)$)
	\begin{equation}\label{cvf9}
		 \mathcal{T}_\circ =\{ v\in (\ker\mu^\circ)^*\ |\ v(d^\circ\alpha\star\beta)=v(d^\circ\alpha)\beta(1)\} \,.
		 	\end{equation}
Thus restricting all three maps to the subdomain $\mathcal{T}_\circ$ give rise to the following:
\begin{lema}\label{cvf10}
	$v\in \mathcal{T} _\circ\ \Longleftrightarrow \hat v=\tilde v\in \mathrm{Hom}_\mathbb{K}(\ker\mu^\circ,H^\circ)$, i.e. 
	$\tilde v(d^\circ\alpha\star\beta)=\hat v(d^\circ\alpha\star\beta)=\tilde v(d^\circ\alpha)\star\beta$.  This implies $v^\triangleright(\alpha)= v(d^\circ\alpha_{(1)})\alpha_{(2)}$ for $v\in \mathcal{T} _\circ$. 
 Therefore, we call the corresponding vector field $\hat v$ right covariant. Moreover, $\mathcal{T}_\circ\cong (\ker\varepsilon^\circ)^*$.  
	\begin{proof}
Applying the rule \eqref{cvf9} to \eqref{cvf6} one easily gets \eqref{cvf7}.
Since the map $v\mapsto \tilde v$ is injective, then two remaining ones are injective too. 	 
From \eqref{cvf9} it follows that functionals from $\mathcal{T} _\circ$ are completely determined by their values on differential, i.e. 
$\mathcal{T} _\circ\cong (d^\circ(H^\circ))^*$. But $d^\circ(H^\circ)\cong \overline{H^\circ}\cong \ker\varepsilon^\circ$.
	\end{proof}
\end{lema}
\begin{rrm} Comparing with \eqref{T1} it can be shown that  $\mathcal{T}^U_{H^\circ}\equiv \ker\varepsilon^{\circ\circ}\subset (\ker\varepsilon^\circ)^*$,
 	where $\ker\varepsilon^{\circ\circ}$ is a counit in $(H^\circ)^\circ$. In this case $\mathcal{T} _\circ$ is a good candidate for a universal tangent space for the universal  differential calculus on $H^\circ$.
\end{rrm}

Consider now $\overline{H}_L$  as a left-left Y-D  module. 
 We are going to show how this is related to the notion of covariant vector fields for the universal FODC on $H^\circ$ studied above.

Firstly, we replace the left coaction $\Xi^L(\overline{X})=X_{(1)}\otimes\overline{X_{(2)}} $  by a right $H^\circ$ action $\overline{X}\star\alpha  \doteq \alpha(X_{(1)}) \overline{X_{(2)}} $. 
Next,  to any element $\bar H\ni\overline{X}\mapsto \eth(\bar X)\in\mathcal{T}_\circ$ one can associate 
\footnote{In fact, this functional is well defined on $H^\circ\otimes H^\circ\supset \ker\mu^\circ$.} 
\begin{equation}\label{svf1}
	\eth(\overline{X})(d^\circ\alpha\star\beta)\doteq \alpha(\delta(X))\beta(1)=(\alpha(X)-\varepsilon(X)\alpha(1))	\beta(1)
\end{equation}
an element of the universal quantum  tangent space. Therefore, the image $\eth(\overline{H}_L)\subset \mathcal{T}_\circ$ one can call a  universal quantum tangent space. This image can be further extend to the images 
$$\widetilde{ \eth(\overline{H}_L)} =\widehat{\eth(\overline{H}_L)} \subset  \mathrm{Com}^{(-,H^\circ)}(\ker\mu^\circ,H^\circ)\cap  \mathrm{Mod}_{(-,H^\circ)}(\ker\mu^\circ,H^\circ) \subset (\ker\mu^\circ)^\dagger\,,$$   
or $\eth (\overline{H}_L)^\triangleright\subset \mathrm{End}^0_{\mathbb{K}}(H^\circ)$, which we would like to call a universal quantum vector fields. We find
\begin{equation}\label{svf1a}
\widehat{\eth(\overline{X})}(d^\circ\alpha\star\beta)= \alpha_{(1)}(X)\alpha_{(2)}\star\beta-\varepsilon(X)\alpha\star\beta
\end{equation}
If $\sum_i \alpha^i\star\beta^i=0$, i.e., $\sum_i \alpha^i\otimes \beta^i\in \ker\mu^\circ$ then
\begin{equation}\label{svf1b}
	\widehat{\eth(\overline{X})}(\sum_i \alpha^i\otimes\beta^i)= \sum_i\alpha^i_{(1)}(X)\alpha^i_{(2)}\star\beta^i\,.
\end{equation} 
It is easy to check the following
\begin{lema}\label{svf2}
	The map $\eth: \overline{H}_L\rightarrow\mathcal{T}_\circ$ 
	 	is compatible with the right module action  $\overline{X}\star\beta=\beta(X_{(1)})\overline{X_{(2)}}$ via the Leibniz rule, c.f. \eqref{cvf3}
	\begin{eqnarray}\label{svf3}
		\eth(\overline{X}\star\beta)( d^\circ\alpha) = \eth(\overline{X}) (\beta\star d^\circ\alpha)  =
		\eth(\overline{X})(d^\circ(\beta\star\alpha)-d^\circ\beta\star\alpha )=
		\beta(X_{(1)})\eth(\overline{X_{(2)}}) (d^\circ\alpha)\,\nonumber\\
		=(\beta\star\alpha)(X)-\alpha(1)\beta(X)\,.
	\end{eqnarray}
	and induces a vector field action
	\begin{equation}\label{svf5b}
		\eth(\overline{X})^\triangleright (\alpha)=\widehat{\eth(\overline{X})}( d^\circ\alpha)  =
		\alpha_{(1)}(X)\alpha_{(2)}-\varepsilon(X)\alpha
	\end{equation} 
	For left adjoint it provides:
	\begin{equation}\label{svf4}
		\eth(\mathrm{ad}^L_X \overline{Y})(d^\circ\alpha) =(\alpha_{(1)}\star S^\circ(\alpha_{(3)}))(X)\,\alpha_{(2)}(Y) -\varepsilon(X)\varepsilon(Y)\alpha(1)\,.
	\end{equation} 
	Consequently :	
	\begin{equation}\label{svf5}
		\eth([\overline{X},\overline{Y}]_q)(d^\circ\alpha) =(\alpha_{(1)}\star S^\circ(\alpha_{(3)}))(X)\,\alpha_{(2)}(Y) - \varepsilon(X) \alpha(Y)\,.
	\end{equation}  
\end{lema}  
However, it should be mentioned that the first map $\overline{X}\mapsto \eth(\overline{X})$ does not need, in general, to be injective. This is because a canonical map
$H\rightarrow (H^\circ)^*$ is not always injective.  If injectivity is fulfilled  then the corresponding algebra  $H$ is called proper \cite{Sweedler69,Montgomery_book,Radford}. For proper algebras the dual pairing \eqref{vf7} is non-degenerate. As the finale result we close this section with the following proposition 
 \begin{pr}
	 Assume that $H$ is a proper Hopf algebra. The subspace of covariant universal vector fields  $\widehat{\eth(\overline{H}_L)}\subset(\ker\mu^\circ)^\dagger$ (or $\eth(\overline{X})^\triangleright \subset 	\mathrm{End}^0_\mathbb{K}(H^\circ)$), bears  induced quantum Lie algebra structure if we define
	 \begin{equation}
	 	\underline \tau_q(\eth(\overline{X})^\triangleright\otimes \eth(\overline{Y})^\triangleright)\doteq
	 	(\eth( )^\triangleright\otimes\eth( )^\triangleright)(\tau_q(\bar X\otimes\bar Y)\,,
	 \end{equation} 
	 and \footnote{Classically, they correspond the so-called fundamental, or Killing, or invariant vector fields on a group manifold.}
	 \begin{equation}
	 	\underline{[\eth(\overline{X})^\triangleright, \eth(\overline{Y})^\triangleright]}_q\doteq
	 	\eth([\bar X, \bar Y]_q)^\triangleright\,.
	 \end{equation}
	 The inherited left Y-D  structure over $H$ is defined  by the left coaction
	 \begin{equation}
	 	\underline \Xi_L\circ \eth( )^\triangleright \doteq (\mathrm{id}\otimes \eth( )^\triangleright)\circ \Xi_L\,,
	 \end{equation}
	and  the left action
	 \begin{equation}
	  \underline\blacktriangleright \circ (\mathrm{id}\otimes\eth( )^\triangleright) \doteq \eth()^\triangleright\circ \blacktriangleright \,.
	 \end{equation}
	   Or the other way around, one can defined a "global" action $\bar H_L\ni \bar X\mapsto \bar X^{\underline{\triangleright}}\in \mathrm{End}^0_{\mathbb{K}}(H^\circ)$ by
	 $$ \bar X^{\underline{\triangleright}}(\alpha) \doteq \alpha_{(1)}(X)\alpha_{(2)}-\varepsilon(X)\alpha\, \quad\mbox{for}\quad \alpha\in H^\circ\,,$$
	  defining universal Killing vector fields (at least in the case of matrix quantum group $H^\circ$). 
	\end{pr}

%

%

\subsection{Some remarks on classification of bicovariant FOCCs}
Pushing forward the methods of generating comodules out of vector subspaces, $V\mapsto 	\mathcal{L}(V)$,  presented earlier, we can state now as follows:
 \begin{pr}Let $V\subset \overline{H}_L$ be a non-trivial vector subspace. We denote by 
 	\begin{equation}
 		V_L\doteq \mathrm{Span}\{  v\star\alpha \,| v\in V, \alpha\in H^* \}\,, 
 	\end{equation}
 	the smallest subcomodule $V_L\subset \overline{H}_L$ containing $V$.   Then one defines
	\begin{eqnarray}\label{spanBimoduleL}
	\mathcal{L}(V)\doteq	H\blacktriangleright V_L=\mathrm{Span}\{ h\blacktriangleright v\, | v\in V_L,  h\in H\}\,    		
	\end{eqnarray}
	as the smallest Y-D submodule containing  $V_L$. Alternatively, one can consider first  
		\begin{eqnarray}\label{spanBimoduleL2}
		V^\blacktriangleright\doteq	H\blacktriangleright V=\mathrm{Span}\{ h\blacktriangleright v\, | v\in V,  h\in H\}\, .  		
	\end{eqnarray}
	Then define
		\begin{equation}\label{spnComoduleL2}
			\mathcal{L}(V)\doteq  \mathrm{Span}\{  v\star\alpha \,| v\in V^\blacktriangleright, \alpha\in H^* \}\,. 
	\end{equation}
	\begin{proof}
	The equivalence of \eqref{spanBimoduleL} and \eqref{spnComoduleL2}	comes from the fact that 	$\mathcal{L}(V)$ is unique.
	We need to show that $\Xi_L(H\blacktriangleright V_L)\subset H\otimes (H\blacktriangleright V_L)$. But this follows directly from formula	  \eqref{b9}. Similar arguments applies to  \eqref{spnComoduleL2}.
	\end{proof}
\end{pr}
Notice that the algebra $H^*$ can be replaced  by  the  subalgebra $H^\circ$ if it is a dense subalgebra in $H^*$, i.e., $H$ is a proper algebra \cite{Sweedler69, Montgomery_book,Radford}. 

Again, one-dimensional subspaces become very important since the properties described in Lemma \ref{lem_gen_space} are still valid for $\mathcal{L}(V)$.

  We already know that for a set coalgebra  $\mathbb{K}(\mathcal{O})$, two FOCCs are isomorphic if, and only if they generate isomorphic graphs \ref{th_set_iso}. 
If $\mathcal{O}=G$ is a group, then we have a canonical Hopf algebra structure, and we may ask about bicovariant FOCCs.
\begin{lema}
	The bicovariant FOCCs for a group algebra $\mathbb{K}(G)$ is generated by non-trivial normal subgroups $H\subset G$:
	\begin{eqnarray}
        \mathcal{L}(\overline{\mathbb{K}(H)})&=&\overline{\mathbb{K}(H)}.
	\end{eqnarray}
	\begin{proof}
		We already know that one deals with a left covariant FOCC. Adinvariance of $\overline{\mathbb{K}(H)}$ is a direct result of $H$ being the normal subgroup.
	\end{proof}
\end{lema}

Now, we consider the universal enveloping algebra $U(\mathfrak{g})$ for some Lie algebra ~~$\mathfrak{g}=\mathrm{span}\{X_i;i\in I\}$.
\begin{lema}\label{pr_ex_borrel}
    The partial classification of bicovariant FOCCs for universal enveloping algebra $U(\mathfrak{g})$: The Y-D modules are given by non-trivial Lie-ideal $\mathfrak{g}\supseteq\mathfrak{h}=\mathrm{span}\{X_i;i\in J^\mathfrak{h}\}$\footnote{First we choose a basis in $\mathfrak{h}$, then we extend it to a basis in $\mathfrak{g}$ and construct the corresponding Poincare-Birkhoff-Witt basis in $U(\mathfrak{g})$. We can also take $\mathfrak{h}=\mathfrak{g}$.}, and a number $M\in\mathbb{N}_+$:
    \begin{eqnarray}
     \mathcal{L}^M({\overline{\mathfrak{h}}})&:=&\overline{U(\mathfrak{h})^M}=\mathrm{span}\left\{\overline{\prod_{i\in J^\mathfrak{h}}X_i^{n_i}};n_i\in\mathbb{N}_0,0<\sum_{i\in J^\mathfrak{h}}n_i\leq M\right\}\,.
    \end{eqnarray}
    is a Yetter-Drinfeld module, and it generates the bicovariant FOCC over Hopf algebra $U(\mathfrak{g})$. Moreover, for $M=1$ the corresponding cocalculi are cocommutative.
    \begin{proof}
       Again we know that one deals with a left covariant FOCC. Moreover, for $M=1$ it is cocommutative since it is generated by primitive elements  (see Proposition~\ref{ClasBicovCocom}). An adjoint action in the universal enveloping algebra reduces to Lie brackets. This shows that Y-D modules related to Lie ideals are indeed adinvariant. The additional index $M$ appears because neither adjoint action nor coaction can increase the order of an element.
    \end{proof}
\end{lema}
\begin{rrm}
  This partial classification is valid for every universal enveloping algebra. There are also, for example,  FOCCs generated by  non-unital central elements $C\in U(\mathfrak{g})$. 
  We will show it in the following example.    
\end{rrm}

\begin{ex}[$U(\mathfrak{b}_+)$]
    Consider the universal enveloping algebra with two generators $H,~E$, and commutator $[H,E]=E$.

    We see that the commutator cannot raise the power of $H$. So we can define the families of Y-D modules with two indices ($m\leq N\geq 1$):
    \begin{eqnarray}
        \mathcal{L}<\overline{H^mE^N}>&=&\mathrm{Span}\{\overline{H^iE^j};i\leq m,i+j\leq N\},\\
        \dim \mathcal{L}<\overline{H^mE^N}>&=&(N-\frac{m}{2})(m+1)+m.
    \end{eqnarray}
    Notice that if $m$ is odd, then $(m+1)$ is even, so the dimension is a natural number. What is more, for $m=0$ we obtain $\mathcal{L}(<\overline{E^N}>)$ and for $m=N$ we are obtaining $\mathcal{L}(<\overline{H^N}>)$. In all other cases, they cannot be presented in the form of lemma \ref{pr_ex_borrel}.

    In a case of enveloping algebra $U(\mathfrak{b}_+)$, we have two non-trivial cocommutative and bicovariant cocalculi:
    \begin{eqnarray}
        \mathcal{L}<\overline{E}>&=&<\overline{E}>\,,\\
        \mathcal{L}<\overline{H}>&=&\mathrm{Span}\{\overline{H},\overline{E}\}\,.
    \end{eqnarray}
\end{ex}
 
\begin{ex}(Sweedler Hopf algebra)
Let us go back to the example~\ref{ex_sweedler_coalgebra} and remind that this is a Hopf algebra $H=\mathrm{Span}\{1,g,X,Xg\}$ with the relations:
\begin{eqnarray}
    g^2=1,\qquad\qquad X^2=0,\qquad\qquad Xg=-gX,
\end{eqnarray}
and the antipod $S(g)=g^{-1}\equiv g$, $S(X)=-gX$.
In the previous classification, we did not search for bicovariant FOCCs. Now we find them and check how they are related to the FOCCs we found earlier.
There are two left-left Y-D submodules in $\overline{H}$:
\begin{itemize}
    \item One-dimensional: $\mathcal{L}<\overline{X}>=<\overline{X}>$,
    \item Two-dimensional: $\mathcal{L}<\overline{g}>=\mathrm{Span}\{\overline{g},\overline{Xg}\}$.
\end{itemize}
Moreover, bicovariant bimodules:
\begin{eqnarray}
    \Phi_R(\mathcal{L}<\overline{X}>)=\Upsilon<[Xg\otimes X]>+\Upsilon<[\phi(Xg)\otimes \phi(X)]>=\nonumber\\
    \mathrm{Span}\{[X\otimes 1],[Xg\otimes g],[Xg\otimes X],[X\otimes Xg]\}\,,
    \end{eqnarray}
 \begin{eqnarray}   
    \Phi_R(\mathcal{L}<\overline{g}>)=\Upsilon<[X\otimes X]>\oplus\Upsilon<[\phi(X)\otimes \phi(X)]>=\nonumber\\
    \mathrm{Span}\{[g\otimes 1],[1\otimes g], [1\otimes X], [g\otimes Xg],\\ {}[Xg\otimes Xg],[X\otimes X], [Xg\otimes 1],[X\otimes 1]\}\nonumber\,,
\end{eqnarray}
decompose into the sum/direct sum of noncovariant calculi and their images under the mapping $\phi$ that is an automorphis of the coalgebra but not the Hopf algebra.
The first four terms in the last formula span the right covariant FOCC. The concluding remark is that the direct sum of the two bicovariant bimodules yields the universal bicomodule in this case. 
\end{ex}

  \section{Examples of bicovariant FOCC on quantum groups}

\subsection{$U_Q(\mathfrak{b}_+)$, $Q\neq 0, 1\neq Q^n\in\mathbb{K}, n\in\mathbb{N}$}\label{ex_Uqb+}

In this section we will consider deformed universal enveloping algebra for $ax+b$ group. We will classify all bicovariant FOCCs and compare it to similar classification of bicovariant FODC in \cite{Oeckl}.

We have three generators $X,g,g^{-1}$ with commutation relations $Xg=QgX$ and $gg^{-1}=1$. Elements $g^n$ ($n\in \mathbb{Z} $) are group-like, and $X$ is skew-primitive one:
\begin{eqnarray}
    \Delta(X)=X\otimes 1+g\otimes X,&\quad&S(X)=-g^{-1}X\,, \quad S(g^n)=g^{-n}.
\end{eqnarray}
Thus ($\varepsilon(g)=1, \varepsilon(X)=0$)
\begin{eqnarray}
	  U_Q(\mathfrak{b}_+)&=&\mathrm{Span}\left\{X^ng^m|(n,m)\in(\mathbb{N}_0\times\mathbb{Z}) \right\}.
\end{eqnarray}

We are looking for left-left Y-D submodules:
\begin{eqnarray}
    \mathcal{L}\subset \overline{U_Q(\mathfrak{b}_+)}_L&=&\mathrm{Span}\left\{\overline{X^ng^m}|(n,m)\in(\mathbb{N}_0\times\mathbb{Z}) \setminus (0,0) \right\}.
\end{eqnarray}
For the classification we used natural parameter $n\in\mathbb{N}_+$. The finite-dimensional Y-D submodulus are:
\begin{eqnarray}
    \mathcal{L}<\overline{g^{-n}}>&=&\mathrm{span}\{\overline{X^ig^{-n}};0\leq i\leq n\}\nonumber\\
    \dim\mathcal{L}<\overline{g^{-n}}>&=&n+1
\end{eqnarray}
There are also infinite-dimensional Y-D submodules:
\begin{eqnarray}
    \mathcal{L}<\overline{g^{n}}>&=&\mathrm{Span}\{\overline{X^ig^{n}};0\leq i\}\\
    \mathcal{L}<\overline{X}>&=&\mathrm{Span}\{\overline{X^i};0<i\}\\
    \mathcal{L}<\overline{X^{n+1}g^{-n}}>&=&\mathrm{Span}\{\overline{X^ig^{-n}};0\leq i\}
\end{eqnarray}
This is classification of all possible non-decomposable Y-D submodules. Let us note, that $\mathcal{L}<\overline{g^{-n}}>\subset\mathcal{L}<\overline{X^{n+1}g^{-n}}>$.

The most important piece of calculation to check which singleton generate finite dimensional Y-D modulus is adjoint action with $X$:
\begin{eqnarray}
    \mathrm{Ad}_X\overline{X^mg^n}&=&(1-Q^{-n-m})\overline{X^{m+1}g^n}.
\end{eqnarray}
If we start with element $\overline{X^mg^n}$, we will generate infinite tower with just this adjoint action, unless it will be cut at some point. This is possible only for $m+n=0$, so we can see, that for finite-dimensional cases, we need negative power of $g$.

Now, we compare our result with \cite{Oeckl}. We have one-to-one correspondence between finite-dimensional calculi. However, it is not the case with the infinite-dimensional ones. A Y-D modules in \cite{Oeckl} take the form:
\begin{eqnarray}
    \Theta_{\lambda\neq 1}&:=&\mathrm{span}\left\{(1-\lambda g)(1-g),X^i(1-\lambda g);i\in\mathbb{N}_+\right\},\\
    \Theta_{1}&:=&\mathrm{span}\left\{(1-g),X^i(1-g);i\in\mathbb{N}_+\right\}.
\end{eqnarray}
 Consequently, it can be checked that only $\Theta_{q^k}$ for $k\in\mathbb{Z}$ have a dual counterpart. 
%
 
 We can write down the quantum Lie algebra structure for lowest dimensional codifferential calculus $\mathcal{L}<\overline{g^{-1}}>=\mathrm{Span}\{\upsilon_0=\overline{g^{-1}},\upsilon_1=\overline{Xg^{-1}}\}\}$. Quantum braiding is:
 \begin{eqnarray}
 	\tau_q(\upsilon_0\otimes\upsilon_0)=\upsilon_0\otimes\upsilon_0,&\qquad&\tau_q(\upsilon_1\otimes\upsilon_0)=\upsilon_0\otimes\upsilon_1+(1-Q)\upsilon_1\otimes\upsilon_0,\nonumber\\
 	\tau_q(\upsilon_0\otimes\upsilon_1)=Q\upsilon_1\otimes\upsilon_0&\qquad&\tau_q(\upsilon_1\otimes\upsilon_1)=\upsilon_1\otimes\upsilon_1.
 \end{eqnarray}
 The only non-trivial brackets are
 \begin{eqnarray}
 	{}[\upsilon_1,\upsilon_0]_q=(Q-1)\upsilon_1,&=&-[\upsilon_0,\upsilon_1]_q\,.
 \end{eqnarray}

\subsection{$U_q(\mathfrak{sl}(2))$, $q\neq 0, q^{2n}\neq 1,\ n\geq 1$}\label{ex_Uqsu2}

Consider the Drinfeld-Jimbo deformed enveloping algebra of $\mathfrak{sl}(2)$. There are four generators: $\left<E,F,K,K^{-1}\right>$ with the relations (c.f. \cite{Klimyk}):
\begin{eqnarray}
    KK^{-1}=K^{-1}K=1,&\qquad&KE=q^2EK,\\
    \left[E,F\right]=\frac{K-K^{-1}}{q-q^{-1}},&\qquad&KF=q^{-2}FK.
\end{eqnarray}
The coalgebra structure is:\\
\begin{center}
\begin{tabular}{ccc}
    $\Delta(K)=K\otimes K$ & $\varepsilon(K)=1$ & $S(K)=K^{-1}$\\
    $\Delta(E)=E\otimes K+1\otimes E$ & $\varepsilon(E)=0$ & $S(E)=-EK^{-1}$\\
    $\Delta(F)=F\otimes 1+K^{-1}\otimes F$ & $\varepsilon(F)=0$ & $S(F)=-KF$
\end{tabular}
\end{center}

The lowest-dimensional bicovariant FOCC is generated by four-dimensional Y-D module $\mathcal{L}<\overline{K}>$:
\begin{eqnarray}
    \upsilon_{00}=\overline{K},&\qquad\qquad&\upsilon_{10}=\overline{E},\\
    \upsilon_{01}=\overline{FK},&\qquad\qquad&\upsilon_{11}=\overline{EF-q^2FE}.
\end{eqnarray}
Quantum Lie algebra structure for this  FOCC is presented in Appendix F.

It is known that there exist two four-dimensional bicovariant FODC on $SL_q(2)$ called $D_-$ and $D_+$ (They were found by Stachura \cite{Stachura92}, see also  \cite{Woronowicz3} for left covariant three dimensional FODC, c.f. \cite{Klimyk}), which exist in the dual Hopf algebra to the one considered here\footnote{Dual pairing is defined in~\cite[section 4.4]{Klimyk} and \cite{Gavarini24} in more specific context.}. However, we have found only one dual four-dimensional bicovariant cocalculi on $U_q(\mathfrak{sl}(2))$. With some calculations, one finds out that cocalculus $\Phi_R(\mathcal{L}<\overline{K}>)$ is dual to $D_-$\footnote{One needs to be careful not compare left Y-D modules with the right ones.}. Because there is no other four-dimensional bicovariant cocalculi on $U_q(\mathfrak{sl}(2))$, we conclude that there is no FOCC dual to $D_+$ in $U_q(\mathfrak{su}(2))$. There is bigger Hopf algebra which is dual to $SU_q(2)$ \cite[chapter 3.1]{Klimyk}, and there exist dual cocalculi for both $D_+$ and $D_-$.

The second lowest-dimensional Y-D submodules $\mathcal{L}<\overline{K^2}>$ have nine elements.
\begin{eqnarray}
    \upsilon_{00}=\overline{K^2},&\qquad\upsilon_{10}=\overline{EK},&\qquad\upsilon_{20}=\overline{E^2},\\
    \upsilon_{01}=\overline{FK^2},&\qquad\upsilon_{11}=\overline{(EF-q^4FE)K},&\qquad\upsilon_{12}=\overline{E^2F-q^4FE^2},\\
    \upsilon_{02}=\overline{F^2K^2},&\qquad\upsilon_{21}=\overline{(EF^2-q^4F^2E)K},&\qquad\upsilon_{22}=\overline{E^2F^2-(q^2+q^4)EF^2E+q^6F^2E^2}.\nonumber\\
\end{eqnarray}
We did not calculate higher-dimensional ones, but we can propose a conjecture:
\begin{obs}[Conjecture]
    All finite-dimensional non-decomposable FOCC for $U_q(\mathfrak{sl}(2))$  are Y-D modules generated by singletons $\mathcal{L}<\overline{K^n}>$ for $n>0$, and their dimension is:
    \begin{eqnarray}
        \dim\mathcal{L}<\overline{K^n}>&=&(n+1)^2.
    \end{eqnarray}
\end{obs}

Using the adjoint action of $E$ and $F$, one can show that all modules in a form $\mathcal{L}<\overline{K^n}>$ have infinite dimension for $n<0$.
\begin{eqnarray}
    \mathrm{Ad}_E\overline{E^mK^n}&=&(1-q^{2n})\overline{E^{m+1}K^{n-1}},\\
    \mathrm{Ad}_F\overline{F^mK^n}&=&(1-q^{2(m-n)})\overline{F^{m+1}K^n}.
\end{eqnarray}
One cannot treat these formulas as a proof that given submodules are finite-dimensional; however, it is enough to show that all Y-D submodules of a type $\mathcal{L}<\overline{K^{-n}}>$ have infinite dimension.

\subsection{$SL_q(2)\,,\ q\neq 0$}\label{ex_SUq2}
Now, we consider the deformed matrix group $SL_q(2)$ (quantum matrix group \cite{Woronowicz2}, see also \cite{Schmuedgen}) in more general context.  This is an associative, unital algebra with the following generating relations:
\begin{eqnarray}
    &ux=qxu,\qquad vx=qxv,\qquad uy=q^{-1}yu,\qquad vy=q^{-1}yv,&\\
    &uv=vu,\qquad xy=1+q^{-1}uv,\qquad yx=1+quv.&
\end{eqnarray}
The coproduct  and antipodes are:\newline
\begin{center}
\begin{tabular}{||c||c|c|c||}\hline
    & $\Delta$ & $\varepsilon$ & $S$ \\\hline
    $x$ & $x\otimes x+u\otimes v$ & $1$ & $y$ \\
    $y$ & $y\otimes y+v\otimes u$ & $1$ & $x$ \\
    $u$ & $x\otimes u+u\otimes y$ & $0$ & $-qu$ \\
    $v$ & $y\otimes v+v\otimes x$ & $0$ & $-q^{-1}v$ \\\hline
\end{tabular}
\end{center}
providing the Hopf algebra structure, see also, e.g. \cite{Klimyk}). 

As we believe, FOCCs are more suited to quantized universal enveloping Lie algebras; this is the only deformed matrix group one considers in this paper. It is still a Hopf algebra, so we can define codifferential calculus, and it can be informative. 

First, we classify the lowest-dimensional right-covariant FOCCs with respect to left subcomodules $M^L\in\overline{SL_q(2)}$. There are no dimension one cocalculi, and for dimension two:
\begin{eqnarray}
	M^L<\overline{x}>&=&\mathrm{span}\{\overline{x},\overline{v}\},\\
	M^L<\overline{y}>&=&\mathrm{span}\{\overline{y},\overline{u}\}.
\end{eqnarray}
Dimension three left comodules:
\begin{eqnarray}
	M^L<\overline{x^2}>&=&\mathrm{span}\{\overline{x^2},\overline{xv},\overline{v^2}\},\\
	M^L<\overline{y^2}>&=&\mathrm{span}\{\overline{y^2},\overline{yu},\overline{u^2}\},\\
	M^L<\overline{xy}>&=&\mathrm{span}\{\overline{xy},\overline{xu},\overline{yv}\}.
\end{eqnarray}

One also finds that all non-decomposable bicovariant FOCCs are infinite-dimensional and can be classified with an integer number $M\in\mathbb{Z}$:
\begin{eqnarray}
	\mathcal{L}^M&=&\mathrm{Span}\{\overline{v^mx^{t}y^{s}u^n};m,n,t,s\in\mathbb{N}; m+t-n-s=M\}.
\end{eqnarray}
For particular cases, one can find generating singletons:
\begin{eqnarray}
	\mathcal{L}^{M>0}=\mathcal{L}<\overline{x^M}>,&\qquad\mathcal{L}^{M=0}=\mathcal{L}<\overline{xy}>,&\qquad\mathcal{L}^{M<0}=\mathcal{L}<\overline{y^M}>.
\end{eqnarray}
If we take into account the automorphism $\phi$ (c.f. Exercise \ref{ex_M2x2_coalgebra}), we can consider only $M\in\mathbb{N}_0$.
By duality arguments (Appendix D), this suggests that there are no finite-dimensional FODCs on the Drinfeld-Jimbo quantized Lie algebra $sl_q(2)$ from the previous subsection. In fact, we did not find any in the literature.

\subsection{$\kappa$-Poincar\'e Hopf algebra}
 We  use a classical basis making $\kappa$-Poincar\'e Hopf algebra a Drinfeld-Jimbo type deformation, as introduced in \cite{Pachol}. There are 11 generators\footnote{This example would work for $D+1$-dimensional Poincar\'e algebra as well.}:
\begin{eqnarray}
    H&:=&\mathrm{gen}\left<\Pi_0,\Pi_0^{-1},P_j,N_j,M_j;j=1,2,3\right>.
\end{eqnarray}
All non-zero commutation relation are (here $\kappa\in \mathbb{R}$ is the deformation parameter):  
\begin{eqnarray}
    {}[N_j,\Pi_0]=-\frac{i}{\kappa}P_j,&\quad[N_j,M_k]=-i\epsilon_{jkl}N_l,&\quad[N_j,N_k]=-i\epsilon_{jkl}M_l,\\
    {}[M_j,P_k]=i\epsilon_{jkl}P_l,&\quad[M_j,M_l]=i\epsilon_{jkl}M_l,&\quad[N_j,P_k]=-\frac{i}{2}\delta_{jk}G_\kappa.
\end{eqnarray}
where $G_\kappa:=\kappa(\Pi_0-\Pi_0^{-1})+\frac{1}{\kappa}\overrightarrow{P}^2\Pi_0^{-1}$ and $\overrightarrow{P}^2=\sum_jP_jP_j$. To have full Hopf algebra, we need coproducts:
\begin{eqnarray}
    \Delta\Pi_0=\Pi_0\otimes\Pi_0,\quad\Delta P_j=P_j\otimes\Pi_0+1\otimes P_j,\qquad\Delta M_j=M_j\otimes 1+1\otimes M_j,\\
    \Delta N_j=N_j\otimes 1+\Pi_0^{-1}\otimes N_j-\frac{1}{\kappa}\epsilon_{jkl}P_k\Pi_0^{-1}\otimes M_l,
\end{eqnarray}
with $\varepsilon\Pi_0=1$ and $\varepsilon P_j=\varepsilon M_j=\varepsilon N_j=0$ and antipods:
\begin{eqnarray}
    S\Pi_0=\Pi_0^{-1},\quad S P_j=-P_j\Pi_0^{-1},\quad S M_j=-M_j,\quad S N_j=-\Pi_0N_j-\frac{1}{\kappa}\epsilon_{jkl}P_kM_l.
\end{eqnarray}

For this example we will write only lowest-dimensional bicovariant FOCC, which is generated by five-dimensional Y-D module:
\begin{eqnarray}
    \mathcal{L}<\overline{\Pi_0}>&=&\mathrm{span}\{\upsilon_0=\overline{\Pi_0},\upsilon_i=\overline{P_i},\upsilon_C=\overline{C_\kappa}\},
\end{eqnarray}
where $C_\kappa=\kappa^2(\Pi_0+\Pi_0^{-1}-2)-\overrightarrow{P}^2\Pi_0^{-1}$.

At the end, let's write explicitly Yetter-Drinfeld structure:
\begin{eqnarray}
    &\Xi_L\upsilon_0=\Pi_0\otimes\upsilon_0,\qquad\qquad\Xi_L\upsilon_i=1\otimes\upsilon_i+P_i\otimes\upsilon_0,\\
    &\Xi_L\upsilon_C=\Pi_0^{-1}\otimes\upsilon_C-S(G_\kappa)\otimes\upsilon_0-2\sum_jP_j\Pi_0^{-1}\otimes\upsilon_j,
\end{eqnarray}
and all non-trivial ($a\triangleright\upsilon\neq\upsilon\varepsilon a$) actions take a form:
\begin{eqnarray}
   M_j\blacktriangleright\upsilon_k=i\epsilon_{jkl}\upsilon_l,\quad N_j\blacktriangleright\upsilon_0=-\frac{i}{\kappa}\upsilon_j,\quad N_j\blacktriangleright\upsilon_k=\delta_{jk}\left(\frac{i}{2\kappa}\upsilon_C-i\kappa\upsilon_0\right).
\end{eqnarray}
\begin{rrm}
	Let's note that for undeformed Poincer\'e Hopf we have four-dimensional Y-D modules, generated by Lie ideal of translations, however, in a limit we would five-dimensional Y-D module generated by mass Cassimir $\mathcal{L}<\overline{C}>=\mathrm{span}\{\overline{P_0},\overline{P_i},\overline{C}\}$.
\end{rrm}

\begin{rrm}
	Alongside one-dimensional right-covariant calculi, we have three-dimensional calculi generated by $\mathrm{Span}\{\overline{P_i\Pi_0^{-1}}\}$ and $\mathrm{Span}\{\overline{M_i}\}$.
\end{rrm}

\begin{rrm}
	One could see that the additional element in the 5-dimensional  cocalculus is the Casimir operator connected to a mass. There is a question whether the addition of a second Casimir\footnote{It is a square of Pauli-Lubański pseudovector.} will result with 12-dimensional cocalculus. Quantum version of the second Casimir has a rather complicated form  (c.f. \cite{Tolstoi}), and we didn't find a minimal Y-D module with this operator. However, we give an argument that the dimenssionality of such Y-D module is higher than 12. There is only one term proportional to $N^2P^2$, and we can use the left coaction to prove that such a Y-D module has to contain the element $\overline{\Pi_0^2}$. We can find the minimal cocalculus with $\overline{\Pi_0^2}$ and it would be Y-D submodule of the one with the second Casimir.
	\begin{eqnarray}
		\mathcal{L}(\overline{\Pi_0^2})&=&\mathrm{Span}\{\overline{\Pi_0^2},\,\overline{\Pi_0P_i},\,\overline{P_iP_j},\,\overline{C_\kappa P_i},\,\overline{C_\kappa G_\kappa}\}.
	\end{eqnarray}
	The $\dim\mathcal{L}(\overline{\Pi_0^2})=14$, so the minimal cocalculus that contain the second Casimir will have even higher dimensionality.
\end{rrm}
 We underline that $\kappa$-deformations \cite{kappa} are one of the most studied in the current literature  and have rich applications in physical models built upon deformed symmetries, see e.g. \cite{QGphenomenology,Lukierski2017,Gulia,Fermi}.
FODCs on the matrix $\kappa$-Poincar\'e Hopf algebra have been previously studied in \cite{Kosinski,Przanowski}, and in more general context in \cite{Aschieri}.

\subsection*{Added in proof}

Recently, two interesting and related articles appeared on the arXiv:   
 \cite{Julius26} studies codifferential calculi (including higher order) on quantum homogeneous spaces, \cite{Rita26} is focused on (covariant) vector fields (Cartan pairs) calculi without any relations to FOCCs.

 \subsection*{Acknowledgements}
This article is based upon work from COST Action 
CaLISTA CA21109 supported by COST (European Cooperation in Science and Technology) www.cost.eu.\\
Both authors are supported by the project UMO-2022/45/B/ST2/01067 from the Polish National Science Center (NCN). AB appreciate the discussion with R. O. Buachalla  and J. Benner during the Corfu Cost Action CaLISTA General Meeting 2025   
 \textit{ Cartan, Generalised and Noncommutative Geometries, Lie Theory and Integrable Systems Meet Vision and Physical Models} (14-22 Sept. 2025).\\
\textbf{Data Availability } No datasets were analyzed or generated as part of this research. \\
\noindent\textbf{Conflict of interest } The authors do not have  any  conflicts of  interest to declare that are relevant to the content of this work. 

 \appendix
\numberwithin{equation}{section} 
\setcounter{section}{0} 

\section{Canonical dualities: coalgebra-algebra and comodule-module correspondence }
\label{appendix:algebra}

For more details, see e.g.,  \cite[Ch. 2]{Sweedler69}, or \cite{BrzezinskiWisbauer,Radford} for a more general setting.  Let $(\mathcal{C},\Delta,\varepsilon)$ be a coalgebra and $M$ right or left comodule. The fundamental theorem on coalgebras tells us that every element $m\in M$ generates some finite-dimensional subcoalgebra. The dual vector space  $(\mathcal{C}^*,\Delta^*,\varepsilon)$, with a multiplication map being the transpose map $\Delta^*$ restricted to the subspace
$\mathcal{C}^*\otimes\mathcal{C}^*$
 \footnote{In fact, $\Delta^*: (\mathcal{C}\otimes \mathcal{C})^*\rightarrow \mathcal{C}^* $, but one can define the multiplication as its restriction to the dense subspace $\mathcal{C}^*\otimes\mathcal{C}^*\subset  (\mathcal{C}\otimes \mathcal{C})^*$.  Of course, this is not needed in a finite dimensional case,  when $\mathcal{C}^*\otimes\mathcal{C}^*\,=\,  (\mathcal{C}\otimes \mathcal{C})^*$. For similar reasons dual of multiplication map is not a comultiplication, except a finite-dimensional case.}, bears a natural unital algebra structure with the so-called convolution product $(\alpha\star\beta) (c)=\alpha(c_{(1)})\beta(c_{(2)})$, and with epsilon being a unit element $\alpha\star\varepsilon=\varepsilon\star\alpha=\alpha$. $\mathcal{C}^*$ is commutative if and only if $\mathcal{C}$ is cocommutive.
 Every $\mathcal{C}$-left (resp. right) comodule $M\in {}^\mathcal{C} \mathfrak{M}$  (resp. $M\in\mathfrak{M}^\mathcal{C}$ ) becomes automatically right (resp. left) module over the algebra $\mathcal{C}^*$, i.e. $M\in \mathfrak{M}_{\mathcal{C}^*}$ (resp. $M\in {}_{\mathcal{C}^*}\mathfrak{M}$),  by
  $m\star\alpha\doteq m_{<0>}\alpha(m_{(-1)})$ (resp. $\alpha\star m\doteq m_{<0>}\alpha(m_{(1)})$). 
  Even more, the subspace $M_1\subset M$ is a subcomodule  if and only if it is a right $\mathcal{C}^*$-submodule. Similarly, all (sub)bicomodules  become equivalent to (sub)bimodules over $\mathcal{C}^*$ by
  \begin{equation}\label{a1}
  	\alpha\star m\star\beta\doteq(\beta\otimes\mathrm{id}\otimes\alpha)\circ {}_L \Delta_{R}(m)=m_{<0>}\alpha(m_{(1)})\beta(m_{(-1)})\,.
  \end{equation}
In particular, the coalgebra itself $ \mathcal{C}\in 
{}_{\mathcal{C}^*} \mathfrak{M}_{\mathcal{C}^*}$  
by $\alpha\star c\star\beta= \beta(c_{(1)})\,c_{(2)}\,\alpha(c_{(3)})$, and 
$$\Delta(\alpha\star c\star\beta)= ( c_{(1)}\star\beta )\,\otimes\, (\alpha\star c_{(2)})\,.$$
However, it should be noted that this correspondence does not work in the opposite direction: not all $\mathcal{C}^*$-modules correspond to $\mathcal{C}$-comodules, in general. If this is the case, they are called rational modules.

There is a canonical pairing $(\mathcal{C}, \mathcal{C}^*, <\cdot,\cdot>)$.
\begin{equation}\label{a2}
	\left<c, \alpha\right>\doteq \alpha(c)
\end{equation}
given by the evaluation map with the properties
\begin{equation}\label{a3}
	\left<c,\alpha\star\beta\right>  =	\left<\Delta(c),\alpha\otimes\beta\right> = \left<\beta\star c,\alpha\right>=\left<c\star\alpha,\beta\right>\,.
\end{equation}
For any coalgebra morphism $\phi: \mathcal{C}\rightarrow\mathcal{D}$ the transpose map $\phi^*:\mathcal{D}^*\rightarrow\mathcal{C}^*$ provides an algebra map  
\begin{equation}\label{a4}
\phi^*(\omega)(c)\doteq\omega(\phi(c))\,,\quad \phi^*(\omega_1\star\omega_2) = \phi^*(\omega_1)\star\phi^*(\omega_2)\,,
\end{equation}
for $\omega,\omega_1,\omega_2\in \mathcal{D}^*$, $c\in\mathcal{C}$. In particular, $\phi\in \mathrm{Aut}(\mathcal{C})$ implies $\phi^*\in \mathrm{Aut}(\mathcal{C}^*)$ and $(\phi^*)^{-1}=(\phi^{-1})^*$.

Further, for any $\mathcal{C}$-left comodule $M\in {}^\mathcal{C} \mathfrak{M}$ (resp. right $M\in \mathfrak{M}^\mathcal{C}$ or bicomodule $M\in
{}^\mathcal{C} \mathfrak{M}^\mathcal{C}$) its dual vector space bears a canonical left $\mathcal{C}^*$-module structure, i.e.  $M^*\in {}_{\mathcal{C}^*} \mathfrak{M}$ (resp. right $M^*\in \mathfrak{M}_{\mathcal{C}^*}$ or bicomodule $M^*\in
{}_{\mathcal{C}^*} \mathfrak{M}_{\mathcal{C}^*}$), where $(\alpha\star\Gamma)(m)=\alpha(m_{(-1)})\Gamma(m_{<0>})=\Gamma(m\star\alpha)$, for any $m\in M, \alpha\in \mathcal{C}^*, \Gamma\in M^*$, etc..

As already mentioned in the footnote, a revers duality: algebra-coalgebra and module-comodule, except the finite-dimensional case, requires some special assumptions, e.g. topological ones, see e.g. \cite{Klimyk}. In fact, for a given multiplication map $\mu: \mathcal{A}\otimes \mathcal{A}\rightarrow \mathcal{A}$ one can always choose a maximal subspace $\mathcal{A}^\circ\subset\mathcal{A}^*$ (a.k.a restricted or finitary dual) such that $\mu^*=\Delta^\circ:\mathcal{A}^\circ\rightarrow  \mathcal{A}^\circ\otimes \mathcal{A}^\circ$ is a coassociative comultiplication. The canonical pairing $<\mathcal{A} | \mathcal{A}^\circ>= \mathcal{A}^\circ(\mathcal{A})$ is not always nondegenerate. In practice, it is convenient to consider nondegenerate pairing $(\mathcal{C}, \mathcal{A}, <\cdot,\cdot>)$ between an algebra $\mathcal{A}$ and the coalgebra $\mathcal{C}$.   
This  is equivalent to the existence of an injective unital algebras morphism 
 \begin{equation}\label{a5}
 	w: \mathcal{A}\rightarrow \mathcal{C}^*
 \end{equation}
 such that $\left<c,x\right>= w(x)(c)$. In addition, for each $x\in\mathcal{A}$ there is $c\in \mathcal{C}$ such that $\left<c,x\right>\neq 0$. One may also consider an associated dual pairing $(\Upsilon,M, <|\cdot,\cdot |> )$, where  $\Upsilon\in {}^\mathcal{C} \mathfrak{M}$ and $M\in {}_{\mathcal{A}} \mathfrak{M}$. Equivalently, one may assume that there exists an injective unital modules morphism
 \begin{equation}\label{a6}
 	\hat{w}: M\rightarrow \Upsilon^*
 \end{equation}
 such that
 \begin{equation}\label{a7}
  <|\upsilon,m |> =\hat{w}(m)(\upsilon)\,,\quad  <|\upsilon,\alpha\star m |> = 
   <\upsilon_{(-1)},\alpha><| \upsilon_{<0>},m |>\,  
 \end{equation}
as well as the appropriate nondegeneracy condition has to be satisfied. 
It is clear that for any finite family of nondegenerate dual pairing of vector spaces $(V_i,W_i, <,>_i)$, $i\in I=\{1,\ldots n\}$ their corresponding tensor products
$(\otimes_{i\in I}V_i, \otimes_{i\in I}W_i, <,>)$ form a nondegenerate dual pair by
\begin{equation}\label{a8   }
	<v_1\otimes\ldots\otimes v_n,w_1\otimes\ldots\otimes w_n>\doteq <v_1,w_1>_1\ldots <v_n,w_n>_n\,.
\end{equation}

 \section{Hopf modules and (bicovariant) bimodules }

Following \cite{Woronowicz,Klimyk,Schauenburg} we recall basic facts about Hopf bimodules (known also as bicovariant bimodules after Woronowicz \cite{Woronowicz}) . 
Let $H=(H, \Delta, \mu=\cdot, 1=1_H, \varepsilon_H=\varepsilon, S=S_H)$ be a Hopf algebra (we do not assume here that the antipod $S$ is bijective).
\begin{defi} \footnote{We can also define a left covariant bicomodule as an object $(M,\triangleright,\Delta_L,\Delta_R)\in{}^H_H\mathfrak{M}^H$.}
	Bicovariant bimodule over $H$   is an object $(M,\triangleright,\triangleleft,\Delta_L,\Delta_R)\in{}^H_H\mathfrak{M}^H_H$ such that:
	
	\begin{itemize}
		 
		\item $(M,\Delta_L,\Delta_R)\in{}^H\mathfrak{M}^H$ is a bicomodule,  $(M,\triangleright,\triangleleft)\in{}_H\mathfrak{M}_H$ is a bimodule,
		and the following compatibility conditions are satisfied
		\begin{equation}\label{b1}
				\Delta_L(a\triangleright m\triangleleft b)= (a_{(1)}\cdot m_{(-1)}\cdot b_{(1)})\otimes(a_{(2)}\triangleright m_{<0>}\triangleleft b_{(2)})\,,
			\end{equation}
	and ($1\triangleright m=m\triangleleft 1=m$)
	\begin{equation}\label{b2}
			\Delta_R(a\triangleright m\triangleleft b)= (a_{(1)}\triangleright m_{<0>}\triangleleft b_{(1)})\otimes (a_{(2)}\cdot m_{(1)}\cdot b_{(2)})\,.
		\end{equation}

	\end{itemize}
\end{defi}
The structure theorem for such objects indicates many important properties:
 \begin{itemize}
	\item $(M,\triangleright,\Delta_L)\in{}^H_H\mathfrak{M}$ is a left Hopf module. Therefore, it is left free and left cofree, i.e. 
	$M\cong H\otimes P_L(M)$, where the vector space $P_L(M)\doteq \mathrm{Im}P_L$, $P_L^2=P_L$ is a projector operator defined by $P_L(m)=S(m_{(-1)})\triangleright m_{<o>}$. Its image coincides with the vector space of left coinvariants $\mathrm{Im}P_L={}^{\mathrm{coH}}M\doteq\{ v\in M|\Delta_L(v)=1\otimes v\}\subset M$. Thus $P_L(a\triangleright m)=\varepsilon(a)P_L(m)$. The maps $\Phi_L: H\otimes \mathrm{Im}P_L \rightarrow M\,, \Phi_L (a\otimes v)=a\triangleright v$ and its inverse $\Phi_L^{-1}(m)= m_{(-1)}\otimes P_L(m_{<0>})= m_{(-2)}\otimes S( m_{(-1)})\triangleright  m_{<0>}$ are left Hopf modules  isomorphisms. 
	
	\item Since $(M\simeq H\otimes \mathrm{Im}P_L, \triangleleft, \Delta_R) \in\mathfrak{M}^H_H$ is also right Hopf module, this causes that the vector space $\mathrm{Im}P_L$ possesses induced a right comodule and a right module structures, i.e. $(\mathrm{Im}P_L, \Xi_R, \blacktriangleleft)$:
	\begin{equation}\label{b3}
		v	\blacktriangleleft a\doteq P_L(v\triangleleft a)=S(a_{(1)})\triangleright v\triangleleft a_{(2)}=
		\mathrm{Ad}^{\triangleright\triangleleft}_{R\,a} (v)\,,
	\end{equation}
	and
	\begin{equation}\label{b4}
		\Xi_R(v)\equiv v_{<0>}\otimes v_{(1)}\doteq (\varepsilon\otimes \mathrm{id}\otimes  \mathrm{id})\circ \Delta_R (1\otimes v)\,,
	\end{equation}
	 which have to satisfy the so-called Yetter-Drinfeld (or crossed module) compatibility condition \cite{Yetter}:
	 \begin{equation}\label{b5}
	  v_{<0>}\blacktriangleleft a_{(1)}\otimes v_{(1)} a_{(2)}=(v\blacktriangleleft a_{(2)})_{<0>}\otimes  a_{(1)}(v\blacktriangleleft a_{(2)})_{(1)} \,.
	 \end{equation}
	 We say that $(\mathrm{Im}P_L,\Xi_R, \blacktriangleleft )\in \mathfrak{YD}^H_H$  is a right-right Yetter-Drinfeld (or crossed) module. \footnote{This indicates that Yetter-Drinfeld modules could be defined in more general settings over bialgebras.}. However, in the case of Hopf algebras the last condition can be equivalently replaced by the more convenient one
	 \begin{equation}\label{b6}
	 	\Xi_R(v\blacktriangleleft a)= v_{<0>}\blacktriangleleft a_{(2)}\otimes S(a_{(1)})v_{(1)}a_{(3)}\,. 
	 \end{equation}
	\item Conversely, assuming above condition, i.e., $(\mathcal{R},\Xi_R, \blacktriangleleft )\in \mathfrak{YD}^H_H$ then $M=H\otimes \mathcal{R}$ becomes automatically Hopf bimodule if we define
	\begin{equation}\label{b7}
		\Delta_L(a\otimes v)=a_{(1)}\otimes a_{(2)}\otimes v\,, \quad \Delta_R(a\otimes v)=a_{(1)}\otimes v_{<0>}\otimes a_{(2)}v_{(1)}\,,
	\end{equation}
	and
	\begin{equation}\label{b8}
		a\triangleright (b\otimes v)\triangleleft c= abc_{(1)}\otimes (v\blacktriangleleft c_{(2)})\,.\end{equation}
We should noticed that the left coaction and  the left action are now free.
Of course, $\mathcal{R}$ is unique up to right-right Yetter-Drinfeld modules isomorphism. In particular $\mathrm{Im}P_L\simeq \mathcal{R}$ in $\mathfrak{YD}^H_H$. Therefore $\Phi_L: H\otimes \mathrm{Im}P_L \rightarrow H\otimes \mathcal{R}$ becomes a Hopf bimodules isomorphism.
	
	\item Similarly, one can introduce the map $\Phi_R: \mathrm{Im}P_R\otimes H\rightarrow M\simeq \mathcal{L}\otimes H\,, \Phi_R (u\otimes a)=u\triangleleft a$ and its inverse $\Phi_R^{-1}(m)=  P_R(m_{<0>})\otimes m_{(1)}= m_{<0>}\triangleleft  S( m_{(1)})\otimes m_{(2)}$, are left module morphisms, where  $\mathcal{L}\simeq\mathrm{Im}P_R= M^\mathrm{coH}\doteq\{ u\in M|\Delta_R(u)=u\otimes 1\}$ can be equipped in induced left-left Yetter-Drinfeld structure, i.e., $(\mathcal{L}, \Xi_L, \blacktriangleright)\in {}^H_H\mathfrak{YD}$, satisfying  \footnote{ Or on bialgebra level: $\ 
		 a_{(1)}u_{(-1)} \otimes 	a_{(2)} \blacktriangleright   u_{<0>}=
		  (a_{(1)}\blacktriangleright  u)_{(-1)}a_{(2)}\otimes (a_{(1)}\blacktriangleright  u)_{<0>} \,.
	$}
	\begin{equation}\label{b9}
		\Xi_L(a\blacktriangleright u)= a_{(1)}u_{(-1)}S(a_{(3)})\otimes a_{(2)}\blacktriangleright u_{<o>}
	\end{equation}
	where
	\begin{equation}\label{b10}
		a\blacktriangleright u\doteq P_R(a\triangleright u)=a_{(1)}\triangleright u\triangleleft S(a_{(2)})=
		\mathrm{Ad}^{\triangleright\triangleleft}_{L\,a} (u)\,.
	\end{equation}
	Thus $M\simeq \mathcal{L}\otimes H$ has the following Hopf bimodule structure
	\begin{equation}\label{b11}
		\Delta_L(u\otimes a)= u_{(-1)}a_{(1)}\otimes u_{<0>}\otimes a_{(2)}\,, \quad \Delta_R(u\otimes a)=u\otimes a_{(1)}\otimes  a_{(2)}\,,
	\end{equation}
	and
	\begin{equation}\label{b12}
		a\triangleright (u\otimes b)\triangleleft c=   (a_{(1)}\blacktriangleright u)\otimes  a_{(2)}bc \,.
	\end{equation}
	We should noticed that now the right coaction and  the right action are free.

\end{itemize}

\section{Woronowicz's bicovariant FODC on Hopf algebras}\label{apx_bicov_FODC}
Any Hopf algebra $H$, can be consider as a right-right (or left-left) Yetter-Drinfeld module in two nonisomorphic ways:

i) with the adjoint action and the standard coaction
\begin{equation}\label{c1}
	b	\blacktriangleleft a\doteq S(a_{(1)}) b a_{(2)}\,,\quad \Xi_R(a)\doteq\Delta(a)=a_{(1)}\otimes a_{(2)}\,;
\end{equation}
(or 	$a	\blacktriangleright b\doteq a_{(1)}) b S(a_{(2)})\,,\quad \Xi_L(a)\doteq\Delta(a)=a_{(1)}\otimes a_{(2)}$)\\
ii) with standard multiplication and a coadjoint action
\begin{equation}\label{c2}
	b\blacktriangleleft a\doteq b a\,,\quad \Xi_R(a)\doteq a_{(2)} \otimes S(a_{(1)}) a_{(3)}\,.
\end{equation}
(or $a\blacktriangleright b\doteq a b\,,\quad \Xi_L(a)\doteq a_{(1)}S( a_{(3)})\otimes  a_{(2)}$)\\
 In the first case 
$\bar H=H/\{\mathbb{K}1_H\}$ becomes a Y-D factor module by the trivial Y-D submodule $\{\mathbb{K}1_H\}$. Its Y-D submodules $\Upsilon\subset \bar H$ generate bicovariant FOCCs considered in Section 3.  In fact, the universal bicomodule $\Upsilon^U_H\simeq H\otimes \bar H\supset H\otimes\Upsilon$ is bicovariant.

In the second case, being a cornerstone of Woronowicz theory \cite{Woronowicz2},  $\ker\varepsilon_H\subset H\in \mathfrak{YD}^H_H$ is a Y-D submodule of $H$. Consequently, $\ker\varepsilon_H\otimes H\simeq \ker\mu\simeq H\otimes\ker\varepsilon\in {}^H_H\mathfrak{M}^H_H$ is a bicovariant bimodule of the universal one-forms. Therefore, any other bicovariant bimodule of one-forms can be be obtained as quotient $\ker\mu/ N$, where $N$ is a bicovariant subbimodule of $\ker\mu$. More precisely, according to Woronowicz prescription, firstly one chooses a Y-D submodule $R\subset \ker\varepsilon$ and then we set $N=r^{-1}(H\otimes R)\subset\ker\mu$. Thus $\ker\mu/ N\simeq H\otimes \ker\varepsilon/R$. The isomorphism is realized by Woronowicz map $r(a\otimes b)=a b_{(1)}\otimes b_{(2)}$.  Moreover, any bicovariant FODC on $H$ can be obtained in this way.    
 Obviously, there is also left-handed version of this provided by the Woronowicz map  $s'(a\otimes b)= a_{(1)}\otimes a_{(2)}b$; $N=s'^{-1}(L\otimes H)\subset\ker\mu$. 
 \begin{rrm}
 Consequently, the vector space $H\otimes H$ enjoys two nonisomorphic Hopf bimodule structures which turns out to be  dual each other (c.f. Section 4.2):
 
 i) with the actions $a\triangleright (b\otimes c)\triangleleft d=a_{(1)}bd_{(1)}\otimes a_{(2)}cd_{(2)}$ and the coactions\\
 $\Delta_L (a\otimes b) =a_{(1)}\otimes (a_{(2)}\otimes b)\,,\quad \Delta_R (a\otimes b)= (a\otimes   b_{(1)})\otimes  b_{(2)}\,$ (c.f. \eqref{DeltaU}, \eqref{triangle-actions} and Section 3). The universal bicovariant bimodule $\Upsilon^U_H$ becomes its factor bimodule.  
 
 ii)  with the actions $a\triangleright (b\otimes c)\triangleleft d=ab\otimes cd$ and the coactions\\
 $\Delta_L (a\otimes b) =a_{(1)}b_{(1)}\otimes (a_{(2)}\otimes b_{(2)})\,,\quad \Delta_R (a\otimes b)= (a_{(1)}\otimes   b_{(1)})\otimes a_{(2)} b_{(2)}\,$ (c.f. Woronowicz \cite{Woronowicz}). 
 In this case $\ker\mu$ is a bicovariant subbimodule of $H\otimes H$.
 
 \end{rrm}
%

%
 \section{Proof of the formulas \eqref{E1}-\eqref{E2}}
\begin{proof}
          One calculates  :
         \begin{eqnarray}
             \Delta^U_L(\upsilon^n)&=&\sum_{i=0}^{n-1}\left(1-\frac{i}{n}\right)\sum_{j=0}^{n-i}X^j\otimes[X^{n-i-j}\otimes X^i]\nonumber\\
             &=&\sum_{j=0}^{n}X^j\otimes\sum_{i=0}^{n-j}\left(1-\frac{i}{n}\right)[X^{n-i-j}\otimes X^i]-\sum_{j=0}^n\frac{j}{n}X^j\otimes[\Delta (X^{n-j})]\nonumber\\
             &=&\sum_{j=0}^{n}X^j\otimes\sum_{i=0}^{n-j}\left(1-\frac{i+j}{n}\right)[X^{n-i-j}\otimes X^i]\nonumber\\
             &=&\sum_{j=0}^{n}\left(1-\frac{j}{n}\right)X^j\otimes\sum_{i=0}^{n-j-1}\left(1-\frac{i}{n-j}\right)[X^{n-i-j}\otimes X^i]\nonumber\\
             &=&\sum_{j=0}^{n}\left(1-\frac{j}{n}\right)X^j\otimes\upsilon^{n-j}\,,
               \end{eqnarray}
             \begin{eqnarray}
             \Delta^U_R(\upsilon^n)&=&\sum_{i=0}^{n-1}\left(1-\frac{i}{n}\right)\sum_{j=0}^{i}[X^{n-i}\otimes X^{i-j}]\otimes X^{j}\nonumber\\
             &=&\sum_{j=0}^{n}\left(\sum_{i=j}^{n}\left(1-\frac{i}{n}\right)[X^{n-i}\otimes X^{i-j}]\right)\otimes X^j\nonumber\\
             &=&\sum_{j=0}^{n}\left(\sum_{i=0}^{n-j}\left(1-\frac{i+j}{n}\right)[X^{n-i-j}\otimes X^{i}]\right)\otimes X^j\nonumber\\
             &=&\sum_{j=0}^{n}\left(1-\frac{j}{n}\right)\upsilon^{n-j}\otimes X^j\,.
         \end{eqnarray}
         In particular, it shows that $\Delta^U_L(v^n) =\Delta^{U\tau}_R(v^n)$ for any $n\geq 1$.
     \end{proof}
\section{The proof of Theorem 3}
 \begin{proof}
 	It remains to prove three generalized Jacobi identities, \eqref{qJ1}, \eqref{qJ2}, \eqref{qJ3}.  For better presentation, we temporarily simplify the notation:
\begin{eqnarray}
	 \mathrm{ad}^L\equiv \mathrm{ad}\,,\quad \tau_q\equiv\tau\,,\quad [\,,]_q\equiv C\,.
\end{eqnarray}

The first deformed Jacobi identity, c.f. \eqref{qJ1}:
\begin{eqnarray}
	(C\otimes\mathrm{id})=C(\mathrm{id}\otimes C)((\mathrm{id}-\tau)\otimes\mathrm{id}).
\end{eqnarray}
We compute:
\begin{eqnarray}
	C(C\otimes\mathrm{id})(\overline{X}\otimes\overline{Y}\otimes\overline{Z})&=&C(\mathrm{ad}_{\delta(\overline{X})}\overline{Y}\otimes\overline{Z})=\mathrm{ad}_{\delta(\mathrm{ad}_{\delta(\overline{X})}\overline{Y})}\overline{Z}\nonumber\\
	&=&\mathrm{ad}_{\mathrm{ad}_XY-Y\varepsilon(X)}\overline{Z},\\
	C(\mathrm{id}\otimes C)(\overline{X}\otimes\overline{Y}\otimes\overline{Z})&=&C(\overline{X}\otimes\mathrm{ad}_{\delta(\overline{Y})}\overline{Z})=\mathrm{ad}_{\delta(\overline{X})\delta(\overline{Y})}\overline{Z}\nonumber\\
	&=&\mathrm{ad}_{XY-X\varepsilon(Y)-Y\varepsilon(X)-\varepsilon(XY)},\\
	C(\mathrm{id}\otimes C)(\mathrm{ad}_{X_{(1)}}\overline{Y}\otimes\overline{X_{(2)}}\otimes\overline{Z})&=&C(\mathrm{ad}_{X_{(1)}}\overline{Y}\otimes\mathrm{ad}_{\delta(\overline{X}_{(2)})}\overline{Z})\nonumber\\
	=\mathrm{ad}_{\delta(\mathrm{ad}_{X_{(1)}}\overline{Y})\delta(\overline{X_{(2)}})}\overline{Z}
	&=&\mathrm{ad}_{XY-\mathrm{ad}_XY-X\varepsilon(Y)-\varepsilon(XY)}\overline{Z},\\
	C(\mathrm{id}\otimes C)((\mathrm{id}-\tau)\otimes\mathrm{id})(\overline{X}\otimes\overline{Y}\otimes\overline{Z})&=&\mathrm{ad}_{\mathrm{ad}XY-Y\varepsilon(X)} Z
\end{eqnarray} 
The second deformed Jacobi identity, c.f. \eqref{qJ2}:
\begin{eqnarray}
	\tau(\mathrm{id}\otimes C)=(C\otimes\mathrm{id})(\mathrm{id}\otimes\tau)(\tau\otimes\mathrm{id})
\end{eqnarray}
One calculates:
\begin{eqnarray}
	\tau(\mathrm{id}\otimes C)(\overline{X}\otimes\overline{Y}\otimes\overline{Z})&=&\tau(\overline{X}\otimes\mathrm{ad}_{\delta(\overline{Y})}\overline{Z})=\mathrm{ad}_{X_{(1)}\delta(\overline{Y})}\overline{Z}\otimes\overline{X_{(2)}},\\
	(C\otimes\mathrm{id})(\mathrm{id}\otimes\tau)(\tau\otimes\mathrm{id})(\overline{X}\otimes\overline{Y}\otimes\overline{Z})&=&(C\otimes\mathrm{id})(\mathrm{id}\otimes\tau)(\mathrm{ad}_{X_{(1)}}\overline{Y}\otimes\overline{X_{(2)}}\otimes\overline{Z})\nonumber\\
	=(C\otimes\mathrm{id})(\mathrm{ad}_{X_{(1)}}\overline{Y}\otimes\mathrm{ad}_{X_{(2)}}\overline{Z}\otimes\overline{X}_{(3)})
	&=&\mathrm{ad}_{\delta(\mathrm{ad}_{X_{(1)}}\overline{Y})X_{(2)}}\overline{Z}\otimes\overline{X_{(3)}}\nonumber\\
	&=&\mathrm{ad}_{X_{(1)}\delta(\overline{Y})}\overline{Z}\otimes\overline{X_{(2)}}.
\end{eqnarray}

The last deformed Jacobi identity, c.f. \eqref{qJ3}:
\begin{eqnarray}
	\tau(C\otimes\mathrm{id})-(\mathrm{id}\otimes C)(\tau\otimes\mathrm{id})(\mathrm{id}\otimes\tau)=(C\otimes\mathrm{id})(\mathrm{id}\otimes\tau)((\mathrm{id}-\tau^2)\otimes\mathrm{id})
\end{eqnarray}
We compute:
\begin{eqnarray}
	\tau(C\otimes\mathrm{id})(\overline{X}\otimes\overline{Y}\otimes\overline{Z})&=&\tau(\mathrm{ad}_{\delta(\overline{X})}\overline{Y}\otimes\overline{Z})\nonumber\\
	&=&\mathrm{ad}_{X_{(1)}Y_{(1)}S(X_{(3)})}\overline{Z}\otimes\mathrm{ad}_{X_{(2)}}\overline{Y_{(2)}}-\mathrm{ad}_{Y_{(1)}}\overline{Z}\otimes\overline{Y_{(2)}}\varepsilon(X)\label{apx_jac3_1}\\
	(\mathrm{id}\otimes C)(\tau\otimes\mathrm{id})(\mathrm{id}\otimes\tau)(\overline{X}\otimes\overline{Y}\otimes\overline{Z})&=&(\mathrm{id}\otimes C)(\tau\otimes\mathrm{id})(\overline{X}\otimes\mathrm{ad}_{Y_{(1)}}\overline{Z}\otimes\overline{Y_{(2)}})\nonumber\\
	=(\mathrm{id}\otimes C)(\mathrm{ad}_{X_{(1)}Y_{(1)}}\overline{Z}\otimes\overline{X_{(2)}}\otimes\overline{Y_{(2)}})
	&=&\mathrm{ad}_{X_{(1)}Y_{(1)}}\overline{Z}\otimes\mathrm{ad}_{\delta(\overline{X_{(2)}})}\overline{Y_{(2)}}\nonumber\\
	&=&\mathrm{ad}_{X_{(1)}Y_{(1)}}\overline{Z}\otimes\mathrm{ad}_{X_{(2)}}\overline{Y_{(2)}}-\mathrm{ad}_{XY_{(1)}}\overline{Z}\otimes\overline{Y_{(2)}},\label{apx_jac3_2}\\
	(C\otimes\mathrm{id})(\mathrm{id}\otimes\tau)(\overline{X}\otimes\overline{Y}\otimes\overline{Z})&=&(C\otimes\mathrm{id})(\overline{X}\otimes\mathrm{ad}_{Y_{(1)}}\overline{Z}\otimes\overline{Y_{(2)}})\nonumber\\
	=\mathrm{ad}_{\delta(\overline{X})Y_{(1)}}\overline{Z}\otimes\overline{Y_{(2)}}
	&=&\mathrm{ad}_{XY_{(1)}}\overline{Z}\otimes\overline{Y_{(2)}}-\mathrm{ad}_{Y_{(1)}}\overline{Z}\otimes\overline{Y_{(2)}}\varepsilon(X),\label{apx_jac3_3}\\
	(C\otimes\mathrm{id})(\mathrm{id}\otimes\tau)(\tau^2\otimes\mathrm{id})(\overline{X}\otimes\overline{Y}\otimes\overline{Z})&=&(C\otimes\mathrm{id})(\mathrm{id}\otimes\tau)(\tau\otimes\mathrm{id})(\mathrm{ad}_{X_{(1)}}\overline{Y}\otimes\overline{X_{(2)}}\otimes\overline{Z})\nonumber\\
	&=&(C\otimes\mathrm{id})(\mathrm{id}\otimes\tau)(\mathrm{ad}_{X_{(1)}Y_{(1)}}\overline{S^2(X_{(3)})}\otimes\mathrm{ad}_{X_(2)}\overline{Y}\otimes\overline{Z}\nonumber\\
	&=&(C\otimes\mathrm{id})(\mathrm{ad}_{X_{(1)}Y_{(1)}}\overline{S^2}(X_{(5)})\otimes\mathrm{ad}_{X_{(2)}Y_{(2)}S(X_{(4)})}\overline{Z}\otimes\mathrm{ad}_{X_{(3)}}\overline{Y_{(3)}}\nonumber\\
	&=&\mathrm{ad}_{\delta(\mathrm{ad}_{X_{(1)}Y_{(1)}}\overline{S^2}(X_{(5)}))X_{(2)}Y_{(2)}S(X_{(4)})}\overline{Z}\otimes\mathrm{ad}_{X_{(3)}}\overline{Y_{(3)}}\nonumber\\
	&=&\mathrm{ad}_{X_{(1)}Y_{(1)}}\overline{Z}\otimes\mathrm{ad}_{X_{(2)}}\overline{Y_{(2)}}-\mathrm{ad}_{X_{(1)}Y_{(1)}S(X_{(3)})}\overline{Z}\otimes\mathrm{ad}_{X_{(2)}}\overline{Y_{(2)}},\nonumber\\\label{apx_jac3_4}
\end{eqnarray}
This calculations  allow us to conclude  that  $\eqref{apx_jac3_1}-\eqref{apx_jac3_2}=\eqref{apx_jac3_3}-\eqref{apx_jac3_4}$, i.e. the identity  \eqref{qJ3} is shown.
\end{proof}


\section{Quantum Lie algebra structure for four-dimensional FOCC over $U_q(\mathfrak{sl}(2))$}
\begin{eqnarray}
    \tau_q(\upsilon_{00}\otimes\upsilon_{00})=\upsilon_{00}\otimes\upsilon_{00} &\quad& \tau_q(\upsilon_{00}\otimes\upsilon_{10})=q^2\upsilon_{10}\otimes\upsilon_{00}\nonumber\\
    \tau_q(\upsilon_{00}\otimes\upsilon_{11})=\upsilon_{11}\otimes\upsilon_{00} && \tau_q(\upsilon_{00}\otimes\upsilon_{01})=q^{-2}\upsilon_{01}\otimes\upsilon_{00}
\end{eqnarray}
\begin{eqnarray}
    \tau_q(\upsilon_{10}\otimes\upsilon_{00})=\upsilon_{00}\otimes\upsilon_{10}+(1-q^2)\upsilon_{10}\otimes\upsilon_{00} &\quad& \tau_q(\upsilon_{10}\otimes\upsilon_{10})=\upsilon_{10}\otimes\upsilon_{10}\nonumber\\
    \tau_q(\upsilon_{10}\otimes\upsilon_{11})=\upsilon_{11}\otimes\upsilon_{10}-(q^3+q)\upsilon_{10}\otimes\upsilon_{00}&&\tau_q(\upsilon_{10}\otimes\upsilon_{01})=\upsilon_{01}\otimes\upsilon_{10}+\upsilon_{11}\otimes\upsilon_{00}\nonumber\\
\end{eqnarray}
\begin{eqnarray}
    \tau_q(\upsilon_{01}\otimes\upsilon_{00})=\upsilon_{00}\otimes\upsilon_{01}+(1-q^{-2})\upsilon_{01}\otimes\upsilon_{00} &\quad& \tau_q(\upsilon_{01}\otimes\upsilon_{10})=\upsilon_{10}\otimes\upsilon_{01}-\upsilon_{11}\otimes\upsilon_{00}\nonumber\\
    \tau_q(\upsilon_{01}\otimes\upsilon_{11})=\upsilon_{11}\otimes\upsilon_{01}+(q+q^{-1})\upsilon_{01}\otimes\upsilon_{00}&&\tau_q(\upsilon_{01}\otimes\upsilon_{01})=\upsilon_{01}\otimes\upsilon_{01}
\end{eqnarray}
\begin{eqnarray}
    \tau_q(\upsilon_{11}\otimes\upsilon_{00})&=&\upsilon_{00}\otimes\upsilon_{11}+(2-q^2-q^{-2})(\upsilon_{11}\otimes\upsilon_{00}+\upsilon_{01}\otimes\upsilon_{10}-\upsilon_{10}\otimes\upsilon_{01})\nonumber\\
    \tau_q(\upsilon_{11}\otimes\upsilon_{11})&=&\upsilon_{11}\otimes\upsilon_{11}+(q^{-1}-q^3)(\upsilon_{11}\otimes\upsilon_{00}+\upsilon_{01}\otimes\upsilon_{10}-\upsilon_{10}\otimes\upsilon_{01})\nonumber\\
    \tau_q(\upsilon_{11}\otimes\upsilon_{10})&=&q^{-2}\upsilon_{10}\otimes\upsilon_{11}+(1-q^{-2})\upsilon_{11}\otimes\upsilon_{10}+(q+q^{-1})\upsilon_{10}\otimes\upsilon_{00}\nonumber\\
    \tau_q(\upsilon_{11}\otimes\upsilon_{01})&=&q^2\upsilon_{01}\otimes\upsilon_{11}+(1-q^2)\upsilon_{11}\otimes\upsilon_{01}-(q^3+q)\upsilon_{01}\otimes\upsilon_{00}
\end{eqnarray}
\begin{eqnarray}
    &{}[\upsilon_{00},\upsilon_{00}]_q=0\qquad[\upsilon_{00},\upsilon_{11}]_q=0\qquad[\upsilon_{00},\upsilon_{10}]_q=(q^2-1)\upsilon_{10}\qquad[\upsilon_{00},\upsilon_{01}]_q=(q^{-2}-1)\upsilon_{01}\nonumber\\
    &{}[\upsilon_{10},\upsilon_{00}]_q=(1-q^2)\upsilon_{10}\qquad[\upsilon_{10},\upsilon_{11}]_q=-(q^3-q)\upsilon_{10}\qquad[\upsilon_{10},\upsilon_{10}]_q=0\qquad[\upsilon_{10},\upsilon_{01}]_q=\upsilon_{11}\nonumber\\
    &{}[\upsilon_{01},\upsilon_{00}]_q=(1-q^{-2})\upsilon_{01}\qquad[\upsilon_{01},\upsilon_{11}]_q=(q-q^{-1})\upsilon_{10}\qquad[\upsilon_{01},\upsilon_{10}]_q=-\upsilon_{11}\qquad[\upsilon_{01},\upsilon_{01}]_q=0\nonumber\\
    &{}[\upsilon_{11},\upsilon_{00}]_q=(2-q^2-q^{-2})\upsilon_{11}\qquad[\upsilon_{11},\upsilon_{11}]_q=(q-q^3)\upsilon_{11}\nonumber\\
    &{}[\upsilon_{11},\upsilon_{10}]_q=(q+q^{-1})\upsilon_{10}\qquad[\upsilon_{11},\upsilon_{01}]_q=-(q+q^3)\upsilon_{11}
\end{eqnarray}
In the classical limit $q\mapsto 1$, $\upsilon_{00}=\overline{K}\mapsto 0$, $\upsilon_{10} \mapsto E$, $\upsilon_{01} \mapsto F$ and $\upsilon_{11} \mapsto [E,F]=H$ one can recover $\mathfrak{sl}(2)$ relations as well as the flip braiding $\tau_1(\upsilon_1\otimes\upsilon_2)= \upsilon_2\otimes\upsilon_1$.


\end{document}